\documentclass[11pt]{article}
\usepackage{amssymb, amsthm, amsmath, amscd}
\newtheorem{thm}{Theorem}[section]
\newtheorem{lem}[thm]{Lemma}

\newtheorem{cor}[thm]{Corollary}
\newtheorem{NN}[thm]{}
\theoremstyle{definition}\newtheorem{df}[thm]{Definition}
\theoremstyle{definition}\newtheorem{rem}[thm]{Remark}
\theoremstyle{definition}

\newcommand{\N}{\mathbb{N}}
\newcommand{\Z}{\mathbb{Z}}

\newcommand{\R}{\mathbb{R}}
\newcommand{\C}{\mathbb{C}}
\newcommand{\T}{\mathbb{T}}

\newcommand{\morp}{contractive completely positive linear map}

\newcommand{\hm}{homomorphism}
\newcommand{\dt}{\delta}
\newcommand{\ep}{\epsilon}
\newcommand{\andeqn}{\,\,\,{\rm and}\,\,\,}
\newcommand{\rforal}{\,\,\,{\rm for\,\,\,all}\,\,\,}
\newcommand{\CA}{$C^*$-algebra}
\newcommand{\SCA}{$C^*$-subalgebra}

\newcommand{\bt}{{\beta}}

\newcommand{\beq}{\begin{eqnarray}}
\newcommand{\eneq}{\end{eqnarray}}
\newcommand{\tforal}{\,\,\,\text{for\,\,\,all}\,\,\,}
\newcommand{\tand}{\,\,\,\text{and}\,\,\,}

\title{Homotopy of unitaries
in simple \CA s with tracial rank one}
\author{Huaxin Lin\\
 }
\date{}

\begin{document}

\maketitle

\begin{abstract}
Let $\ep>0$ be a positive number. Is there  a number $\dt>0$ satisfying the following? Given any
pair of unitaries $u$ and $v$ in a unital   simple \CA\,
$A$ with $[v]=0$ in $K_1(A)$ for which
$$
\|uv-vu\|<\dt,
$$
there is a continuous path of unitaries $\{v(t): t\in [0,1]\}\subset
A$ such that
$$
v(0)=v,\,\,\, v(1)=1\andeqn \|uv(t)-v(t)u\|<\ep\tforal t\in [0,1].
$$
An answer is given to this question when $A$ is assumed to be a
unital  simple \CA\, with tracial rank no more than one.
Let $C$ be a unital separable amenable simple \CA\, with tracial
rank no more than one which also satisfies the UCT. Suppose that
$\phi: C\to A$ is a unital monomorphism and suppose that $v\in A$ is a unitary
with $[v]=0$ in $K_1(A)$  such that $v$ almost commutes with $\phi.$
It is shown that 
 there is a
continuous path of unitaries $\{v(t): t\in [0,1]\}$ in $A$ with
$v(0)=v$ and $v(1)=1$ such that the entire path $v(t)$ almost
commutes with $\phi,$ provided that an induced  Bott map vanishes.  Other versions of the so-called Basic Homotopy
Lemma are also presented.

\end{abstract}


\section{Introduction}
Fix a positive number $\ep>0.$ Can one find a positive number $\dt$
such that, for any pair of unitary matrices $u$ and $v$ with
$\|uv-vu\|<\dt,$ there exists a continuous path of unitary matrices
$\{v(t): t\in [0,1]\}$ for which  $v(0)=v,$ $v(1)=1$ and
$\|uv(t)-v(t)u\|<\ep$ for all $t\in [0,1]?$ The answer is negative
in general. A Bott element associated with the pair of unitary
matrices  may appear.
The hidden topological obstruction can be detected in a limit
process.  This was first found by Dan Voiculescu (\cite{V1}). On the
other hand, it has been proved that there is such a path of unitary
matrices  if an additional condition, ${\rm bott}_1(u,v)=0,$ is
provided (see, for example,  \cite{BEEK} and also 3.12 of
\cite{Lnhomp}).

It was recognized by Bratteli,  Elliott,  Evans and A. Kishimoto
(\cite{BEEK})
 that the presence of such continuous
path of unitaries in general simple \CA s played an important role
in the study of classification of simple \CA s and perhaps plays
important roles in some other areas.  They proved what they called
the Basic Homotopy Lemma: For any $\ep>0,$ there exists $\dt>0$
satisfying the following: For any pair of unitaries $u$ and $v$ in
$A$ with $sp(u)$ $\dt$-dense in $\T$ and $[v]=0$ in $K_1(A)$ for
which
$$
\|uv-vu\|<\dt\andeqn {\rm bott}_1(u,v)=0,
$$
there exists a continuous path of unitaries $\{v(t):t\in
[0,1]\}\subset A$ such that
$$
v(0)=v,\,\,\, v(1)=1_A\andeqn \|v(t)u-uv(t)\|<\ep
$$
for all $t\in [0,1],$ where $A$ is a unital purely infinite simple
\CA\, or a unital simple \CA\, with real rank zero and stable rank
one.  Define $\phi: C(\T)\to A$ by $\phi(f)=f(u)$ for all $f\in
C(\T).$ Instead of considering a pair of unitaries, one may consider
a unital \hm\, from $C(\T)$ into $A$ and a unitary $v\in A$ for
which $v$ almost commutes with $\phi.$

In the study of asymptotic unitary equivalence of \hm s from an
AH-algebra to a unital simple \CA, as well as the study of homotopy
theory in simple \CA s, one considers the following problem: Suppose
that $X$ is a compact metric space and $\phi$ is a unital \hm\, from
$C(X)$ into a unital simple \CA\, $A.$ Suppose that there is a
unitary $u\in A$  with $[u]=0$ in $K_1(A)$ and
 $u$ almost commutes with $\phi.$ When  can one
find a continuous path of unitaries $\{u(t): t\in [0,1]\}\subset A$
with $u(0)=u$ and $u(1)=1$  such that $u(t)$ almost commutes with
$\phi$ for all $t\in [0,1]?$

Let $C$ be a unital AH-algebra and let $A$ be a unital  simple \CA.
Suppose that $\phi, \psi: C\to A$ are two unital monomorphisms. Let
us consider the question when  $\phi$ and $\psi$ are asymptotically
unitarily equivalent, i.e., when there is a continuous path of
unitaries $\{w(t): t\in [0, \infty)\}\subset A$ such that
$$
\lim_{t\to\infty}w(t)^*\phi(c)w(t)=\psi(c)\tforal c\in C.
$$
When $A$ is assumed to have tracial rank zero, it was proved in
\cite{Lnasym} that they are asymptotically unitarily equivalent if
and only if $[\phi]=[\psi]$ in $KK(C,A),$ $\tau\circ\phi=\tau\circ
\psi$ for all tracial states $\tau$ of $A$ and a rotation map
associated with $\phi$ and $\psi$ is zero.  Apart from some direct
applications, this result plays crucial roles in the study of the
problem to embed crossed products into unital simple AF-algebras and
in the classification of simple amenable \CA s which do not have the 
finite tracial rank (see \cite{Wz}, \cite{LnZ2} and \cite{Lnapn}).
One of the key machinery in the study of the above mentioned
asymptotic unitary equivalence is the so-called Basic Homotopy Lemma
concerning a unital monomorphism $\phi$ and a unitary $u$ which
almost commutes with $\phi.$

In this paper, we study the case that $A$ is no longer assumed to
have real rank zero, or tracial rank zero. The result of W. Winter
in \cite{Wz} provides the possible classification of simple finite
\CA s far beyond the cases of finite tracial rank. However, it
requires to understand much more about asymptotic unitary
equivalence in those unital separable simple \CA s which have been
classified. An immediate problem is to give a classification of
monomorphisms (up to asymptotic unitary equivalence) from a unital
separable simple AH-algebra into a unital separable simple \CA\,
with tracial rank one. For that goal, it is paramount to study the
Basic Homotopy Lemmas in a simple separable \CA s with tracial rank
one. This is the main purpose of this paper.

A number of problems occur  when one replaces \CA s of  tracial rank
zero by those of tracial rank one. First, one has to deal with \morp
s from $C(X)$ into a  unital \CA\, $C$ with the form $C([0,1], M_n)$
which are {\it not} \hm s but almost multiplicative. Such problem is
already difficult when $C=M_n$ but it has been proved that these
above mentioned maps are close to \hm s if the associated
$K$-theoretical data of these maps are consistent with those of \hm
s. It is problematic when one tries to replace $M_n$ by $C([0,1],
M_n).$
In addition to the usual $K$-theory and trace information, one
also has to handle the maps from $U(C)/CU(C)$ to $U(A)/CU(A),$
where $CU(C)$  and $CU(A)$ are the closure of the subgroups of
$U(C)$ and $U(A)$ generated by commutators, respectively. Other
problems occur because of lack of projections in \CA s which are
not of real rank zero.

The main theorem is stated as follows: Let $C$ be a unital separable simple
amenable \CA\, with tracial rank one which satisfies the Universal
Coefficient Theorem.
 For any $\ep>0$ and any finite subset ${\cal F}\subset C,$
there exists $\dt>0,$ a finite subset ${\cal G}\subset C$ and a
finite subset ${\cal P}\subset\underline{K}(C)$ satisfying the
following:

Suppose that $A$ is a unital simple  \CA\, with tracial rank no more
than one, suppose that $\phi: C\to A$ is a unital \hm\, and $u\in
U(A)$ such that
\beq\label{pMT-1}
\|[\phi(c),\, u]\|<\dt\tforal c\in {\cal G}\andeqn {\rm
Bott}(\phi, u)|_{\cal P}=0.
\eneq
Then there exists a continuous path of unitaries $\{u(t): t\in
[0,1]\}\subset A$ such that
\beq\label{pMT-2}
u(0)=u,\,\,\,u(1)=1\andeqn \|[\phi(c),\, u(t)]\|<\ep\tforal c\in
{\cal F}
\eneq
and for all $t\in [0,1].$

We also give the following  Basic Homotopy Lemma in simple \CA\,
with tracial rank one (see  \ref{CMain} below) :

Let $\ep>0$ and let $\Delta: (0,1)\to (0,1)$ be a nondecreasing
map.  We show that there exists $\dt>0$ and $\eta>0$ (which does
not depend on $\Delta$) satisfying the following:

Given any pair of unitaries $u$ and $v$ in a unital  simple \CA\,
$A$ with tracial rank no more than one such that $[v]=0$ in
$K_1(A),$
\beq\nonumber
\|[u, \, v]\|<\dt,\,\,\, {\rm bott}_1(u,v)=0\andeqn \mu_{\tau\circ
\imath}(I_a)\ge \Delta(a)
\eneq
for all open arcs $I_a$  with length $a\ge \eta,$ there exists a
continuous path of unitaries $\{v(t): t\in [0,1]\}\subset A$ such
that
$$
v(0)=v,\,\,\,v(1)=1\andeqn \|[u,\, v(t)]\|<\ep\tforal t\in [0,1],
$$
where $\imath: C(\T)\to A$ is the \hm\, defined by $\imath(f)=f(u)$
for all $f\in C(\T)$ and $\mu_{\tau\circ \imath}$ is the Borel
probability measure induced by the state $\tau\circ \imath.$ It
should be noted that, unlike the case that $A$ has real rank zero,
the length of $\{v(t)\}$ can not be controlled. In fact, it could be
as long as one wishes.

In  a subsequent paper, we use the main homotopy result (Theorem
\ref{MT}) of this paper and the results in \cite{Lnn1} to establish
a $K$-theoretical necessary and sufficient condition for \hm s from
 unital simple AH-algebras into a unital separable simple \CA\, with
tracial rank no more than one to be asymptotically unitarily
equivalent which, in turn, combining with a result of W. Winter,
provides a classification theorem for a class of unital separable
simple amenable \CA s which properly contains all unital separable
simple amenable \CA s with tracial rank no more than one which
satisfy the UCT as well as some projectionless \CA s such as the
Jiang-Su algebra.

\section{Preliminaries and notation}

\begin{NN}

{\rm Let $A$ be a unital \CA. Denote by $\text{T}(A)$ the tracial state
space of $A$ and denote by ${\rm Aff}(\text{T}(A))$ the set of affine
continuous functions on $\text{T}(A).$

Let $C=C(X)$ for some compact metric space $X$ and let $L: C\to A$
be a unital positive linear map. Denote by $\mu_{\tau\circ L}$ the
Borel probability measure induced by the state $\tau\circ L,$
where $\tau\in \text{T}(A).$

 }
\end{NN}

\begin{NN}
{\rm Let $a$ and $b$ be two elements in a \CA\, $A$ and let $\ep>0$
be a positive number. We write $a\approx_{\ep} b$ if $\|a-b\|<\ep.$
Let $L_1, L_2: A\to C$ be two maps from $A$ to another \CA\, $C$ and
let ${\cal F}\subset A$ be a subset. We write
$$
L_1\approx_{\ep} L_2\,\,\,{\rm on}\,\,\, {\cal F},
$$
if $L_1(a)\approx_{\ep}L_2(a)$ for all $a\in {\cal F}.$

Suppose that $B\subset A.$ We write $a\in_{\ep} B$ if there is an
element $b\in B$ such that $\|a-b\|<\ep.$

Let ${\cal G}\subset A$ be a subset. We say $L$ is $\ep$-${\cal
G}$-multiplicative if, for any $a, b\in {\cal G},$
$$
L(ab)\approx_{\ep} L(a)L(b)
$$
for all $a, \, b\in {\cal G}.$

}
\end{NN}

\begin{NN}\label{DU0}
{\rm  Let $A$ be a unital \CA.  Denote by $U(A)$ the unitary group
of $A.$ Denote by $U_0(A)$ the normal subgroup of $U(A)$ consisting
of those unitaries in the path connected component of $U(A)$
containing the identity. Let $u\in U_0(A).$ Define
$$
{\rm cel}_A(u)=\inf\{ {\rm length}(\{u(t)\}): u(t)\in C([0,1],
U_0(A)), u(0)=u\andeqn u(1)=1_A\}.
$$
We use ${\rm cel}(u)$ if the \CA\, $A$ is not in question.

}\end{NN}

\begin{NN}\label{DU1}
{\rm

Denote  by $CU(A)$ the {\it closure} of the subgroup generated by
the commutators of $U(A).$ For $u\in U(A),$ we will use ${\bar u}$
for the image of $u$ in $U(A)/CU(A).$ If ${\bar u}, {\bar v}\in
U(A)/CU(A),$ define
$$
{\rm dist}({\bar u}, {\bar v})=\inf\{\|x-y\|: x, y\in U(A)\,\,\,{\rm
such\,\,\, that}\,\,\, {\bar x}={\bar u}, {\bar y}={\bar v}\}.
$$
If $u, v\in U(A),$ then
$$
{\rm dist}({\bar u}, {\bar v})=\inf\{\|uv^*-x\|: x\in CU(A)\}.
$$

}

\end{NN}

\begin{NN}\label{DU2}
{\rm Let $A$ and $B$ be two unital \CA s and let $\phi: A\to B$ be a
unital \hm. It is easy to check that $\phi$ maps $CU(A)$ to $CU(B).$
Denote by $\phi^{\ddag}$ the \hm\, from $U(A)/CU(A)$ into
$U(B)/CU(B)$ induced by $\phi.$  We also use $\phi^{\ddag}$ for the
\hm\, from $U(M_k(A))/CU(M_k(A))$ into $U(M_k(B))/CU(M_k(B))$
($k=1,2,...$).

}

\end{NN}

\begin{NN}\label{DU3}
{\rm Let $A$ and $C$ be two unital \CA s and let $F\subset U(C)$ be
a subgroup of $U(C).$  Suppose that $L: F\to U(A)$  is a \hm\, for
which $L(F\cap CU(C))\subset CU(A).$ We will use $L^{\ddag}:
F/CU(C)\to U(A)/CU(A)$ for the induced map.

}

\end{NN}

\begin{NN}
{\rm Let $A$ and $B$ be in \ref{DU3}, let $1>\ep>0$ and let ${\cal
G}\subset A$ be a subset. Suppose that $L$ is $\ep$-${\cal
G}$-multiplicative unital completely positive linear map. Suppose
that $u,\,u^*\in {\cal G}.$ Define $\langle
L\rangle(u)=L(u)L(u^*u)^{-1/2}.$

}
\end{NN}

\begin{df}\label{Dbot2}
{\rm Let
$A$ and $B$ be  two unital \CA s.  Let $h: A\to B$ be a \hm\, and
let $v\in U(B)$ such that
$$
h(g)v=vh(g)\,\rforal\, g\in A.
$$
 Thus we
obtain a \hm\, ${\bar h}: A\otimes C(S^1)\to B$ by ${\bar
h}(f\otimes g)=h(f)g(v)$ for $f\in A$ and $g\in C(S^1).$ 
From the  following
splitting exact sequence:
\beq\label{botadd-1}
0\to SA \to A\otimes C(S^1)\leftrightarrows A\to 0
\eneq
and the isomorphisms $K_i(A)\to K_{1-i}(SA)$ ($i=0,1$) given by the Bott periodicity, one obtains
 two injective \hm s:
\beq\label{dbot01}
\bt^{(0)}&:& K_0(A)\to K_1(A\otimes C(S^1))\\
 \bt^{(1)}&:&
K_1(A)\to K_0(A\otimes C(S^1)).
\eneq
Note, in this way, one can write
$K_i(A\otimes C(S^1))=K_i(A)\bigoplus \bt^{(1-i)}(K_{1-i}(A)).$ We
use $\widehat{\bt^{(i)}}: K_i(A\otimes C(S^1))\to
\bt^{(1-i)}(K_{1-i}(A))$ for the  the projection to the summand
$\bt^{(1-i)}(K_{1-i}(A)).$

For each integer $k\ge 2,$ one also obtains the following injective
\hm s:
\beq\label{dbot02}
\bt^{(i)}_k: K_i(A, \Z/k\Z)\to K_{1-i}(A\otimes C(S^1), \Z/k\Z),
i=0,1.
\eneq
Thus we write
\beq\label{dbot02-1}
K_{1-i}(A\otimes C(S^1), \Z/k\Z)=K_{1-i}(A,\Z/k\Z)\bigoplus
\bt^{(i)}_k(K_i(A, \Z/k\Z)),\,\,i=0,1.
\eneq
Denote by $\widehat{\bt^{(i)}_k}: K_{i}(A\otimes C(S^1), \Z/k\Z)\to
\bt^{(1-i)}_k(K_{1-i}(A,\Z/k\Z))$ similarly to that of
$\widehat{\bt^{(i)}}.,$ $i=1,2.$ If $x\in \underline{K}(A),$ we use
${\boldsymbol{\beta}}(x)$ for $\bt^{(i)}(x)$ if $x\in K_i(A)$ and
for $\bt^{(i)}_k(x)$ if $x\in K_i(A, \Z/k\Z).$ Thus we have a map
${\boldsymbol{ \bt}}: \underline{K}(A)\to \underline{K}(A\otimes
C(S^1))$ as well as $\widehat{\boldsymbol{\bt}}:
\underline{K}(A\otimes C(S^1))\to
 {\boldsymbol{ \bt}}(\underline{K}(A)).$ Thus one may write
 $\underline{K}(A\otimes C(S^1))=\underline{K}(A)\bigoplus {\boldsymbol{ \bt}}( \underline{K}(A)).$

On the other hand ${\bar h}$ induces \hm s ${\bar h}_{*i,k}:
K_i(A\otimes C(S^1)), \Z/k\Z)\to K_i(B,\Z/k\Z),$ $k=0,2,...,$ and
$i=0,1.$ We use $\text{Bott}(h,v)$ for all  \hm s ${\bar
h}_{*i,k}\circ \bt^{(i)}_k.$ We write
$$
\text{Bott}(h,v)=0,
$$
if ${\bar h}_{*i,k}\circ \bt^{(i)}_k=0$ for all $k\ge 1$ and
$i=0,1.$

We will use $\text{bott}_1(h,v)$ for the \hm\, ${\bar h}_{1,0}\circ
\bt^{(1)}: K_1(A)\to K_0(B),$ and $\text{bott}_0(h,u)$ for the \hm\,
${\bar h}_{0,0}\circ \bt^{(0)}: K_0(A)\to K_1(B).$

Since $A$ is unital, if $\text{bott}_0(h,v)=0,$ then $[v]=0$ in
$K_1(B).$




 For a fixed finite subset ${\mathcal P}\subset \underline{K}(A),$
 there exists $\dt>0$ and a finite subset ${\mathcal
G}\subset A$ such that, if $v\in B$ is a unitary for which
$$
\|h(a)v-vh(a)\|<\dt\,\rforal\, a\in {\mathcal G},
$$
then $\text{Bott}(h,v)|_{{\mathcal P}}$ is well defined. In what
follows, whenever we write $\text{Bott}(h,v)|_{\mathcal P},$ we mean
that $\dt$ is sufficiently small and ${\mathcal G}$ is sufficiently
large so it is well defined.

Now suppose that $A$ is also amenable and $K_i(A)$ is finitely
generated ($i=0,1$). For example, $A=C(X),$ where $X$ is a finite CW
complex.
When $K_i(A)$ is finitely generated, $\text{Bott}(h,v)|_{{\mathcal
P}_0}$ defines $\text{Bott}(h,v)$ for some sufficiently large finite
subset ${\mathcal P}_0.$  In what follows such ${\mathcal P}_0$ may
be denoted by ${\mathcal P}_A.$ Suppose that ${\mathcal P}\subset
\underline{K}(A)$ is a larger finite subset, and ${\mathcal
G}\supset {\mathcal G}_0$ and $0<\dt<\dt_0.$

{\it A fact that we be used in this paper is that,
${\rm Bott}(h,v)|_{\mathcal P}$ defines the same map
${\rm Bott}(h,v)$ as ${\rm Bott}(h,v)|_{{\mathcal P}_0}$ defines,
if}
$$
\|h(a)v-vh(a)\|<\dt\rforal \, a\in {\mathcal G},
$$
{\it when $K_i(A)$ is finitely generated.} In what follows, in the case
that $K_i(A)$ is finitely generated, whenever we write
$\text{Bott}(h,v),$ we always assume that $\dt$ is smaller than
$\dt_0$  and ${\mathcal G}$ is  larger than ${\mathcal G}_0$  so
that $\text{Bott}(h,v)$ is well-defined (see 2.10 of \cite{Lnhomp}
for more details).

}

\end{df}

\begin{NN}\label{botsmall}

{\rm In the case that $A=C(S^1),$ there is a concrete way to
visualize $\text{bott}_1(h,v).$ It is perhaps helpful to describe it
here. The map $\text{bott}_1(h,v)$ is determined by
$\text{bott}_1(h, v)([z]),$ where $z$ is the identity map on the
unit circle.

Denote $u=h(z)$ and define
\beq\nonumber
f(e^{2\pi i t})=\begin{cases} 1-2t, &\text{if $0\le t\le 1/2$,}\\
  -1+2t, & \text{if $1/2<t\le 1$,}
  \end{cases}
  \eneq
  \beq\nonumber
g(e^{2\pi i t})=\begin{cases} (f(e^{2\pi i t})-f(e^{2\pi it})^2)^{1/2} &\text{if $0\le t\le 1/2$,}\\
 0 , & \text{if $1/2<t\le 1$ \,\,\,and}
  \end{cases}
  \eneq
\beq\nonumber
h(e^{2\pi i t})=\begin{cases}0 , &\text{if $0\le t\le 1/2$,}\\
 (f(e^{2\pi i t})-f(e^{2\pi it})^2)^{1/2} , & \text{if $1/2<t\le 1$,}
  \end{cases}
  \eneq
These are non-negative continuous functions defined on the unit
circle. Suppose that $uv=vu.$ Define
\beq\label{bot1}
{\mathtt{b}}(u,v)=\begin{pmatrix} f(v) & g(v)+h(v)u^*\\
                                     g(v)+uh(v) & 1-f(v)
           \end{pmatrix}
\eneq

Then ${\mathtt{b}}(u,v)$ is a projection. There is $\dt_0>0$
(independent of unitaries $u, v$ and $A$)  such that if
$\|[u,v]\|<\dt_0,$ the spectrum of the positive element
${\mathtt{p}}(u,v)$ has a gap at $1/2.$ The bott element of $u$ and
$v$ is an element in $K_0(A)$ as defined  in \cite{EL1}  and
\cite{EL2}  which may be represented by
\beq\label{bot2}
{\rm bott}_1(u,v)=[\chi_{[1/2, \infty)}({\mathtt{b}}(u,v))]-[\begin{pmatrix} 1 & 0\\
0 & 0
\end{pmatrix}].
 \eneq

Note that $\chi_{[1/2, \infty)}$ is a continuous function on ${\rm
sp}({\mathtt{b}}(u,v)).$ Suppose that ${\rm
sp}({\mathtt{b}}(u,v))\subset (-\infty, a] \cup [1-a,\infty)$ for
some $0<a<1/2.$ Then $\chi_{[1/2, \infty)}$ can be replaced by any
other positive continuous function $F$ for which  $F(t)=0$ if
$t\le a$ and $F(t)=1$ if $t\ge 1/2.$

}

\end{NN}

\begin{df}\label{THfull}
{\rm Let $A$ and $C$ be two unital \CA s. Let $N:
C_+\setminus\{0\}\to \N$ and $K: C_+\setminus\{0\} \to
\R_+\setminus\{0\}$ be two maps. Define $T=N\times K:
C_+\setminus\{0\}\to \N\times \R_+\setminus \{0\}$ by $T(c)=(N(c),
K(c))$ for $c\in C_+\setminus\{0\}.$ Let $L: C\to A$ be a unital
positive linear map. We say $L$ is $T$-full if for any $c\in
C_+\setminus\{0\},$ there are $x_1,x_2,...,x_{N(c)}\in A$ with
$\|x_i\|\le K(c)$ such that
$$
\sum_{i=1}^{N(C)}x_i^*L(c)x_i=1_A.
$$
Let ${\cal H}\subset C_+\setminus\{0\}.$ We say that $L$ is $T$-${\cal
H}$-full if
$$
\sum_{i=1}^{N(c)}x_i^*L(c)x_i=1_A
$$
for all $c\in {\cal H}.$

}

\end{df}

\begin{df}\label{DI}

{\rm Denote by ${\cal I}$ the class of unital \CA s with the form
$\bigoplus_{i=1}^mC(X_i, M_{n(i)}),$ where $X_i=[0,1]$ or $X_i$ is one
point.

}

\end{df}

\begin{df}\label{Dtr1}
{\rm Let $k\ge 0$ be an integer. Denote by ${\cal I}_k$ the class of
all \CA s $B$ with the form $B=PM_m(C(X))P,$ where $X$ is a finite
CW complex with dimension no more than $k,$ $P$ is a projection in
$M_m(C(X)).$

Recall that a unital simple \CA\, $A$ is said to have tracial rank
no more than $k$ (write $TR(A)\le k$) if the following holds:
For any $\ep>0,$ any positive element $a\in A_+\setminus\{0\}$ and
any finite subset ${\cal F}\subset A,$ there exists a non-zero
projection $p\in A$ and a \SCA\, $B\in {\cal I}_k$ with $1_B=p$ such
that

(1) $\|xp-px\|<\ep\tforal x\in {\cal F};$

(2) $pxp\in_{\ep} B$ for all $x\in {\cal F}$ and

(3) $1-p$ is von Neumann equivalent to a projection in
$\overline{aAa}.$

If $TR(A)\le k$ and $TR(A)\not=k-1,$ we say $A$ has tracial rank $k$
and write $TR(A)\le k.$ It has been shown that if $TR(A)=1,$ then,
in the above definition, one can replace $B$ by a \CA\,  in ${\cal
I}$
(see \cite{Lntr}).
All unital simple AH-algebra with slow dimension growth and real
rank zero have tracial rank zero (see \cite{EG} and also \cite{Lnah}) and
all unital simple AH-algebras with no dimension growth have tracial
rank no more than one (see \cite{G}, or, Theorem 2.5  of
\cite{Lnctr1}). Note that all AH-algebras satisfy the Universal
Coefficient Theorem. There are unital separable simple \CA\, $A$
with $TR(A)=0$ (and $TR(A)=1$)  which are not amenable.

}
\end{df}

\section{Unitary groups }

 \begin{lem}\label{length}
Let $H>0$ be a positive number and let $N\ge 2$ be an integer.
Then, for any unital \CA\, $A,$ any projection $e\in A$  and any
$u\in U_0(eAe)$ with ${\rm cel}_{eAe}(u)<
 H,$
 \beq\label{length-1}
 {\rm dist}({\overline{ u+(1-e)}}, {\bar 1})<H/N,
 \eneq
 if there are mutually orthogonal and mutually equivalent
 projections $e_1, e_2,...,e_{2N}\in (1-e)A(1-e)$ such that $e_1$ is
 also equivalent to $e.$

 \end{lem}

\begin{proof}

Since ${\rm cel}_{eAe}(u)<H,$ there are unitaries $u_0,
u_1,...,u_{N}\in eAe$ such that
\beq\label{length-2}
u_0=u,\,\,\, u_{N}=1\andeqn \|u_i-u_{i-1}\|<H/N,\,\,i=1,2,...,N.
\eneq

 We will use the fact that
$$
\begin{pmatrix} v & 0\\
                  0 & v^*\end{pmatrix} =\begin{pmatrix} v &0\\
                                                       0 &
                                                       1\end{pmatrix}
                 \begin{pmatrix} 0 &1\\
                                 1 & 0\end{pmatrix} \begin{pmatrix}
                                 v^* &0\\
                                  0 &1\end{pmatrix}\begin{pmatrix} 0
                                  & 1\\
                            1& 0\end{pmatrix}.
                            $$
In particular, $\begin{pmatrix} v & 0\\
                  0 & v^*\end{pmatrix}$ is a commutator.
Note that
\beq\label{length-3}
\hspace{-0.2in}\|(u\oplus u_1^* \oplus u_1\oplus u_2^*\oplus
\cdots\oplus u_N^*\oplus u_N)-(u\oplus u^*\oplus u_1\oplus
u_1^*\cdots\oplus u_{N-1}^*\oplus u_{N})\|<H/N.
\eneq
Since $u_N=1,$ $u\oplus u^*\oplus u_1\oplus u_1^*\cdots\oplus
u_{N-1}^*\oplus u_{N}$ is a commutator.

Now we write
$$
u\oplus e_1\cdots \oplus e_{2N}=(u\oplus u_1^* \oplus u_1\oplus
\cdots\oplus u_N^*\oplus u_N)(e\oplus u_1\oplus u_1^*\oplus \cdots
\oplus u_N \oplus u_N^*).
$$
We obtain  $z\in CU((e+\sum_{i=1}^{2N}e_i)A(e+\sum_{i=1}^{2N}e_i))$
such that
$$
\|u\oplus e_1\cdots \oplus e_{2N}-z\|<H/N.
$$
It follows that
$$
{\rm dist}({\overline{u+(1-e)}}, {\bar 1})<H/N.
$$

\end{proof}

\begin{df}\label{Ddeter}
{\rm Let $C=PM_k(C(X))P,$ where $X$ is a compact metric space and
$P\in M_k(C(X))$ is a projection.   Let $u\in U(C).$ Recall (see
\cite{Ph1}) that
$$
D_c(u)=\inf\{\|a\|: a\in C_{s.a.} \,\,\,{\rm such
\,\,\,that}\,\,\,{\rm det}(\exp(ia)\cdot u)(x)=1\tforal  x\in X\}.
$$
If no self-adjoint element $a\in A_{s.a.}$ exists for which ${\rm
det}(\exp(ia)\cdot u)(x)=1$ for all $x\in X,$ define
$D_c(u)=\infty.$

}
\end{df}

\begin{lem}\label{ph}
Let $K\ge 1$ be an integer.  Let $A$ be a  unital  simple \CA\,
with $TR(A)\le 1,$ let  $e\in A$ be a projection and let $u\in
U_0(eAe).$ Suppose that $w=u+(1-e)$ and suppose   $\eta>0.$
Suppose also that
\beq\label{PH-1}
[1-e]\le  K[e]\,\,\,{\rm in}\,\,\, K_0(A) \andeqn {\rm dist}({\bar
w},{\bar 1})<\eta.
\eneq
Then, if $\eta<2,$
$$
{\rm cel}_{eAe}(u)<({K\pi\over{2}}+1/16)\eta+8\pi \andeqn {\rm
dist}({\bar u}, {\bar e})<(K+1/8)\eta,
$$
and if $\eta=2,$
$$
{\rm cel}_{eAe}(u)<{K\pi\over{2}}{\rm cel}(w)+1/16+8\pi.
$$

\end{lem}

\begin{proof}
We  assume that (\ref{PH-1}) holds. Note that $\eta\le 2.$ Put
$L={\rm cel}(w).$

We first consider the case that $\eta<2.$ There is a projection
$e'\in M_2(A)$ such that
$$
[(1-e)+e']=K[e].
$$
To simplify notation,  by replacing $A$ by
$(1_A+e')M_2(A)(1_A+e')$ and $w$ by $w+e',$ without loss of
generality, we may now assume that
\beq\label{PH-10}
[1-e]=K[e]\andeqn {\rm dist}({\bar w},{\bar 1})<\eta.
\eneq
There is $R_1>1$ such that
$\max\{L/R_1,2/R_1,\eta\pi/R_1\}<\min\{\eta/64, 1/16\pi\}.$

 For
any ${\eta\over{32K(K+1)\pi}}>\ep>0$ with $\ep+\eta<2,$  since
$TR(A)\le 1,$ there exists a projection $p\in A$ and a \SCA\, $D\in
{\cal I}$ with $1_D=p$ such that
\begin{enumerate}

\item $\|[p,\,x]\|<\ep$ for $x\in \{u, w, e, (1-e)\},$

\item $pwp, pup, pep, p(1-e)p\in_{\ep} D,$

\item there is a  projection $q\in D$ and a unitary $z_1\in qDq$
such that $\|q-pep\|<\ep,$ $\|z_1-quq\|<\ep,$ $\|z_1\oplus (p-
q)-pwp\|<\ep$ and $\|z_1\oplus (p-q)-c_1\|<\ep+{\eta};$

\item  there is a  projection $q_0\in (1-p)A(1-p)$ and a unitary
$z_0\in q_0Aq_0$ such that\\ $\|q_0-(1-p)e(1-p)\|<\ep,$
$\|z_0-(1-p)u(1-p)\|<\ep,$ $\|z_0\oplus
(1-p-q_0)-(1-p)w(1-p)\|<\ep,$ $\|z_0\oplus
(1-p-q_0)-c_0\|<\ep+{\eta},$

\item $[p-q]=K [q]$ in $K_0(D),$ $[(1-p)-q_0]=K[q_0]$ in $K_0(A);$

\item $2(K+1)R_1[1-p]<  [p]$ in $K_0(A);$

\item $ {\rm cel}_{(1-p)A(1-p)}(z_0\oplus (1-p-q_0))\le L+\ep,$

\end{enumerate}
where
 $c_1\in CU(D)$ and
$c_0\in CU((1-p)A(1-p)).$

Note that $D_D(c_1)=0$ (see \ref{Ddeter}). Since $\ep+\eta<2,$
there is $h\in D_{s.a}$ with $\|h\|\le
2\arcsin({\ep+\eta\over{2}})$ such that (by (3) above)
\beq\label{ph-5}
(z_1\oplus (p-q))\exp(ih)=c_1.
\eneq
It follows that
\beq\label{ph-6}
D_D((z_1\oplus(p-q))\exp(ih))=0.
\eneq
By (5) above and applying 3.3 of \cite{Ph1}, one obtains that
\beq\label{ph-7}
|D_{qDq}(z_1)|\le K2\arcsin({\ep+\eta\over{2}}).
\eneq
If $2K\arcsin({\ep+\eta\over{2}})\ge \pi,$ then
$$
2K({\ep+\eta\over{2}}){\pi\over{2}}\ge \pi.
$$
It follows that
\beq\label{ph-8-}
K(\ep+\eta)\ge 2\ge {\rm dist}(\overline{z_1}, {\overline{q}}).
\eneq
 Since those unitaries in $D$ with ${\rm det}(u)=1$ (for all
points) are in $CU(D)$ (see, for example, 3.5 of \cite{EGL}), from
(\ref{ph-7}), one computes that, when
$2K\arcsin({\ep+\eta\over{2}})< \pi,$
\beq\label{ph-8}
{\rm dist}(\overline{z_1},
{\overline{q}})<2\sin(K\arcsin({\ep+\eta\over{2}}))\le
K(\ep+\eta).
\eneq
By combining both (\ref{ph-8-}) and (\ref{ph-8}), one obtains that
\beq\label{ph-8+}
{\rm dist}(\overline{z_1}, {\overline{q}})\le K(\ep+\eta) \le
K\eta+{\eta\over{32(K+1)\pi}}.
\eneq
By (\ref{ph-7}), it follows from 3.4 of \cite{Ph1} that
\beq\label{ph-8+1}
{\rm cel}_{qDq}(z_1)\le 2K\arcsin{\ep+\eta\over{2}} +6\pi \le
K(\ep+\eta){\pi\over{2}}+6\pi\le
(K{\pi\over{2}}+{1\over{64(K+1)}})\eta +6\pi .
\eneq
By (5) and (6) above,
$$
(K+1)[q]=[p-q]+[q]=[p]>2(K+1)R_1[1-p].
$$
Since $K_0(A)$ is weakly unperforated, one has
\beq\label{ph-9}
2R_1[1-p]<[q].
\eneq

There is a unitary $v\in A$ such that
\beq\label{ph-10}
v^*(1-p-q_0)v\le q.
\eneq
Put $v_1=q_0\oplus (1-p-q_0)v.$ Then
\beq\label{ph-11}
v_1^*(z_0\oplus (1-p-q_0))v_1=z_0\oplus v^*(1-p-q_0)v.
\eneq

Note that
\beq\label{ph-10+n}
\|(z_0\oplus v^*(1-p-q_0)v)v_1^*c_0^*v_1-q_0\oplus
v^*(1-p-q_0)v\|<\ep+\eta.
\eneq
Moreover, by (7) above,
\beq\label{ph-10+}
{\rm cel}(z_0\oplus v^*(1-p-q_0)v)\le L+\ep,
\eneq
It follows from (\ref{ph-9}) and Lemma 6.4 of \cite{Lnctr1} that
\beq\label{ph-10+1}
{\rm cel}_{(q_0+q)A(q_0+q)}(z_0\oplus q)\le 2\pi+(L+\ep)/R_1.
\eneq
Therefore, combining (\ref{ph-8+1}),
\beq\label{ph-10+2}
{\rm cel}_{(q_0+q)A(q_0+q)}(z_0+z_1)\le
2\pi+(L+\ep)/R_1+(K{\pi\over{2}}+{1\over{64(K+1)}})\eta +6\pi.
\eneq
By (\ref{ph-10+}), (\ref{ph-9}) and \ref{length},
in $U_0((q_0+q)A(q_0+q))/CU((q_0+q)A(q_0+q)),$
\beq\label{ph-13}
{\rm dist}(\overline{z_0+q},
{\overline{q_0+q}})<{(L+\ep)\over{R_1}}.
\eneq
Therefore, by (\ref{ph-8}) and (\ref{ph-13}),
\beq\label{ph-14}
{\rm dist}({\overline{z_0\oplus z_1}},
{\overline{q_0+q}})<{(L+\ep)\over{R_1}}+K\eta+{\eta\over{32(K+1)\pi}}<(K+1/16)\eta.
\eneq
We note that
\beq
\|e-(q_0+q)\|<2\ep\andeqn \|u-(z_0+z_1)\|<2\ep.
\eneq

It follows that
\beq\label{ph-15}
{\rm dist}({\bar u}, {\bar e})<4\ep +(K+1/16)\eta<(K+1/8)\eta.
\eneq

Similarly, by (\ref{ph-10+2}),
\beq\label{ph-15+1}
{\rm cel}_{eAe}(u)&\le& 4\ep\pi+2\pi+(L+\ep)/R_1+(K{\pi\over{2}}+{1\over{64(K+1)}})\eta
+6\pi\\
&<&(K{\pi\over{2}}+1/16)\eta+8\pi.
\eneq

This proves the case that $\eta<2.$

Now suppose that $\eta=2.$
Define $R=[{\rm cel}(w)+1].$ Note that ${{\rm cel}(w)\over{R}}<1.$
There is a projection $e'\in M_{R+1}(A)$ such that
$$
[(1-e)+e']=(K+RK)[e].
$$
It follows from \ref{length} that
\beq\label{PH-111}
{\rm dist}(\overline{w\oplus e'},{\overline{ 1_A+e'}})<{{\rm
cel}(w)\over{R+1}}.
\eneq

Put $K_1=K(R+1).$ To simplify notation,  without loss of
generality, we may now assume that
\beq\label{PH-102}
[1-e]=K_1[e]\andeqn {\rm dist}({\bar w},{\bar 1})<{{\rm
cel}(w)\over{R+1}}.
\eneq

It follows from the first part of the lemma that
\beq\label{PH-103}
\hspace{-0.3in}{\rm cel}_{eAe}(u)&< &
({K_1\pi\over{2}}+{1\over{16}}){{\rm cel}(w)\over{R+1}}+8\pi\\
&\le &{K\pi{\rm cel}(w)\over{2}}+{1\over{16}}+8\pi.
\eneq

\end{proof}

\begin{thm}\label{inject}
Let $A$ be a unital  simple \CA\, with $TR(A)\le 1$ and let
$e\in A$ be a non-zero projection. Then the map $u\mapsto u+(1-e)$
induces an isomorphism $j$ from $U(eAe)/CU(eAe)$ onto $U(A)/CU(A).$
\end{thm}

\begin{proof}
It was shown in Theorem 6.7 of \cite{Lnctr1} that $j$ is a
surjective \hm. So it remains to show that it is also injective. To
do this, fix a unitary $u\in eAe$ so that ${\bar u}\in {\rm ker}\,
j.$ We will show that $u\in CU(eAe).$

There is an integer $K\ge 1$ such that
$$
K[e]\ge [1-e]\,\,\,{\rm in}\,\,\, K_0(A).
$$
Let $1>\ep>0.$ Put $v=u+(1-e).$ Since ${\bar u}\in {\rm ker}j,$
$v\in CU(A).$ In particular,
$$
{\rm dist}({\bar v}, \bar{1})<\ep/(K\pi/2+1).
$$
It follows from Lemma \ref{ph} that
$$
{\rm dist}({\bar u}, {\bar
e})<({K\pi\over{2}}+1/16)(\ep/(K\pi/2+1))<\ep.
$$
It then follows that
$$
u\in CU(eAe).
$$

\end{proof}

\begin{cor}\label{c1}
Let $A$ be a unital  simple \CA\, with $TR(A)\le 1.$ Then
the map $j: a\to {\rm diag}(a, \overbrace{1,1,..,1}^m)$ from $A$ to
$M_n(A)$ induces an isomorphism from $U(A)/CU(A)$ onto
$U(M_n(A))/CU(M_n(A))$ for any integer $n\ge 1.$
\end{cor}

\section{Full spectrum}

%
%

%

One should compare the following with Theorem 3.1 of \cite{Su}.

\begin{lem}\label{diag}
Let $X$ be a path connected finite CW
complex, let $C=C(X)$ and let $A=C([0,1], M_n)$ for some integer
$n\ge 1$. For any unital \hm\, $\phi: C\to A,$ any finite subset
${\cal F}\subset C$ and any $\ep>0,$ there exists a unital \hm\,
$\psi: C\to B$ such that
\beq\label{diag-1}
\|\phi(c)-\psi(c)\|<\ep\tforal c\in {\cal F}\andeqn\\
\psi(f)(t)=W(t)^*\begin{pmatrix} f(s_1(t)) & &\\
                                   & \ddots &\\
                                    && f(s_n((t)))\end{pmatrix}
                                    W(t),
                                    \eneq
where $W\in U(A),$ $s_j\in C([0,1],X),$ $j=1,2,...,n,$ and $t\in
[0,1].$

\end{lem}

\begin{proof}
 To simplify the notation, without loss of generality, we may assume
 that ${\cal F}$ is in the unit ball of $C.$
Since $X$ is also locally path connected, choose $\dt_1>0$ such that, for any point $x\in X,$ $B(x, \dt_1)$
is path connected.
Put $d=2\pi/n.$  Let $\dt_2>0$ (in place of $\dt$) be as required
by Lemma 2.6.11 of \cite{Lnbk} for $\ep/2.$

We will also apply Corollary 2.3 of \cite{Su}. By Corollary 2.3 of
\cite{Su}, there exists a finite subset ${\cal H}$ of positive
functions in $C(X)$ and $\dt_3>0$ satisfying the following: For any
pair of points $\{x_i\}_{i=1}^n$ and $\{y_i\}_{i=1}^n,$ if
$\{h(x_i)\}_{i=1}^n$ and $\{h(y_i)\}_{i=1}^n$ can be paired to
within $\dt_3$ one by one, in increasing order, counting
multiplicity,  for all $h\in {\cal H},$ then $\{x_i\}_{i=1}^n$ and
$\{y_i\}_{i=1}^n$ can be paired to within $\dt_3/2,$ one by one.

Put $\ep_1=\min\{\ep/16, \dt_1/16, \dt_2/4, \dt_3/4\}.$ There
exists $\eta>0$ such that
\beq\label{diag-2}
|f(t)-f(t')|<\ep_1/2\tforal f\in \phi({\cal F}\cup {\cal H})
\eneq
provided that $|t-t'|<\eta.$ Choose a partition of the interval:
$$
0=t_0<t_1<\cdots <t_N=1
$$
such that $|t_i-t_{i-1}|<\eta,$ $i=1,2,...,N.$ Then
\beq\label{diag-3}
\|\phi(f)(t_i)-\phi(f)(t_{i-1})\|<\ep_1\tforal f\in {\cal
F}\cup{\cal H},
\eneq
$i=1,2,...,N.$ There are unitaries $U_i\in M_n$ and
$\{x_{i,j}\}_{j=1}^n,$ $i=0,1,2,...,N,$ such that
\beq\label{diag-5}
\phi(f)(t_i)=U_i^*\begin{pmatrix} f(x_{i,1}) & &\\
                              & \ddots &\\
                              && f(x_{i, n})\end{pmatrix}U_i.
                              \eneq

By the Weyl spectral variation inequality  (see \cite{rB}), the
eigenvalues of $\{h(x_{i,j})\}_{j=1}^n$ and
$\{h(x_{i-1,j})\}_{j=1}^n$ can be paired to within $\dt_3,$ one by
one, counting multiplicity, in decreasing order. It follows from
Corollary 2.3 of \cite{Su} that $\{x_{i,j}\}_{j=1}^n$ and $\{x_{i-1,
j}\}_{j=1}^n$ can be paired within $\dt_3/2.$  We may assume that,
\beq\label{diag-6}
{\rm dist}(x_{i,\sigma_i(j)}, x_{i-1,j})<\dt_3/2,
\eneq
where $\sigma_i: \{1,2,...,n\}\to \{1,2,...,n\}$ is a permutation.
By the choice of $\dt_3,$  there is continuous path $\{x_{i-1,j}(t): t\in
[t_{i-1}, (t_i+t_{i-1})/2]\}\subset B(x_{i-1}, \dt_3/2)$ such that
\beq\label{diag-6+-}
x_{i-1,j}(t_{i-1})=x_{i-1,j}\andeqn x_{i-1,
j}((t_{i-1}+t_i)/2)=x_{i,\sigma_i(j)},
\eneq
$j=1,2,...,n.$ Put
\beq\label{diag-6+}
\psi(f)(t)=U_{i-1}^*\begin{pmatrix} f(x_{i,1}(t)) & &\\
&\ddots &\\
& & f(x_{i,n}(t))\end{pmatrix}U_{i-1}
\eneq
for $t\in [t_{i-1}, (t_{i-1}+t_i)/2]$ and for $f\in C(X).$ In
particular,
\beq\label{diag-6++}
\psi(f)({t_{i-1}+t_i\over{2}})=U_{i-1}^*\begin{pmatrix} f(x_{i,
\sigma_i(1)}) & &\\
&\ddots &\\
& & f(x_{i, \sigma_i(n)})\end{pmatrix}U_{i-1}
\eneq
for $f\in C(X).$ Note that
\beq\label{diag-7}
\|\phi(f)(t_{i-1})-\psi(f)(t)\|<\dt_2/4 \andeqn
\|\psi(f)(t)-\phi(f)(t_i)\|<\dt_2/4+\ep_1/2<\dt_2/2
\eneq
for all $f\in {\cal F}$ and $t\in [t_{i-1}, {t_{i-1}+t_i\over{2}}].$
 There exists a unitary $W_i\in M_n$ such
that
\beq\label{diag-8}
W_i^*\psi(f)({t_{i-1}+t_{i}\over{2}})W_i=\phi(f)(t_i)
\eneq
for all $f\in C(X).$ It follows from (\ref{diag-7}) and
(\ref{diag-8}) that
\beq\label{diag-9}
\|W_i\psi(f)({t_{i-1}+t_i\over{2}})-\psi(f)({t_{i-1}+t_i\over{2}})W_i\|<\dt_2
\eneq
for all $f\in {\cal F}.$ By the choice of $\dt_2$ and by applying
Lemma 2.6.11 of \cite{Lnbk}, we obtain $h_i\in M_n$ such that
$W_i=\exp(\sqrt{-1}h_i)$ and
\beq\label{diag-10}
\|h_i\psi(f)({t_{i-1}+t_i\over{2}})-\psi(f)({t_{i-1}+t_i\over{2}})h_i\|<\ep/4\andeqn\\
\|\exp(\sqrt{-1}th_i)\psi(f)({t_{i-1}+t_i\over{2}})-\psi(f)({t_{i-1}+t_i\over{2}})\exp(\sqrt{-1}th_i)\|<\ep/4
\eneq
for all $f\in {\cal F}$ and $t\in [0,1].$ From this we obtain a continuous path
of unitaries $\{W_i(t): t\in [{t_{i-1}+t_i\over{2}}, t_i]\}\subset
M_n$ such that
\beq\label{diag-11}
W_i({t_{i-1}+t_i\over{2}})=1,\,\,\,
W_i(t_i)=W_i\andeqn\\\label{diag-12}
\|W_i(t)\psi(f)({t_{i-1}+t_i\over{2}})-\psi(f)({t_{i-1}+t_i\over{2}})W_i(t)\|<\ep/4
\eneq
for all $f\in {\cal F}$ and $t\in [{t_{i-1}+t_i\over{2}}, t_i].$
Define $\psi(f)(t)=W_i^*(t)\psi({t_{i-1}+t_i\over{2}})W_i(t)$ for
$t\in [{t_{i-1}+t_i\over{2}}, t_i],$ $i=1,2,...,N.$ Note that
$\psi: C(X)\to A.$ We conclude that
\beq\label{diag-12+}
\|\phi(f)-\psi(f)\|<\ep\tforal {\cal F}.
\eneq
Define
\beq\label{diag-12+1}
&&\hspace{-0.3in}U(t)= U_0\,\,\,\,{\rm for}\,\,\, t\in [0,
{t_1\over{2}}),\,\,
U(t)=U_0W_1(t)\,\,\,{\rm for}\,\,\, t\in [{t_1\over{2}}, t_2),\\
&&\hspace{-0.3in}U(t)=U(t_i)\,\,\,{\rm for}\,\,\,
t\in [t_i, {t_i+t_{i+1}\over{2}}),\,\,
 U(t)=
U(t_i)W_{i+1}(t)\,\,\,{\rm for}\,\,\, t\in [{t_i+t_{i+1}\over{2}},
t_{i+1}],
\eneq
$i=1,2,...,N-1$ and define
\beq\label{diiag-12+2}
&&\hspace{-0.7in}s_j(t)=x_{0,j}(t)\,\,\,{\rm for}\,\,\,  t\in [0, {t_1\over{2}}),\,\,
s_j(t)= s_j({t_1\over{2}})\,\,\,{\rm for}\,\,\, t\in [{t_1\over{2}}, t_2),\\
&&\hspace{-0.7in}s_j(t)= x_{i,\sigma_i(j)}(t)\,\,\,{\rm for}\,\,\, t\in [t_i,
{t_i+t_{i+1}\over{2}}),\,\, s_j(t)= s_j(
{t_i+t_{i+1}\over{2}})\,\,\,{\rm for}\,\,\, t\in
[{t_i+t_{i+1}\over{2}}, t_{i+1}],
\eneq
$i=1,2,...,N-1.$ Thus $U(t)\in A$ and, by (\ref{diag-6+-}),
$s_j(t)\in C([0,1],X).$

One then checks that $\psi$ has the form:
\beq\label{diag-13}
\psi(f)=U(t)^*\begin{pmatrix}f(s_1(t)) & &\\
                     & \ddots &\\
                      && f(s_n(t))\end{pmatrix} U(t)
                      \eneq
                      for $f\in C(X).$
In fact, for $t\in [0, t_1],$ it is clear that (\ref{diag-13})
holds. Suppose that (\ref{diag-13}) holds for $t\in [0, t_i].$
Then, by (\ref{diag-8}), for $f\in C(X),$
\beq\label{diag-14}
\hspace{-0.3in}\psi(f)(t_i)= U(t_i)^*\begin{pmatrix}  f(x_{i,\sigma_i(1)}) &&\\
                     & \ddots &\\
                      && f(x_{i, \sigma_i(n)})\end{pmatrix}U(t_i)
          = U_i^*\begin{pmatrix}  f(x_{i,1}) &&\\
                     & \ddots &\\
                      && f(x_{i,n})\end{pmatrix}U_i. \eneq
                      Therefore, for $t\in [t_i,
                      {t_i+t_{i+1}\over{2}}],$
\beq\label{diag-15}
\psi(f)(t)&=& U_i^*\begin{pmatrix}  f(x_{i,1}(t)) &&\\
                     & \ddots &\\
                      && f(x_{i,n}(t))\end{pmatrix}U_i\\
                      &=& U(t_i)^*\begin{pmatrix}
                      f(x_{i,\sigma_i(1)}(t)) &&\\
                       & \ddots &\\
                      && f(x_{i, \sigma_i(n)}(t))\end{pmatrix}U(t_i)\\
                      &=& U(t)^*\begin{pmatrix}
                      f(s_1(t)) & &\\
                       & \ddots &\\
                      && f(s_n(t))\end{pmatrix}U(t).
                      \eneq
For $t\in [{t_i+t_{i+1}\over{2}}, t_{i+1}],$
\beq\label{diag-16}
\hspace{-0.2in}\psi(f)(t)&=&W_{i+1}(t)^*\psi({t_i+t_{i+1}\over{2}})W_{i+1}(t)\\
                      &=&
                      W_{i+1}(t)^*U(t_i)^*
                      \begin{pmatrix}
                      f(s_1({t_i+t_{i+1}\over{2}})) & &\\
                       & \ddots &\\
  && f(s_n({t_i+t_{i+1}\over{2}}))
  \end{pmatrix}U(t_i)W_{i+1}(t)\\
  &=& U(t)^*\begin{pmatrix}
                      f(s_1(t)) & &\\
                       & \ddots &\\
                      && f(s_n(t))\end{pmatrix}U(t).
                      \eneq
                      This verifies (\ref{diag-13}).

\end{proof}

\begin{lem}\label{commhp}
Let $X$ be a finite CW  complex and let $A\in {\cal I}.$
Suppose that $\phi: C(X)\otimes C(\T)\to A$ is a unital \hm. Then,
for any $\ep>0$ and any finite subset ${\cal F}\subset C(X),$ there
exists a continuous path of unitaries $\{u(t):\, t\in [0,1]\}$ in
$A$ such that
\beq\label{commhp-1}
u(0)=\phi(1\otimes z),\,\,\, u(1)=1\andeqn \|[\phi(f\otimes 1),\,
u(t)]\|<\ep
\eneq
for $f\in {\cal F}$ and $t\in [0,1].$


\end{lem}

\begin{proof}
It is clear that the general case can be reduced to the case that
$A=C([0,1], M_n).$
Let $q_1, q_2,...,q_m$ be projections of $C(X)$ corresponding to each path connected component 
of $X.$ Since $\phi(q_i)A\phi(q_i)\cong C([0,1], M_{n_i})$ for some $1\le n_i\le n,$ $i=1,2,...,$ 
we may reduce the general case to the case that $X$ is path connected and 
$A=C([0,1], M_n).$ 

Note that we use $z$ for the identity function on the unit circle.

For any $\ep>0$ and any finite subset ${\cal F}\subset C(X),$  by
applying \ref{diag}, one obtains a unital \hm\, $\psi: C(X)\otimes
C(\T)\to A$ such that
\beq\label{comm-1}
&&\|\phi(g)-\psi(g)\|<\ep\tforal g\in \{ f\otimes 1: f\in {\cal
F}\}\cup \{1\otimes z\}\andeqn\\
&&\psi(f)(t)=U(t)^*\begin{pmatrix} f(s_1(t)) &&\\
                     &\ddots &\\
                      && f(s_n(t))\end{pmatrix} U(t),
\eneq
for all $f\in C(X\times \T),$ where $U(t)\in U(C([0,1], M_n)),$
$s_j: [0,1]\to X\times \T$ is a continuous map, $j=1,2,...,n,$ and
for all $t\in [0,1].$ There are continuous paths of unitaries
$\{u_j(r): r\in [0,1]\}\subset C([0,1])$ such that
\beq\label{comm-2}
u_j(0)(t)=(1\otimes z)(s_j(t)),\,\,\,u_j(1)=1, \,\,\, j=1,2,...,n.
\eneq
Define
\beq\label{comm-3}
u(r)(t)=U(t)^*\begin{pmatrix} u_j(r)(t) &&\\
     & \ddots&\\
     && u_n(r)(t)\end{pmatrix} U(t).
     \eneq
     Then
     $$
     u(r)\psi(f\otimes 1)=\psi(f\otimes 1) u(r)\tforal r\in [0,1].
     $$
     It follows that
     $$
     \|[\phi(f\otimes 1),\,u(r)]\|<\ep\rforal r\in [0,1]\andeqn
     \tforal
     f\in {\cal F}.
     $$

\end{proof}

\begin{df}\label{dfH}
{\rm Let $X$ be a compact metric space. We say  that $X$ satisfies
property (H) if the following holds:

For any  $\ep>0,$  any finite subsets ${\cal F}\subset C(X)$ and any
non-decreasing map $\Delta: (0,1)\to (0,1),$ there exists $\eta>0$ (
which depends on $\ep$ and ${\cal F}$ but not $\Delta$), $\dt>0,$ a
finite subset ${\cal G}\subset C(X)$ and a finite subset ${\cal
P}\subset \underline{K}(C(X))$ satisfying the following:

Suppose that $\phi: C(X)\to C([0,1], M_n)$ is a unital $\dt$-${\cal
G}$-multiplicative \morp\, for which
\beq\label{TAM-1}
\mu_{\tau\circ \phi}(O_a)\ge \Delta(a)
\eneq
for any open ball $O_a$ with radius $a\ge \eta$ and for all tracial
states $\tau$ of $C([0,1], M_n),$ and
\beq\label{TAM-2}
[\phi]|_{\cal P}=[\Phi]|_{\cal P},
\eneq
where $\Phi$ is a point-evaluation.

Then there exists a unital \hm\, $h: C(X)\to C([0,1],M_n)$ such that
\beq\label{ATM-3}
\|\phi(f)-h(f)\|<\ep
\eneq
for all $f\in {\cal F}.$

It is a restricted version of some relatively weakly semi-projectivity property. It
has been shown in \cite{Lnn1} that any $k$-dimensional torus has the
property (H). So do those finite CW compleces $X$ with torsion free
$K_0(C(X))$ and torsion $K_1(C(X)),$ any finite CW compleces with
form $Y\times \T$ where $Y$ is contractive and all one-dimensional
finite CW compleces.

}

\end{df}

 \begin{thm}\label{L2}
Let $X$ be a finite CW complex for which $X\times \T$ has the
property (H). Let $C=C(X)$ and let $\Delta: (0, 1)\to (0,1)$ be a
non-decreasing map. For any $\ep>0$ and any finite subset
 ${\cal F}\subset C,$ there exists  $\dt>0,$ $\eta>0$ and there exists a
 finite subset ${\cal G}\subset C$ satisfying the following:

 Suppose that $A$ is a unital  simple \CA\, with $TR(A)\le
 1,$ $\phi: C\to A$ is a unital \hm\, and $u\in A$ is a
 unitary and  suppose that
\beq\label{L2-0}
\|[\phi(c),\, u]\|<\dt\tforal c\in {\cal G} \tand {\rm Bott}(\phi,
\,u)=\{0\}.
\eneq
Suppose also that there exists a unital \morp\, $L: C\otimes
C(\T)\to A$ such that (with $z$ the identity function on the unit
circle)
\beq\label{L2-1}
&&\|L(c\otimes 1)-\phi(c)\|<\dt,\,\,\,\|L(c\otimes z)-\phi(c)u\|<\dt
\tforal c\in {\cal G}\\\label{L2-2} &&\tand \mu_{\tau\circ
L}(O_a)\ge \Delta(a) \tforal \tau\in T(A)\tand
 \eneq
for all open balls $O_a$ of $X\times \T$ with radius $1>a\ge
\eta,$
 where $\mu_{\tau\circ L}$ is the Borel probability  measure defined by
 $L.$  Then
 there exists a continuous path of unitaries $\{u(t): t\in [0,1]\}$
 in $A$
 such that
 \beq\label{L2-3}
u(0)=u,\,\,\, u(1)=1\andeqn \|[\phi(c),\, u(t)]\|<\ep
 \eneq
for all $c\in {\cal F}$ and for all $t\in [0,1].$

\end{thm}

\begin{proof}

Let $\Delta_1(a)=\Delta(a)/2.$ Denote by $z\in C(\T)$ the identity
map on the unit circle. Let $B=C\otimes C(\T)=C(X\times \T).$
Put $Y=X\times \T.$ Without loss of generality, we may assume that
${\cal F}$ is in the unit ball of $C.$ Let ${\cal F}_1=\{c\otimes
1: c\in {\cal F}\}\cup \{1\otimes z\}.$

Let $\eta_1>0$ (in place of $\eta$),   $\dt_1>0$ (in place of
$\dt$), ${\cal G}_1\subset C(Y)$ (in place of ${\cal G}$) be a
finite subset, ${\cal P}_1\subset \underline{K}(C(Y))$ (in place of
${\cal P}$) and ${\cal U}_1\subset U(C(Y))$ be as required by
Theorem 10.8 of \cite{Lnn1} corresponding to $\ep/16$ (in place of
$\ep$), ${\cal F}_1$ and $\Delta_1$ (in place  of $\Delta$) above.

Without loss of generality, we may assume that
 \beq\label{L2-n1}
 {\cal
U}_1=\{\zeta_1\otimes 1, ...,\zeta_{K_1}\otimes 1, \,\,\,p_1\otimes
z\oplus (1-p_1\otimes 1), ...,p_{K_2}\otimes z\oplus (1-p_{K_2}\otimes
1)\},
\eneq
where $\zeta_k\in U(C),$ $k=1,2,...,K_1$ and $p_j\in C$ is a
projection, $j=1,2,...,K_2.$  Denote $z_i=p_i\otimes z\oplus
(1-p_i\otimes 1),$ $i=1,2,..,K_2.$ We may also assume that ${\cal
U}_1\subset U(M_k(C(Y))).$

For any \morp\, $L'$ from $C(Y),$ we will also use $L'$ for
$L'\otimes {\rm id}_{M_k}.$

Fix a finite subset ${\cal G}_2\subset C(Y)$ which contains ${\cal
G}_1.$ Choose a small $\dt_1'>0.$ We choose ${\cal G}_2$ so large
and $\dt_1'$ so small that, for any $\dt_1'$-${\cal
G}_2$-multiplicative map $L'$ form $C(Y)$ to a unital \CA\, $B',$
there are unitaries $w_1', w_2',...,w_{K_1}'$ and $u_1',
u_2',...,u_{K_2}'$ in $M_k(B')$ such that
\beq\label{L2-n2}
\|L'(\zeta_i)-w_i'\|<\dt_1/16\andeqn \|L'(z_j)-u_j'\|<\dt_1/16,
\eneq
$i=1,2,...,K_1$ and $j=1,2,...,K_2.$

Let
 $\eta_2>0$ (in place of $\eta$), $\dt_2>0$ (in place of $\dt$), ${\cal G}_3\subset
C(Y)$ (in place of ${\cal G}$) be required by 10.7 of \cite{Lnn1}
for
 $\min\{ \dt_1/16, \dt_1'/16, \Delta_1(\eta_1)/16,\ep/16\}$ (in place of $\ep$), ${\cal
G}_1\cup {\cal F}_1$ and $\Delta_1$ (in place  of $\Delta$) above.
We may assume that ${\cal G}_3\supset {\cal G}_2\cup{\cal
G}_1\cup{\cal F}_1$
and ${\cal G}_3$ is in the unit ball of $C(Y).$ Moreover, we may
further assume that ${\cal G}_3=\{c\otimes 1: c\in {\cal
F}_2\}\cup\{z\otimes 1, 1\otimes z, 1\otimes 1 \}$ for some finite
subset  ${\cal F}_2.$ Suppose that ${\cal G}\subset A$ is a finite
subset which contains at least ${\cal F}_2.$
We may assume that $\dt_2<\dt_1.$ Let
$\dt=\min\{{\dt_1\over{16}},{\dt_1'\over{16}},
{\Delta_1(\eta)\over{4}}\}.$

Let $A$ be a unital simple \CA\, with $TR(A)\le 1,$ let $\phi:
C(X)\to A$ and $u\in U_0(A)$ be such that
\beq\label{L2-10}
\|[\phi(c), \, u]\|<\dt\tforal c\in {\cal G}
\andeqn {\rm Bott}(\phi,\,u)=\{0\}.
\eneq
We may  assume that (\ref{L2-1})  holds for $\eta= \eta_1/2.$
We also  assume that there is a $\dt$-${\cal G}_3$-multiplicative
\morp\, $L: C(Y)\to A$ such that
\beq\label{L2-11}
\|L(f\otimes 1)-\phi(f)\|<\dt,\,\,\,\|L(1\otimes
z)-u\|<\dt\andeqn\\
\mu_{\tau\circ L}(O_a)\ge 2\Delta_1(a)\tforal a\ge \eta,
\eneq
for all $\tau\in \text{T}(A)$ and for all $f\in {\cal F}_2.$ We
will continue to use $L$ for $L\otimes {\rm id}_{M_k}.$

By (\ref{L2-n2}), one may also assume that there are unitaries
$w_1,w_2,...,w_{K_1}$ and unitaries $u_1,u_2,...,u_{K_2}$ such that
\beq\label{L2-11+}
\|L(\zeta_i\otimes 1)-w_i\|<\dt_1/16\andeqn
\|L(z_j)-u_j\|<\dt_1/16,
\eneq
$i=1,2,...,K_1$ and $j=1,2,...,K_2.$

We note that, by (\ref{L2-0}),  $[u_i]=0$ in $K_1(A).$ Put
$$
H=\max\{{\rm cel}(u_i): 1\le i\le K_2\}.
$$
Let $N\ge 1$ be an integer such that
\beq\label{L2-n}
{\max\{1,\pi, H+\dt_1+\dt\}\over{N}}<\dt/4\andeqn
{1\over{N}}<\Delta_1(\eta)/4.
\eneq
For each $i,$ there are self-adjoint elements $a_{i,1},
a_{i,2},...,a_{i,L(i)}\in A$ such that
\beq\label{L2-n3}
u_i=\prod_{j=1}^{L(i)}\exp(\sqrt{-1}a_{i,j})\andeqn
\sum_{i=1}^{L(i)}\|a_{i,j}\|\le H+\dt/64,
\eneq
$i=1,2,...,K_2.$

Put $\Lambda=\max\{L(i): 1\le i\le K_2\}.$ Let $\ep_0>0$ such that
if $\|p'a-ap'\|<\ep_0$ for any self-adjoint element $a$ and
projection $p',$ $\|p'\exp(ia)-\exp(ia)p'\|<\dt_1/16\Lambda.$

By applying  Corollary 10.7 of \cite{Lnn1}, we obtain mutually
orthogonal  projections $P_0, P_1, P_2\in A$ with $P_0+P_1 +P_2=1$ and a
\SCA\, $D=\bigoplus_{j=1}^sC(X_j, M_{r(j)}),$ where $X_j=[0,1]$ or
$X_j$ is a point, with $1_D=P_1,$ a finite dimensional \SCA\,
$D_0\subset A$ with $1_{D_0}=P_2,$ a unital \morp\, $L_0: C(X)\to
D_0$ and there exists a unital \hm\, $\Phi: C(Y)\to D$ such that
\beq\label{L2-12}
\hspace{-0.4in}
\|L(g)-(P_0L(g)P_0+L_0(g)+\Phi(g))\|&<&\min\{{\dt_1\over{16}},{\dt_1'\over{16}},
{\Delta_1(\eta_1)\over{4}},{\ep\over{16}}\} \tforal g\in {\cal G}_2
\\\label{L2-12++} \andeqn (2N+1)\tau(P_0+P_2)&<&\tau(P_1)\rforal
\tau\in \text{T}(A).
\eneq
Moreover,
\beq\label{L2-12++1}
\|[x, \, P_0]\|<\min\{\ep_0,\dt_1/16,\dt_1'/16,
\Delta_1(\eta_1)/4,\ep/16\}
\eneq
for all $x\in \{: a_{i,j}1\le j\le L(i), 1\le i\le
K_2\}\cup\{L({\cal G}_2)\}.$

There exists a unitary $u_j'\in M_k(P_0AP_0)$ and $u_j''\in
M_k(D_0)$ such that
\beq\label{L2-12+}
\|u_j'-\overline{P_0}L(z_j)\overline{P_0}\|<\dt_1/16\andeqn
\|u_j''-L_0(z_j)\|<\dt_1/16
\eneq
where $\overline{P_0}={\rm diag}(\overbrace{P_0,P_0,...,P_0}^k).$ It
follows from (\ref{L2-12++1}) and (\ref{L2-n3}) that $[u_j']=0$ in
$K_1(A)$ and
\beq\label{L2-12+++}
{\rm cel}(u_j')\le H+\dt_1/16+\dt/64,\,\,\,j=1,2,...,K_2
\eneq
(in $M_k(P_0AP_0)$). Note also, since $u_j''\in M_k(D_0),$
$$
{\rm cel}(u_j)\le \pi,\,\,\,j=1,2,...,K_2.
$$

By applying \ref{commhp}, there exists a continuous path of
unitaries $\{v(t):t\in [0,1]\}\subset D$ such that
\beq\label{L2-13}
v(0)=\Phi(1\otimes z),\,\,\, v(1)=P_1\andeqn \|[\Phi(f\otimes 1),\,
v(t)]\|<\ep/4
\eneq
for all $t\in [0,1].$ Define a \morp\, $L_1: C(Y)=C(X)\otimes
C(\T)\to A$ by
\beq
L_1(f\otimes 1)&=&P_0L(f\otimes 1)P_0+ L_0(f\otimes 1)+\Phi(f\otimes
1) \andeqn\\
 L_1(1\otimes g)&=&g(1)\cdot P_0+g(1)\cdot P_2+ \Phi(1\otimes
g(z)).
\eneq
for all $f\in C(X)$ and $g\in C(\T).$ We compute
 (by choosing
large ${\cal G}_2$)
 that
\beq\label{L2-14}
\mu_{\tau\circ L_1}(O_a)&\ge &\Delta_1(a)\rforal a\ge
\eta\andeqn\\\label{L2-14+}
 |\tau\circ L_1(g)-\tau\circ
L(g)|&<&\dt\tforal g\in {\cal G}_1
\eneq
and (by the fact that ${\rm Bott}(\phi, \,u)=\{0\}$)
\beq\label{L2-15}
[L]|_{{\cal P}_1}=[L_1]|_{{\cal P}_1}.
\eneq
We also have (by (\ref{L2-11}) and (\ref{L2-12}))
\beq\label{L2-15+}
{\rm dist}(L^{\ddag}(\zeta_i\otimes 1), L_1^{\ddag}(\zeta_i\otimes
1))<\dt_1/16, \,\,\,i=1,2,...,K_1.
\eneq

Moreover, for $j=1,2,...,K_2,$
\beq\label{L2-16}
{\rm dist}(L^{\ddag}(z_j), L_1^{\ddag}(z_j)) &<& \dt_1/16+ {\rm
dist}(({\overline{u_j' +u_j''+
\Phi(z_j))^*(\overline{P_0}+\overline{P_2}+ \Phi(z_j))}},{\bar
1})\\\label{L2-17} &=& \dt_1/16+{\rm dist}({\overline{u_j'+u_j''+
\overline{P_1})^*}}, {\bar 1})\\\label{L2-17+}
&<&\dt_1/16 +{\max\{\pi,H+\dt_1/16+\dt\}/64\over{N}}\\
&<&\dt_1/16+\dt/2<\dt_1,
\eneq
where the third inequality follows from (\ref{L2-12+++}) and
\ref{length}.

From (\ref{L2-14}), (\ref{L2-14+}), (\ref{L2-15}),
and (\ref{L2-17+}), by applying Theorem 10.8 of \cite{Lnn1},  one
obtains a unitary $W\in A$ such that
\beq\label{L2-18}
\|{\rm ad}\, W\circ L_1(g)-L(g)\|<\ep/16\tforal g\in \{ c\otimes
1: c\in {\cal  F}\otimes 1\}\cup\{1\otimes z\}.
\eneq
Define
\beq\label{L2-19}
u'(t)=W^*(P_0\oplus v(t))W\,\,\,t\in [0,1].
\eneq
Then $u'(0)=W^*(P_0\oplus \Phi(1\otimes z))W$ and $u'(1)=1.$  It
follows from (\ref{L2-13}) and (\ref{L2-18}) that
\beq\label{L2-20}
\|[\phi(c), \,u'(t)]\|<\ep/2\tforal c\in {\cal F}
\eneq
and for $t\in [0,1].$ Note that
\beq\label{L2-21}
\|u'(0)-u\|<\ep/8.
\eneq
One then obtains a continuous path $\{u(t): t\in [0,1]\}\subset A$
by connecting $u'(0)$ with $u$ by a path with length no more than
$\ep/2.$  The theorem follows.

\end{proof}


 \begin{cor}\label{L1}
 Let $C=C(X, M_n)$ where $X=[0,1]$ or $X=\T$ and $\Delta: (0,1)\to (0,1)$  be
 a non-decreasing map. For any $\ep>0$ and any finite subset
 ${\cal F}\subset C,$ there exist  $\dt>0,$ $\eta>0$ and there exists a
 finite subset ${\cal G}\subset C$ satisfying the following:

 Suppose that $A$ is a unital  simple \CA\, with $TR(A)\le
 1,$ $\phi: C\to A$ is a unital monomorphism and $u\in A$ is a
 unitary and  suppose that
\beq\label{L1-0}
\|[\phi(c),\, u]\|<\dt\tforal c\in {\cal G},\\
{\rm bott}_0(\phi, u)=\{0\}\tand {\rm bott}_1(\phi, u)=\{0\}.
\eneq
Suppose also that there exists a unital \morp\, $L: C\otimes
C(\T)\to A$ such that (with $z$ the identity function on the unit
circle)
\beq\label{L1-1}
\hspace{-0.3in}\|L(c\otimes 1)-\phi(c)\|<\dt,\,\,\,\|L(c\otimes
z)-\phi(c)u\|<\dt\tforal c\in {\cal G}\andeqn
 \mu_{\tau\circ L}(O_a)\ge \Delta(a)
 \eneq
 for all open balls $O_a$ of $[0,1]\times \T$ with radius $1>a\ge \eta,$
 where $\mu_{\tau\circ L}$ is the Borel probability  measure defined by restricting
 $L$ on the center of $C\otimes C(\T).$ Then
 there exists a continuous path of unitaries $\{u(t): t\in [0,1]\}$
 such that
 \beq\label{CL1-3}
u(0)=u,\,\,\, u(1)=1\andeqn \|[\phi(c),\, u(t)]\|<\ep
 \eneq
for all $c\in {\cal F}$ and for all $t\in [0,1].$

\end{cor}

\begin{cor}\label{CL1}
Let $C=C([0,1], M_n)$ and let $T=N\times K: (C\otimes
C(\T))_+\setminus\{0\}\to \N\times \R_+\setminus\{0\}$  be
 a map. For any $\ep>0$ and any finite subset
 ${\cal F}\subset C,$ there exists  $\dt>0,$ a finite subset ${\cal H}\subset (C\otimes C(\T))_+\setminus\{0\}$
 and there exists a
 finite subset ${\cal G}\subset C$ satisfying the following:

 Suppose that $A$ is a unital  simple \CA\, with $TR(A)\le
 1,$ $\phi: C\to A$ is a unital monomorphism and $u\in A$ is a
 unitary and suppose that
\beq\label{CL1-0}
\|[\phi(c),\, u]\|&<&\dt\tforal c\in {\cal G}\andeqn \\
{\rm bott}_0(\phi, u)&=&\{0\}.
\eneq
Suppose also that there exists a unital \morp\, $L: C\otimes
C(\T)\to A$ which is $T$-${\cal H}$-full such that (with $z$ the
identity function on the unit circle)
\beq\label{CL1-1}
\|L(c\otimes 1)-\phi(c)\|<\dt \tand \|L(c\otimes
z)-\phi(c)u\|<\dt\tforal c\in {\cal G}.
 \eneq
 Then
 there exists a continuous path of unitaries $\{u(t): t\in [0,1]\}$
 in $A$
 such that
 \beq\label{CCL1-3}
u(0)=u,\,\,\, u(1)=1\tand \|[\phi(c),\, u(t)]\|<\ep
 \eneq
for all $c\in {\cal F}$ and for all $t\in [0,1].$

\end{cor}

\begin{proof}
Fix $T=N\times K: \N\times \R_+\setminus \{0\}.$ Let $\Delta:
(0,1)\to (0,1)$ be the non-decreasing map associated with $T$ as in
Proposition 11.2 of \cite{Lnn1}. Let ${\cal G}\subset C,$ $\dt>0$
and $\eta>0$ be as required by \ref{L1} for $\ep$ and ${\cal F}$
given and the above $\Delta.$

It follows from 11.2 of \cite{Lnn1} that there exists a finite
subset ${\cal H}\subset (C\otimes C(\T))_+\setminus \{0\}$ such that
for any unital \morp\, $L: C\otimes C(\T)\to A$ which is $T$-${\cal
H}$-full, one has that
\beq\label{pCL1-1}
\mu_{\tau\circ L}(O_a)\ge \Delta(a)
\eneq
for all open balls $O_a$ of $X\times \T$ with radius $a\ge \eta.$

The corollary then follows immediately from \ref{L1}.

\end{proof}

The following is an easy but  known fact.

\begin{lem}\label{point}
Let $C=M_n.$ Then, for any $\ep>0$ and any finite subset ${\cal F},$
there exists $\dt>0$ and a finite subset ${\cal G}\subset C$
satisfying the following: For any unital \CA\, $A$ with
$K_1(A)=U(A)/U_0(A)$ and any unital \hm\, $\phi: C\to A$ and any
unitary $u\in A$ if
\beq\label{point-1}
\|[\phi(c),\, u]\|<\dt\tand {\rm bott}_0(\phi,u)=\{0\},
\eneq
then there exists a continuous path of unitaries $\{u(t):t\in
[0,1]\}\subset A$ such that
\beq\label{pt-2}
u(0)=u,\,\,\, u(1)=1\tand \|[\phi(c),\, u(t)]\|<\ep
\eneq
for all $c\in {\cal F}$ and $t\in [0,1].$
\end{lem}

\begin{proof}
First consider the case that $\phi(c)$ commutes with $u$ for all
$c\in M_n.$ Then one has a unital \hm\, $\Phi:  M_n\otimes C(\T)\to
A$ defined by $\Phi(c\otimes g)=\phi(c)g(u)$ for all $c\in C$ and
$g\in C(\T).$ Let $\{e_{i,j}\}$ be a matrix unit for $M_n.$ Let
$u_j=e_{j,j}\otimes z,$ $j=1,2,...,n.$ The assumption ${\rm
bott}_0(\phi, u)=\{0\}$ implies that $\Phi_{*1}=\{0\}.$ It
follows that $u_j\in U_0(A),$ $j=1,2,...,n.$ One then obtains a
continuous path of unitaries $\{u(t): t\in [0,1]\}\subset A$ such
that
$$
u(0)=u,\,\,\,u(1)=1\andeqn \|[\phi(c), u(t)]\|=0
$$
for all $c\in C(\T)$ and $t\in [0,1].$

The general case follows from the fact that $C\otimes C(\T)$ is
weakly semi-projective.

\end{proof}

\begin{rem}
{\rm Let $X$ be a compact metric space and let $A$ be a unital
simple \CA. Suppose that $\phi: C(X)\to A$ is a unital injective
completely positive linear map. Then it is easy to check (see 7.2 of
\cite{Lnhomp}, for example ) that there exists a non-decreasing map
$\Delta: (0,1)\to (0,1)$ such that
$$
\mu_{\tau\circ \phi}(O_a)\ge \Delta(a)
$$
for all $a\in (0,1)$ and for all $\tau\in \text{T}(A).$

}
\end{rem}

\section{Changing spectrum}

\begin{lem}\label{cmatr}
Let $n\ge 64$  be an integer. Let
 $\ep>0$ and $1/2>\ep_1>0.$
There exists ${\ep\over{2n}}>\dt>0$  and a finite subset ${\cal
G}\subset D\cong M_n$ satisfying the following:

Suppose that $A$ is a unital \CA\, with $\text{T}(A)\not=\emptyset,$
$D\subset A$ is a \SCA\, with $1_D=1_A,$ suppose that ${\cal
F}\subset A$ is a finite subset  and suppose that $u\in U(A)$ such
that
\beq\label{cmatr-0}
\|[f,\,x]\|
&<&\dt \tforal f\in {\cal F}\tand x\in {\cal G},\tand\\
\|[u, \, x]\|&<&\dt\tforal x\in {\cal G}.
\eneq
 Then, there exists a
unitary $v\in D$  and a continuous path of unitaries $\{w(t): t\in
[0,1]\}\subset D$ such that
\beq\label{cmatr-2}
&&\|[u,\, w(t)]\|<n\dt<\ep,\,\,\,\|[f,\,w(t)]\|<n\dt<\ep/2\\
&&\tforal f\in {\cal F}\tand \tforal t\in [0,1],\\\label{cmatr-3}
&&w(0)=1,\,\,\, w(1)=v\tand
\mu_{\tau\circ \imath}(I_a)\ge {2\over{3n^2}}
\eneq
for all open arcs  $I_a$ of $\T$ with length  $a\ge  4\pi/n$ and for
all $\tau\in \text{T}(A),$ where $\imath: C(\T)\to A$ is defined by
$\imath(f)=f(vu)$ for all $f\in C(\T).$

Moreover,
\beq\label{cmatr-n2}
{\rm length}(\{w(t)\})\le \pi.
\eneq

If, in addition, $\pi>b_1>b_2>\cdots >b_m>0$ and $1=d_0>d_1> d_2\ge
\cdots
> d_m>0$ are given so that
\beq\label{cmatr-n3}
\mu_{\tau\circ \imath_0}(I_{b_i})\ge d_i\tforal \tau\in
\text{T}(A),\,\,\,i=1,2,...,m,
\eneq
where $\imath_0: C(\T)\to A$ is defined by $\imath_0(f)=f(u)$ for
all $f\in C(\T),$ then one also has that
\beq\label{cmatr-n4}
\mu_{\tau\circ \imath}(I_{c_i})\ge (1-\ep_1) d_i\tforal \tau\in
\text{T}(A),
\eneq
where  $I_{b_i}$ and $I_{c_i}$ are any open arcs with length $b_i$ and $c_i,$ respectively, and where $c_i=b_i+\ep_1,$ $i=1,2,...,m.$

\end{lem}

\begin{proof}

Let
$$
0<\dt_0<\min\{{\ep_1d_i\over{16n^2}}: 1\le i\le m\}.
$$

Let $\{e_{i,j}\}$ be a matrix unit for $D$ and let ${\cal
G}=\{e_{i,j}\}.$ Define
\beq\label{cmatr-01}
 v=\sum_{j=1}^n
e^{2\sqrt{-1}j\pi/n}e_{j,j}.
\eneq

Let $f_1\in C(\T)$ with $f_1(t)=1$ for
$|t-e^{2\sqrt{-1}\pi/n}|<\pi/n$ and $f_1(t)=0$ if
$|t-e^{2\sqrt{-1}\pi/n}|\ge 2\pi/n$ and $1\ge f_1(t)\ge 0.$
Define $f_{j+1}(t)=f_1(e^{2\sqrt{-1}j\pi/n}t),$ $j=1,2,...,n-1.$
Note that
\beq\label{cmatr-02}
f_i(e^{2\sqrt{-1}j\pi/n}t)=f_{i+j}(t)\tforal t\in \T
\eneq
where $i, j\in \Z/n\Z.$


Fix a finite subset ${\cal F}_0\subset C(\T)_+$ which contains
$f_i,$ $i=1,2,...,n.$

Choose  $\dt$ so small  that the following hold:

\begin{enumerate}
\item there exists a unitary $u_i\in e_{i,i}Ae_{i,i}$ such that
$\|e^{2\sqrt{-1}i\pi/n}e_{i,i}ue_{i,i}-u_i\|< \dt_0^2/16n^2$,
$i=1,2,...,n;$

\item  $\|e_{i,j}f(u)-f(u)e_{i,j}\|<\dt_0^2/16n^2$ for all $f\in
{\cal F}_0,$

\item
$\|e_{i,i}f(vu)-e_{i,i}f(e^{2\sqrt{-1}i\pi/n}u)\|<\dt_0^2/16n^2$
for all $f\in {\cal F}_0$  and

\item $\|e_{i,j}^*f(u)e_{i,j}-e_{j,j}f(u)e_{j,j}\|<\dt_0^2/16n^2$
for all $f\in {\cal F}_0.$

\end{enumerate}

Fix  $k.$
 For each $\tau\in \text{T}(A),$ by (2), (3) and (4) above, there is at least one $i$ such
that
\beq\label{cmatr-04}
\tau(e_{j,j}f_i(u))\ge 1/n^2-\dt_0^2/16n^2.
\eneq

Choose $j$ so that $k+j=i\hspace{-0.05in}\mod (n).$ Then,
\beq\label{cmatr-05}
\tau(f_k(vu))&\ge & \tau(e_{j,j}f_k(vu))\\
&\ge &\tau(e_{j,j}f_k(e^{2\sqrt{-1}j\pi/n}u)))-{\dt_0^2\over{16n^2}}\\
&= &\tau(e_{j,j}f_{i}(u))-{\dt_0^2\over{16n^2}}
\ge  {1\over{n^2}}-{2\dt_0^2\over{16n^2}}.
\eneq

It follows that
\beq\label{cmatr-06}
\mu_{\tau\circ \imath}(B(e^{2\sqrt{-1} k\pi/n}, \pi/n)) \ge
{1\over{n^2}}-{2\dt_0^2\over{16n^2}}\tforal \tau\in \text{T}(A)
\eneq
and for  $k=1,2,...,n.$

It is then  easy to compute that
\beq\label{cmatr-07}
\mu_{\tau\circ \imath}(I_a)\ge 2/3n^2\rforal \tau\in \text{T}(A)
\eneq
and for any open arc with length $a\ge 2(2\pi/n)=4\pi/n.$

Note that if $\|[x,\, e_{i,i}]\|<\dt,$ then
$$
\|[x, \sum_{i=1}^n\lambda_ie_{i,i}]\|<n\dt<\ep/2 \andeqn \|[u,
\sum_{i=1}^n \lambda_ie_{i,i}\|<n\dt<\ep/2
$$
for any $\lambda_i\in \T.$
 Thus, one obtains a continuous path $\{w(t): t\in
[0,1]\}\subset D$ with ${\rm length}(\{w(t)\})\le \pi$ and with
$w(0)=1$ and $w(1)=v$ so that (\ref{cmatr-2}) holds.

Let $\{x_1,x_2,...,x_K\}$ be an $\ep_1/64$-dense set of $\T.$ Let
$I_{i,j}$ be an open arc with center $x_j$ and length $b_i,$
$j=1,2,...,K$ and $i=1,2,...,m.$ For each $j$ and $i,$ there is a
positive function $g_{j,i}\in C(\T)_+$ with $0\le g_{j,i}\le 1$
and $g_{j,i}(t)=1$ if $|t-x_j|<d_i$ and $g_{j,i}(t)=0$ if
$|t-x_j|\ge d_i+\ep_1/64,$ $j=1,2,...,K,$ $i=1,2,...,m.$ Put
$g_{i,j,k}(t)=g_{j,i}(e^{2\sqrt{-1}k\pi/n}\cdot t)$ for all $t\in
\T,$ $k=1,2,...,n.$ Suppose that ${\cal F}_0$ contains all
$g_{j,i}$ and $g_{j,i,k}.$  We have, by (2), (3) and (4) above,
\beq\label{cmatr-08}
\tau(g_{j,i}(u)e_{l,l}),\,\,\tau(g_{j,i,k}(u)e_{l,l})\ge
{d_i\over{n}}-{\dt^2/16n^2}\tforal \tau\in \text{T}(A),
\eneq
$l=1,2,...,n,$ $j=1,2,...,K$ and $i=1,2,...,m.$ Thus
\beq\label{cmatr-09}
\tau(e_{k,k}g_{j,i}(vu))&\ge &
\tau(e_{k,k}g_{j,i}(e^{2\sqrt{-1}k\pi/n}u)-n{\dt_0^2\over{16n^2}}\\
&\ge & {d_i\over{n}}-{\dt_0^2\over{8n}}\tforal \tau\in \text{T}(A),
\eneq
$k=1,2,...,n,$ $j=1,2,...,K$ and $i=1,2,...,m.$ Therefore
\beq\label{cmatr-010}
\tau(g_{j,i}(vu))\ge d_i-{\dt_0^2\over{8n}}\ge (1-\ep_1)d_i\tforal
\tau\in \text{T}(A),
\eneq
$j=1,2,...,K$ and $i=1,2,...,m.$

It follows that
\beq\label{cmatr-011}
\mu_{\tau\circ \imath}(I_{i,j})\ge (1-\ep_1)d_i\tforal \tau\in
\text{T}(A),
\eneq
$j=1,2,...,K$ and $i=1,2,...,m.$ Since $\{x_1,x_2,...,x_K\}$ is
$\ep_1/64$-dense in $\T,$ it follows that
\beq\label{cmatr-012}
\mu_{\tau\circ \imath}(I_{c_i})\ge (1-\ep_1)d_i\tforal \tau\in
\text{T}(A),\,\,\,i=1,2,...,m.
\eneq

\end{proof}

\begin{rem}\label{Rcmatr}
{\rm If the assumption that $\|[f,\, x]\|<\dt$ for all $f\in {\cal
F}$ and for all $x\in {\cal G}$ is replaced by for all $x\in D$ with
$\|x\|\le 1,$ then the conclusion can also be strengthened to $\|[f,
\, w(t)]\|<\dt$ for all $f\in {\cal F}$ and $t\in [0,1].$

}
\end{rem}

 The proof of the following is similar to that of \ref{cmatr}.

\begin{lem}\label{matrix}
Let $n\ge 64 $ be an integer. Let $\ep>0$ and $1/2>\ep_1>0.$ There
exists ${\ep\over{2n}}>\dt>0$
and a finite subset ${\cal G}\subset D\cong M_n$ satisfying the
following:

Suppose that $X$ is a compact metric space,  ${\cal F}\subset C(X)$
is a finite subset and $1>b>0.$  Then there exists a finite subset
${\cal F}_1\subset C(X)$ satisfying the following:

Suppose that $A$ is a unital \CA\, with $\text{T}(A)\not=\emptyset,$
$D\subset A$ is a \SCA\, with $1_D=1_A,$ $\phi: C(X)\to A$ is a
unital \hm\, and suppose that $u \in U(A)$  such that
\beq\label{matr-0}
\|[x,\,u]\|<\dt \tand \|[x,\,\phi(f) ]\|<\dt \tforal  x\in {\cal
G} \tand f\in {\cal F}_1.
\eneq
Suppose also that, for some $\sigma>0,$
\beq\label{matr-1}
\tau(\phi(f))\ge \sigma \tforal \tau\in \text{T}(A)\tand
\eneq
for all $f\in C(X)$ with $0\le f\le 1$ whose support contains an
open ball of $X$ with radius $b.$
 Then, there exists a unitary $v\in D$ and  a
continuous path of unitaries $\{v(t): t\in [0,1]\}\subset D$ such
that
\beq\label{matr-2}
&&\|[u, \, v(t)]\|<n\dt<\ep,\,\,\,\|[\phi(f),\, v(t) ]\|<n\dt<\ep\\
&&\,\,\, \tforal f\in {\cal F}\tand t\in [0,1],\\
&&v(0)=1,\,\,\, v(1)=v\andeqn\\
&&\tau(\phi(f)g(vu))\ge {2\sigma \over{3n^2}}\tforal \tau\in \text{T}(A)
\eneq
for any pair of $f\in C(X)$ with $0\le f\le 1$ whose support
contains an open ball with radius $2b$ and $g\in C(\T)$ with $0\le
g\le 1$ whose support contains an  open arc of $\T$ with length at
least $8\pi/n.$

Moreover,
\beq\label{matr-n2}
{\rm length}(\{v(t)\})\le \pi.
\eneq

If, in addition, $1>b_1>b_2>\cdots > b_k>0,$ $1>d_1\ge d_2\ge \cdots
\ge d_k>0$ are  given and
\beq\label{Matr-1}
\tau(\phi(f')g'(u))\ge d_i\tforal \tau\in \text{T}(A)
\eneq
for any functions $f'\in C(X)$ with $0\le f'\le 1$ whose support
contains an open ball of $X$ with radius $b_i/2$ and $g'\in C(\T)$
with $0\le g'\le 1$ whose support contains an arc   with length
$b_i,$ then one also has that
\beq\label{Matr-2}
\tau(\phi(f'')g''(vu))
\ge (1-\ep_1) d_i\tforal \tau\in \text{T}(A),
\eneq
where  $f''\in C(X)$ with $0\le f''\le 1$ whose support contains
an open ball of radius $c_i$ and $g''\in C(\T)$ with $0\le g''\le
1$ whose support contains an arc with length $2c_i$ with
$c_i=b_i+\ep_1,$ $i=1,2,...,k.$
\end{lem}

\begin{proof}
Let $0<\dt_0=\min\{{\ep_1d_i\over{16n^2}}: i=1,2,...,k\}.$

Let $\{e_{i,j}\}$ be a matrix unit for $D$ and let ${\cal
G}=\{e_{i,j}\}.$ Define
\beq\label{pmatr-2}
 v=\sum_{j=1}^n
e^{2\sqrt{-1}j\pi/n}e_{j,j}.
\eneq

Let $g_j\in C(\T)$ with $g_j(t)=1$ for
$|t-e^{2\sqrt{-1}j\pi/n}|<\pi/n$ and $g_j(t)=0$ if
$|t-e^{2\sqrt{-1}j\pi/n}|\ge 2\pi/n$ and $1\ge g_j(t)\ge 0,$
$j=1,2,...,n.$   As in the proof \ref{cmatr}, we may also assume
that
\beq\label{matr-3}
g_i(e^{2\sqrt{-1}j\pi/n}t)=g_{i+j}(t)\tforal t\in \T
\eneq
where $i, j\in \Z/n\Z.$

Let $\{x_1, x_2,...,x_m\}$ be a $b/2$-dense subset of $X.$ Define
$f_i\in C(X)$ with $f_i(x)=1$ for $x\in B(x_i, b)$ and $f_i(x)=0$
if $x\not\in B(x_i, 2b)$ and $0\le f_i\le 1,$ $i=1,2,...,m.$

Note that
\beq\label{matr-4-1}
\tau(\phi(f_i))\ge \sigma\tforal \tau\in \text{T}(A),\,\,\,i=1,2,...,m.
\eneq

Fix a finite subset ${\cal F}_0\subset C(\T)$ which at least
contains $\{g_1,g_2,..., g_n\}$ and ${\cal F}_1\subset C(X)$ which
at least contains ${\cal F}$ and $\{f_1,f_2,...,f_m\}.$

Choose  $\dt$ so small that the following hold:

\begin{enumerate}
\item there exists a unitary $u_i\in e_{i,i}Ae_{i,i}$ such that
$\|e^{2\sqrt{-1}i\pi/n}e_{i,i}ue_{i,i}-u_i\|< \dt_0^2/16n^4$,
$i=1,2,...,n;$

\item  $\|e_{i,j}g(u)-g(u)e_{i,j}\|<\dt_0^2/16n^4,$
$\|e_{i,j}\phi(f)-\phi(f)e_{i,j}\|<\dt_0^2/16n^4,$ for $f\in {\cal
F}_1$ and $g\in {\cal F}_0,$ $j,\,k=1,2,...,n$ and $s=1,2,...,m;$

\item
$\|e_{i,i}g(vu)-e_{i,i}g(e^{2\sqrt{-1}i\pi/n}u)\|<\dt_0^2/16n^4$
for all $g\in {\cal F}_0$ and

\item $\|e_{i,j}^*g(u)e_{i,j}-e_{j,j}g(u)e_{j,j}\|<\dt_0^2/16n^4,$
$\|e_{i,j}^*\phi(f)e_{i,j}-e_{j,j}\phi(f)e_{j,j}\|<\dt_0^2/16n^4$
for all $f\in {\cal F}_1$ and $g\in {\cal F}_0,$  $j,k=1,2,...,n$
and $s=1,2,...,m.$

\end{enumerate}
It follows from (4) that, for any $k_0\in \{1,2,...,m\},$
\beq\label{matr-4n1}
\tau(\phi(f_{k_0})e_{j,j})\ge \sigma/n-n\dt_0^2/16n^4.
\eneq

Fix $k_0$ and $k.$
 For each $\tau\in \text{T}(A),$ there is at least one $i$ such
that
\beq\label{matr-4n2}
\tau(\phi(f_{k_0})e_{j,j}g_i(u))\ge \sigma/n^2-\dt_0^2/16n^4.
\eneq

Choose $j$ so that $k+j=i\hspace{-0.05in}\mod (n).$ Then,
\beq\label{matr-4}
\tau(\phi(f_{k_0})g_k(vu))&\ge &
\tau(\phi(f_{k_0})e_{j,j}g_k(e^{2\sqrt{-1}j\pi/n}u)))-{\dt_0^2\over{16n^4}}\\
&= &\tau(\phi(f_{k_0})e_{j,j}g_{i}(u))-{\dt_0^2\over{16n^4}}
\\
&\ge & {\sigma\over{n^2}}-{2\dt_0^2\over{16n^4}}\tforal \tau\in
\text{T}(A).
\eneq



It is then  easy to compute that
\beq\label{matr-6}
\tau(\phi(f)g(vu))\ge {2\sigma \over{3n^2}}\rforal \tau\in
\text{T}(A)
\eneq
and for any  pair of $f\in C(X)$ with $0\le f\le 1$ whose support
contains an open ball with radius $2b$ and $g\in C(\T)$ with $0\le
g\le 1$ whose support contains an open arc of length at least
$8\pi/n.$

Note that if $\|[\phi(f),\, e_{i,i}]\|<\dt,$ then
$$
\|[\phi(f), \sum_{i=1}^n\lambda_ie_{i,i}]\|<n\dt<\ep
$$
for any $\lambda_i\in \T$ and $f\in {\cal F}_1.$  We then also
require that $\dt<\ep/2n.$ Thus, one obtains a continuous path
$\{v(t): t\in [0,1]\}\subset D$ with ${\rm length}(\{v(t)\})\le \pi$
and with $v(0)=1$ and $v(1)=v$ so that the second part of
(\ref{matr-2}) holds.

Now we consider the last part of the lemma.  Note also that, if
$f\in {\cal F}_1$ and $g\in {\cal F}_0$ with $0\le f, g\le 1,$
\beq\label{Matr-nn1}
\tau(\phi(f)g(vu))&\ge & \sum_{j=1}^n \tau(\phi(f)e_{j,j}g(vu))-{\dt_0^2\over{16n^3}}\\
&\ge &\sum_{j=1}^n
\tau(\phi(f)e_{j,j}g^{(j)}(vu))-{\dt_0^2\over{16n^2}}\tforal
\tau\in \text{T}(A),
\eneq
where $g^{(j)}(t)=g(e^{2\sqrt{-1} j\pi/n} \cdot t)$ for $t\in \T.$
If the support of $f$ contains an open ball with radius $b_i/2$
and that of $g$ contains open arcs with length at least $b_i,$ so
does that of $g^{(j)}.$ So, if ${\cal F}_0$ and ${\cal F}_1$ are
sufficiently large, by the assumptions of the last part of the
lemma, as in the proof of \ref{cmatr}, we have
\beq\label{Matr-nn2}
\tau(\phi(f)g(vu))\ge d_i-{\dt_0^2\over{16n^2}}\tforal \tau\in
\text{T}(A)
\eneq
for all $\tau\in \text{T}(A).$ As in the proof of \ref{cmatr}, this lemma
follows when we choose ${\cal F}_0$ and ${\cal F}_1$ large enough to
begin with.

\end{proof}

\begin{lem}\label{Div}
Let $C$ be a unital separable simple \CA\, with $TR(C)\le 1$  and
let $n\ge 1$ be an integer. For any $\ep>0,$ $\eta>0,$ any finite
subset ${\cal F}\subset C,$ there exists $\dt>0,$ a projection $p\in
A$ and a \SCA\, $D\cong M_n$  with $1_D=p$ such that
\beq\label{Div-1}
\|[x, p]\|&<&\ep\tforal x\in {\cal F};\\
\|[pxp,y]\|&<& \ep \tforal x\in {\cal F}\tand y\in D\,\,\,{\rm
with}\,\,\,
\|y\|\le 1 \tand\\
\tau(1-p)&<&\eta\tforal \tau\in T(C).
\eneq
\end{lem}

\begin{proof}
Choose an integer $N\ge 1$ such that $1/N<\eta/2n$ and $N\ge 2n.$ It
follows from (the proof of ) Theorem 5.4 of \cite{Lnctr1} that there
is a projection $q\in C$ and there exists a \SCA\, $B$ of $C$ with
$1_B=q$ and $B\cong \oplus_{i=1}^LM_{K_i}$ with  $K_i\ge N$ such
that
\beq\label{Div-2}
&&\|[x, q]\|<\eta/4\tforal x\in {\cal F};\\
&&\|[qxq,\, y]\|<\ep/4\tforal x\in {\cal F}\andeqn y\in B\,\,\,{\rm with}\,\,\,\|y\|\le 1\andeqn\\
&&\tau(1-q)<\eta/2n\tforal \tau\in T(C).
\eneq
Write $K_i=k_in+r_i$ with $k_i\ge 1$ and $0\le r_i<n$ for some
integers $k_i$ and $r_i,$ $i=1,2,...,L.$ Let $p\in B$ be a
projection such that the rank of $p$ is $k_i$ in each summand
$M_{K_i}$ of $B.$ Take $D_1=pBp.$ We have
\beq\label{Div-3}
&&\|[x, p]\|<\ep/2\tforal x\in {\cal F};\\
&&\|[pxp,\, y]\|<\ep\tforal x\in {\cal F}, y\in D_1\,\,\,{\rm with}\,\,\,\|y\|\le 1\andeqn\\
&&\tau(1-p)<\eta/2n+n/N<\eta/2n+\eta/2<\eta\tforal \tau\in T(C).
\eneq

Note that there is a unital \SCA\, $D\subset D_1$ such that $D\cong
M_n.$

\end{proof}

\begin{lem}\label{Umatr}
Let $n\ge 1$  be an integer with $n\ge 64.$ Let
 $\ep>0$ and $1/2>\ep_1>0.$
Suppose that $A$ is a unital simple \CA\, with $TR(A)\le 1,$
 suppose that ${\cal
F}\subset A$ is a finite subset  and suppose that $u\in U(A).$
 Then, for any $\ep>0,$ there exists a
unitary $v\in A$  and a continuous path of unitaries $\{w(t): t\in
[0,1]\}\subset A$ such that
\beq\label{Umatr-2}
&&\|[x,\,w(t)]\|<\ep\tforal f\in {\cal F}\tand \tforal t\in
[0,1],\\\label{Umatr-3}
&&w(0)=1,\,\,\, w(1)=v\tand\\
&&\mu_{\tau\circ \imath}(I_a)\ge {15\over{24n^2}}
\eneq
for all open arcs  $I_a$ of $\T$ with length  $a\ge  4\pi/n$ and for
all $\tau\in \text{T}(A),$ where $\imath: C(\T)\to A$ is defined by
$\imath(f)=f(vu).$
Moreover,
\beq\label{Umatr-n2}
{\rm length}(\{w(t)\})\le \pi.
\eneq

If, in addition, $\pi>b_1>b_2>\cdots >b_m>0$ and $1=d_0>d_1> d_2\ge
\cdots
> d_m>0$ are given so that
\beq\label{Umatr-n3}
\mu_{\tau\circ \imath_0}(I_{b_i})\ge d_i\tforal \tau\in
\text{T}(A),\,\,\,i=1,2,...,m,
\eneq
where $\imath_0: C(\T)\to A$ is defined by $\imath_0(f)=f(u)$ for
all $f\in C(\T),$ then one also has that
\beq\label{Umatr-n4}
\mu_{\tau\circ \imath}(I_{c_i})\ge (1-\ep_1) d_i\tforal \tau\in
\text{T}(A),
\eneq
where $I_{b_i}$ and $I_{c_i}$ are any open arcs with length $b_i$ and $c_i,$ respectively, and where $c_i=b_i+\ep_1,$ $i=1,2,...,m.$

\end{lem}

\begin{proof}
Let $\ep>0,$
 and let $n\ge 64$ be an integer. Put $\ep_2=\min\{\ep_1/16, 1/64n^2\}.$
Let ${\cal F}\subset A$ be a finite subset and let $u\in U(A).$
Let $\dt_1>0$ (in place of $\dt$) be as in \ref{cmatr} for $\ep,$
$\ep_2$ (in place of $\ep_1$) and let ${\cal G}=\{e_{i,j}\}\subset
D\cong M_n$ be as required by \ref{cmatr}.

Put $\dt=\dt_1/16.$
By applying \ref{Div}, there is a projection $p\in A$ and a \SCA\,
$D\cong M_n$  with $1_D=p$ such that
\beq\label{pDiv-1}
\|[x, p]\|&<&\dt\tforal x\in {\cal F};\\
\|[pxp,y]\|&<& \dt \tforal x\in {\cal F}\andeqn y\in D\,\,\,{\rm
with}\,\,\,
\|y\|\le 1 \andeqn\\
\tau(1-p)&<&\ep_2\tforal \tau\in T(C).
\eneq
There is a unitary $u_0\in (1-p)A(1-p)$ and a unitary $u_1\in
pAp.$ Put $A_1=pAp$ and ${\cal F}_1=\{pxp: x\in {\cal F}\}.$  We
apply \ref{cmatr} to $A_1,$ ${\cal F}_1$ and $u_1.$ We check that
lemma follows.

\end{proof}

The proof of the following lemma follows the same argument using
\ref{Div} as in that of \ref{Umatr} but one applies \ref{matrix}
instead of \ref{cmatr}.

\begin{lem}\label{Nuvmatrix}
Let $n\ge 64 $ be an integer. Let $\ep>0$ and $1/2>\ep_1>0.$
Suppose that $A$ is a unital simple \CA\, with $TR(A)\le 1,$ $X$
is a compact metric space, $\phi: C(X)\to A$ is a unital \hm,
${\cal F}\subset C(X)$ is a finite subset
and suppose that $u\in U(A).$
Suppose also that, for some $\sigma>0$ and $1>b>0,$
\beq\label{Nuvmatr-1}
\tau(\phi(f))\ge \sigma \tforal \tau\in \text{T}(A)\tand
\eneq
for all $f\in C(\T)$ with $0\le f\le 1$ whose supports contain an
open ball with radius at least  $b.$
 Then, there exists a unitary $v\in A$ and  a
continuous path of unitaries $\{v(t): t\in [0,1]\}\subset A$ such
that $v(0)=1,$ $v(1)=v,$
\beq\label{Nuvmatr-2}
&&\|[\phi(f),\, v(t) ]\|<\ep \tand \|[u,\, v(t)]\|<\ep
 \tforal f\in {\cal F}\tand t\in [0,1],\\
&&\tau(\phi(f)g(vu))\ge {15\sigma \over{24n^2}}\tforal \tau\in \text{T}(A)
\eneq
for any $f\in C(X)$ with $0\le f\le 1$ whose support contains an
open ball of radius at least $2b$ and any $g\in C(\T)$ with $0\le
g\le 1$ whose support contains an open arc of $\T$ with length $a\ge
8\pi/n.$

Moreover,
\beq\label{Nuvmatr-n2}
{\rm length}(\{v(t)\})\le \pi.
\eneq

If, in addition, $1>b_1>b_2>\cdots > b_k>0,$ $1>d_1\ge d_2\ge \cdots
\ge d_k>0$ are  given and
\beq\label{NuvMatr-1}
\tau(\phi(f')g'(u))\ge
d_i\tforal \tau\in \text{T}(A)
\eneq
for any functions $f'\in C(X)$ with $0\le f'\le 1$ whose support
contains an open ball with radius $b_i/2$ and any function $ g'\in
C(\T)$ with $0\le g'\le 1$ whose support contains an arc with length
$b_i,$ then one also has that
\beq\label{NuvMatr-2}
\tau(\phi(f'')g''(vu))
\ge (1-\ep_1) d_i\tforal \tau\in \text{T}(A),
\eneq
where  $f''\in C(X)$ with $0\le f''\le 1$ whose support contains
an open ball with radius $c_i$ and $ g''\in C(\T)$ with $0\le g''$
whose support contains an arc with length $2c_i,$ where
$c_i=b_i+\ep_1,$ $i=1,2,...,k.$
\end{lem}

\begin{NN}\label{DDel}
{\rm Define
\beq\label{DD1}
\Delta_{00}(r)={1\over{2(n+1)^2}}\,\,\,{\rm
if}\,\,\,0<{8\pi\over{n+1}}+{4\pi\over{2^{n+2}(n+1)}}<r\le
{8\pi\over{n}}+{4\pi\over{2^{n+1}n}}\\
\eneq
for $n\ge 64$ and
\beq\label{DD2}
\Delta_{00}(r)={1\over{2(65)^2}}\,\,\,{\rm if}\,\,\, r\ge
8\pi/64+{4\pi\over{2^{65}(64)}}.
\eneq

Let $\Delta: (0,1)\to (0,1)$ be a non-decreasing map. Define
\beq\label{DD3}
D_0(\Delta)(r)=\Delta(\pi/n)\Delta_{00}(r)\,\,\,{\rm
if}\,\,\,0<{8\pi\over{n+1}}+{4\pi\over{2^{n+2}(n+1)}}<r\le
{8\pi\over{n}}+{4\pi\over{2^{n+1}n}}\\
\eneq
for $n\ge 64$ and
\beq\label{DD4}
D_0(\Delta)(r)=D_0(\Delta)(4\pi/64)\,\,\,{\rm if}\,\,\, r\ge
8\pi/64+{4\pi\over{2^{65}(64)}}.
\eneq

}
\end{NN}

\begin{lem}\label{NUmatr}

Suppose that $A$ is a unital separable simple \CA\, with $TR(A)\le
1,$ suppose that ${\cal F}\subset A$ is a finite subset and suppose
that $u\in U(A).$
For any $\ep>0$ and any $\eta>0,$ there exists a unitary $v\in
U_0(A)$ and a continuous path of unitaries $\{w(t): t\in
[0,1]\}\subset U_0(A)$ such that
\beq\label{NUM-2}
w(0)=1,\,\,\, w(1)=v,&&\,\,\,\|[f,\, w(t)]\|<\ep\tforal f\in {\cal
F}
\tand t\in [0,1],\tand\\
&&\mu_{\tau\circ \imath}(I_a)\ge \Delta_{00}(a)\tforal \tau\in \text{T}(A)
\eneq
for any open arc $I_a$ with length $a\ge \eta,$ where $\imath: C(\T)\to A$
is defined by $\imath(g)=g(vu)$ for all $g\in C(\T)$ and
$\Delta_{00}$ is defined in \ref{DDel}.

\end{lem}

\begin{proof}
Define
\beq\label{NUM-3}
&&\Delta_{00,n}(r)={7\over{12(k+1)^2}}-\sum_{m=k}^n{2\over{9\cdot 2^{m+1}(m+1)^2}}\\
&&\,\,\,{\rm
if}\,\,\,0<{4\pi\over{k+1}}+\sum_{m=k}^n{4\pi\over{2^{m+1}2^{m+2}(m+1)}}<r\le
{4\pi\over{k}}+\sum_{m=k}^n{4\pi\over{2^{m+1}2^{m+1}m}}
\eneq
if $n\ge k\ge 32,$ and
$\Delta_{00,n}(r)=\Delta_{00,n}(4\pi/32+{4\pi\over{2^{32+1}32}})$
if $r\ge 4\pi/32+{4\pi\over{2^{32+1}32}}.$

Without loss of generality, we may assume that $\eta=4\pi/n$ for
some $n\ge 32.$ We will use the induction to prove the statement
which is exactly the same as that of Lemma \ref{NUmatr}  but
replace $\Delta_{00}$ by $\Delta_{00,k}$ for  $k\ge 32.$ It
follows from \ref{Umatr}, by choosing small $\ep_1,$ the statement
holds for $k=32.$

Now suppose that the  statement  holds for all integers $m$ with
$k\ge m\ge 32.$ Thus we have a continuous path of unitaries
$\{w'(t): t\in [0,1]\}\subset U_0(A)$ such that
\beq\label{NUM-4}
&&w'(0)=1,\,\,\,w'(1)=v',\,\,\,\|[f, \, w'(t)]\|<\ep/2\tforal t\in [0,1]\andeqn\\
\hspace{-0.1in}&&\mu_{\tau\circ \imath_k}(I_a)\ge \Delta_{00,k}\tforal \tau\in \text{T}(A),
\eneq
for all open arcs with length $a\ge 4\pi/k,$ where $\imath_k:
C(\T)\to A$ is defined by $\imath_k(g)=g(v'u)$ for all $g\in C(\T).$

Let
\beq\label{NUM-5}
b_j={4\pi\over{j+1}}+\sum_{m=j}^k{4\pi\over{2^{m+1}2^{m+2}(m+1)}}\andeqn
d_j={7\over{12(j+1)^2}}-\sum_{m=j}^k{2\over{9\cdot
2^{m+1}(m+1)^2}}
\eneq
$j=32, 33,...,k.$ Choose $\ep_1={2\over{9\cdot 2^{k+2}(k+3)^2}}.$
By applying \ref{Umatr}, we obtain a continuous path of unitaries
$\{w''(t): t\in [0,1]\}\subset U_0(A)$ such that
\beq\label{UNM-6}
w''(0)=1,\,\,\,w''(1)=v'',\,\,\,\|[f,\, w''(t)]\|<\ep/2\tforal t\in [0,1]\andeqn\\
\mu_{\tau\circ \imath_{k+1}}(I_b)\ge
{15\pi\over{24(k+1)^2}}\tforal \tau\in \text{T}(A)
\eneq
for all open arcs $I_b$ with length $b\ge {4\pi\over{(k+1)}},$ where
$\imath_{k+1}: C(\T)\to A$ is defined by
$\imath_{k+1}(g)=g(v''(v'u))$ for $g\in C(\T).$ Moreover,  for any open arc $I_{c_j}$ with length $c_j,$
\beq\label{UNM-7}
{\rm \tau\circ \imath_{k+1}}(I_{c_j})\ge (1-\ep_1)d_j\ge
{7\over{12(j+1)^2}}-\sum_{m=j}^{k+1}{2\over{9\cdot
2^{m+1}(m+1)^2}}\tforal \tau\in \text{T}(A),
\eneq
 $j=32,33,...,k.$ Now define $w(t)=w''(t)w'(t)$ for $t\in [0,1].$
Then
\beq\label{UNM-7+}
w(0)=1,\,\,\, w(1)=v''v'\andeqn \|[f,\, w(t)]\|<\ep\tforal t\in
[0,1].
\eneq
  This shows that the statement holds for $k+1.$ By the
induction, this proves the statement.

Note that $\Delta_{00,n}(r)\ge \Delta_{00}(r)$ for all $r\ge
4\pi/n=\eta.$ The lemma follows immediately from the statement.

\end{proof}

\begin{cor}\label{Full1}
Let $C$ be   a unital separable simple amenable  \CA \, with
$TR(C)\le 1$ which satisfies the UCT. Let  $\ep>0,$
 ${\cal F}\subset C$  be a finite subset and let
$1>\eta>0.$

Suppose that $A$ is a unital  simple \CA\, with $TR(A)\le 1,$
$\phi: C\to A$ is a unital \hm\, and $u\in U(A)$  is a unitary
with
\beq\label{Full1-2}
\|[\phi(c),\, u]\|<\ep\tforal c\in {\cal F}.
\eneq

Then there exist a continuous path of unitaries $\{u(t): t\in
[0,1]\}\subset U(A)$ such that
\beq\label{Full-3}
u(0)=u,\,\,\,u(1)=w\andeqn \|[\phi(f),\, u(t)]\|<2\ep
\eneq
for all $f\in {\cal F}$ and $t\in [0,1].$ Moreover, for any open arc $I_a$ with length $a,$
\beq\label{Full-4}
\mu_{\tau\circ \imath}(I_a)\ge \Delta_{00}(a)\tforal a\ge \eta,
\eneq
where $\imath: C(\T)\to A$ is defined by $\imath(f)=f(w)$ for all
$f\in C(\T).$

\end{cor}

\begin{proof}
 Let $\ep>0$ and ${\cal
F}\subset C$ be as described.  Put ${\cal F}_1=\phi({\cal F}).$
 The corollary follows from \ref{NUmatr} by taking $u(t)=w(t)u.$

\end{proof}

The proof of the following lemma follows from the same argument used
in that of \ref{NUmatr} by applying \ref{Nuvmatrix} instead.

\begin{lem}\label{UVh}
Let $\Delta: (0,1)\to (0,1)$ be a non-decreasing map, let $\eta>0,$
let $X$ be a compact metric space and let ${\cal F}\subset C(X)$ be
a finite subset. Suppose that $A$ is a unital simple \CA\,
with $TR(A)\le 1,$ suppose that $\phi: C(X)\to A$ is a unital \hm\,
and suppose that $u\in U(A)$ such that
\beq\label{NUh-1}
\mu_{\tau\circ \phi}(O_a)\ge \Delta(a)\tforal \tau\in \text{T}(A)
\eneq
for any open ball with radius  $a\ge \eta.$
  For
any $\ep>0,$ there exists a unitary $v\in U_0(A)$ and a continuous
path of unitaries $\{v(t): t\in [0,1]\}\subset U_0(A)$ such that
\beq\label{NUh-2}
&&v(0)=1,\,\,\, v(1)=v\\
&&\|[\phi(f),\,v(t)]\|<\ep, \,\,\|[u,\, v(t)]\|<\ep\tforal f\in
{\cal F}
\tand t\in [0,1]\tand\\
&&\tau(\phi(f)g(vu))\ge D_0(\Delta)(a)\tforal \tau\in \text{T}(A)
\eneq
for any $f\in C(X)$ with $0\le f\le 1$ whose support contains an
open ball with radius $a\ge 4\eta$ and any $g\in C(\T)$ with $0\le
g\le 1$ whose support contains an open arc with length $a\ge
4\eta,$ where
$D_0(\Delta)$ is defined in \ref{DDel}.

\end{lem}

\section{The Basic Homotopy Lemma for $C(X)$ }

In this section we will prove Theorem \ref{Bh2} below. We will
apply the results of the previous section to produce the map $L$ which
was required in Theorem \ref{L1} by using a continuous path of
unitaries.

\begin{lem}\label{BL1}
Let $X$ be a compact metric space, let $\Delta: (0,1)\to (0,1)$ be a
non-decreasing map, let $\ep>0,$ let $\eta>0$ and let ${\cal
F}\subset C(X)$ be a finite subset. There exists $\dt>0$ and a
finite subset ${\cal G}\subset C(X)$  satisfying the following:

Suppose that $A$ is a unital  simple \CA\, with $TR(A)\le
1,$ suppose that $\phi: C(X)\to A$ and suppose that $u\in U(A)$ such
that
\beq\label{BL-0}
\|[\phi(f),\, u]\|<\dt\tforal f\in {\cal G}\tand
\eneq
\beq\label{BL-1}
\mu_{\tau\circ \phi}(O_b)\ge \Delta(a)\tforal \tau\in \text{T}(A)
\eneq
for any open balls $O_b$ with radius $b\ge \eta/2.$  There exists a
unitary $v\in U_0(A),$ a unital completely positive linear map $L:
C(X\times \T)\to A$ and  a continuous path of unitaries $\{v(t):
t\in [0,1]\}\subset U_0(A)$ such that
\beq\label{BL-2}
&&v(0)=u,\,\,\, v(1)=v,\,\,\,\|[\phi(f),\, v(t)]\|<\ep
\tforal f\in {\cal F}\tand t\in [0,1],\\
&&\|L(f\otimes z)-\phi(f)v\|<\ep,\,\,\,\|L(f\otimes 1)-\phi(f)\|<\ep
\tforal f\in {\cal F}\tand\\
&&\mu_{\tau\circ L}(O_a)\ge (2/3)D_0(\Delta)(a/2)\tforal \tau\in
\text{T}(A)
\eneq
for any open balls $O_a$ of $X\times \T$ with radius  $a\ge 5\eta,$
where
$D_0(\Delta)$ is defined in \ref{DDel}.

\end{lem}

\begin{proof}
Fix $\ep>0,$ $\eta>0$ and a finite subset ${\cal F}\subset C(X).$
Let ${\cal F}_1\subset C(X)$ be a finite subset containing ${\cal
F}.$ Let $\ep_0=\min\{\ep/2, D_0(\Delta)(\eta)/4\}.$ Let ${\cal
G}\subset C(X)$ be a finite subset containing ${\cal F},$
$1_{C(X)}$  and $z.$ There is $\dt_0>0$ such that there is a
unital completely positive linear map $L': C(X\times \T)\to B$
(for unital \CA\, $B$) satisfying the following:
\beq\label{BL-4}
\|L'(f\otimes z)-\phi'(f)u'\|<\ep_0\tforal f\in {\cal F}_1
\eneq
for any unital \hm\, $\phi': C(X)\to B$ and any unitary $u'\in B$
whenever
\beq\label{BL-5}
\|[\phi'(g),\, u']\|<\dt_0\tforal g\in{\cal G}.
\eneq
Let $0<\dt<\min\{\dt_0/2, \ep/2, \ep_0/2\}$ and suppose that
\beq\label{BL-6}
\|[\phi(g),\, u]\|<\dt\tforal g\in {\cal G}.
\eneq
It follows from \ref{UVh} that there is a continuous path of
unitaries $\{z(t): t\in [0,1]\}\subset U_0(A)$ such that
\beq\label{BL-7}
&&z(0)=1,\,\,\,z(1)=v_1,\\
&& \|[\phi(f),\, z(t)]\|<\dt/2,\,\,\,\|[u,\, z(t)]\|<\dt/2\tforal t\in [0,1]\andeqn\\
&&\tau(\phi(f)g(v_1u))\ge D_0(\Delta)(a)
\eneq
for any $f\in C(X)$ with $0\le f\le 1$ whose support contains an
open ball with radius $4\eta$ and $g\in C(\T)$ with $0\le g\le 1$
whose support contains open arcs with length $a\ge 4\eta.$

Put $v=v_1u.$ Then we obtain a unital completely positive linear
map $L: C(X\times \T)\to A$ such that
\beq\label{BL-8}
\|L(f\otimes z)-\phi(f)v\|<\ep_0\andeqn \|L(f\otimes
1)-\phi(f)\|<\ep_0\tforal f\in {\cal F}_1.
\eneq
If ${\cal F}_1$ is sufficiently large (depending on $\eta$ only),
we may also assume that
\beq\label{BL-8+}
\mu_{\tau\circ L}(B_a\times J_a)\ge (2/3)D_0(\Delta)(a/2)
\eneq
for any open ball $B_a$ with radius $a$ and  open arcs with length
$a,$ where $a\ge 5\eta.$

\end{proof}

\begin{thm}\label{Bh2}
Let $X$ be a finite CW complex so that $X\times \T$ has the property
{\rm (H)}.  Let $C=PC(X, M_n)P$ for some projection $P\in C(X, M_n)$
and let $\Delta: (0,1)\to (0,1)$ be a non-decreasing map. For any
$\ep>0$ and any finite subset
 ${\cal F}\subset C,$ there exists  $\dt>0,$ $\eta>0$ and there exists a
 finite subset ${\cal G}\subset C$ satisfying the following:

 Suppose that $A$ is a unital  simple \CA\, with $TR(A)\le
 1,$ $\phi: C\to A$ is a unital \hm\, and $u\in A$ is a
 unitary and  suppose that
\beq\label{L3-0}
\|[\phi(c),\, u]\|<\dt\tforal c\in {\cal G}  \tand
{\rm Bott}(\phi, u)=\{0\}.
\eneq
Suppose also that
\beq\label{L3-1}
\mu_{\tau\circ \phi}(O_a)\ge \Delta(a)
 \eneq
 for all open balls $O_a$ of $X$ with radius $1>a\ge \eta,$
 where $\mu_{\tau\circ \phi}$ is the Borel probability  measure defined by restricting
 $\phi$ on the center of $C.$ Then
 there exists a continuous path of unitaries $\{u(t): t\in [0,1]\}$
 in $A$
 such that
 \beq\label{L3-3}
u(0)=u,\,\,\, u(1)=1\tand \|[\phi(c),\, u(t)]\|<\ep
 \eneq
for all $c\in {\cal F}$ and for all $t\in [0,1].$

\end{thm}

\begin{proof}
First it is easy to see that the general case can be reduced to the
case that $C=C(X, M_n).$  It is then  easy to see that this case can
be further reduced to the case that $C=C(X).$  Then the theorem
follows from the combination of \ref{L2} and \ref{BL1}.

\end{proof}

\begin{cor}\label{CMain}
Let  $k\ge 1$ be an integer, let $\ep>0$ and let $\Delta: (0,1)\to
(0,1)$ be any non-decreasing map. There exist $\dt>0$ and $\eta>0$
($\eta$ does not depend on $\Delta$)  satisfying the following: For
any $k$ mutually commutative unitaries $u_1,u_2,...,u_k$ and a
unitary $v\in U(A)$ in a unital separable simple \CA\, $A$ with
tracial rank no more than one for which
$$
\|[u_i,\, v]\|<\dt, \,{\rm bott}_j(u_i, v)=0,\,\,\, j=0,1, \,i=1,2,...,k,
\tand \mu_{\tau\circ \phi}(O_a)\ge \Delta(a)\tforal \tau\in
\text{T}(A),
$$
for any open ball $O_a$ with radius $a\ge \eta,$ where $\phi:
C(\T^k)\to A$ is the \hm\, defined by $\phi(f)=f(u_1,u_2,...,u_k)$
for all $f\in C(\T^k),$ there exists a continuous path of unitaries
$\{v(t): t\in [0,1]\}\subset A$ such that $v(0)=v,$ $v(1)=1$ and
$$
\|[u_i,\, v(t)]\|<\ep\tforal t\in [0,1],\,\,\,\,
i=1,2,...,k.$$

\end{cor}



\begin{rem}\label{CCrem}
{\rm In \ref{CMain}, if $k=1,$ the condition that
$\text{bott}_0(u_1,v)=0$ is the same as $v\in U_0(A).$ Note that in
Theorem \ref{Bh2}, the constant $\dt$ depends not only on $\ep$ and
the finite subset ${\cal F}$ but also depends on the measure
distribution $\Delta.$ As in section 9  of \cite{Lnhomp}, in
general,  $\dt$ can not be chosen independent of $\Delta.$

 Unlike the Basic Homotopy Lemma in simple \CA s of real rank
zero,  in Theorem \ref{Bh2} as well as in  \ref{CMain}, the length
of $\{u(t)\}$ (or $\{v_t\}$) can not be possibly controlled. To see
this, one notes that, it is known (see \cite{P1}) that ${\rm
cel}(A)=\infty$ for some simple AH-algebras with no dimension
growth. It is proved (see \cite{G}, or Theorem 2.5 of \cite{Lnctr1}
) that all of these \CA s $A$ have tracial rank one. For those
simple \CA s,  let $k=1.$ For any number $L>\pi,$ choose $u=v$ and
$v\in U_0(A)$ with ${\rm cel}(v)>L.$ This gives an example that the
length of $\{v_t\}$ is longer than $L.$ This shows that, in general,
the length of $\{v_t\}$ could be as long as one wishes.

However, we can always assume that the path $\{u(t): t\in [0,1]\}$
is piece-wise smooth. For example,  suppose that  $\{u(t): t\in
[0,1]\}$ satisfies the conclusion of \ref{Bh2} for $\ep/2.$ There
are $0=t_0<t_1<\cdots t_n=1$ such that
$$
\|u(t_i)-u(t_{i-1})\|<\ep/32,\,\,\,i=1,2,...,n.
$$
There is a selfadjoint element $h_i\in A$ with $\|h_i\|\le \ep/8$
such that
$$
u(t_i)=u(t_{i-1})\exp(\sqrt{-1} h_i),\,\,\,i=1,2,...,n.
$$
Define
$$
w(t)=u(t_{i-1})\exp(\sqrt{-1}
({t-t_{i-1}\over{t_i-t_{i-1}}})h_i)\tforal t\in [t_{i-1}, t_i),
$$
$i=1,2,...,n.$ Note that
$$
\|[\phi(c),\, w(t)]\|<\ep\tforal t\in [0,1].
$$
On the other hand, it is easy to see that $w(t)$ is continuous and
piece-wise smooth.

}

\end{rem}

\section{ An approximate unitary equivalence result}

The following is a variation of some results in \cite{GLk}. We refer to \cite{GLk}
for the terminologies used in the following statement.

\begin{thm}{\rm (cf.\,Theorem 1.1 of \cite{GLk})} \label{Uni}
Let $C$ be a unital separable amenable \CA\, satisfying the UCT. Let
$b\ge 1,$ let $T: \N^2\to \N,$ $L: U(M_{\infty}(C))\to \R_+,$ $E:
\R_+\times \N\to \R_+$ and $T_1=N\times K: C_+\setminus \{0\}\to
\N\times \R_+\setminus\{0\}$ be four  maps.
For any $\ep>0$ and any finite subset ${\cal F}\subset C,$ there
exists $\dt>0,$ a finite subset ${\cal G}\subset C,$ a finite subset
${\cal H}\subset C_+\setminus\{0\},$ a finite subset ${\cal
P}\subset \underline{K}(C),$ a finite subset ${\cal U}\subset
U(M_{\infty}(C)),$ an integer $l>0$ and an integer $k>0$ satisfying
the following:

for any unital \CA\, $A$ with stable rank one, $K_0$-divisible rank
$T,$ exponential length divisible rank $E$ and $\rm{cer}(M_m(A))\le b$
(for all $m$), if $\phi, \psi: C\to A$ are two unital $\dt$-${\cal
G}$-multiplicative \morp s with
\beq\label{Uni-1}
[\phi]|_{\cal P}=[\psi]|_{\cal P}\andeqn {\rm
cel}(\langle\phi\rangle(u)^*\langle\psi\rangle (u))\le L(u)
\eneq
for all $u\in {\cal U},$ then for any unital $\dt$-${\cal
G}$-multiplicative \morp\, $\theta: C\to M_l(A)$ which is also
$T$-${\cal H}$-full,  there exists a unitary $u\in M_{lk+1}(A)$ such
that
\beq\label{Uni-2}
\|u^*{\rm diag}(\phi(a), {\overbrace{\theta(a), \theta(a),\cdots,
\theta(a)}^k})u- {\rm diag}(\psi(a), {\overbrace{\theta(a),
\theta(a),\cdots, \theta(a)}^k}\|<\ep
\eneq
for all $a\in {\cal F}.$

\end{thm}

\begin{proof}

Suppose that the theorem is false. Then there exists $\ep_0>0$ and a
finite subset ${\cal F}\subset C$  such that there are a sequence of
positive numbers $\{\dt_n\}$ with $\dt_n\downarrow 0,$ an increasing
sequence of finite subsets $\{{\cal G}_n\}$ whose union is dense in
$C,$  an increasing sequence of finite subsets $\{{\cal
H}_n\}\subset C_+\setminus\{0\}$ whose union is dense in $C_+,$ a
sequence of finite subsets $\{{\cal P}_n\}$ of $\underline{K}(C)$
with $\cup_{n=1}^{\infty}{\cal P}_n=\underline{K}(C),$ a sequence of
finite subsets $\{{\cal U}_n\}\subset U(M_{\infty}(C)),$ two
sequences of $\{l(n)\}$ and $\{k(n)\}$ of integers (with
$\lim_{n\to\infty}l(n)=\infty$), a sequence of unital
 \CA\, $A_n$ with stable rank one, $K_0$-divisible rank $T,$
exponential length divisible rank $E$ and $\text{cer}(M_m(A_n))\le b$ (for
all $m$) and sequences $\{\phi_n\},$ $\{\psi_n\}$ of ${\cal
G}_n$-$\dt_n$-multiplicative \morp s from $C$ into $A_n$ with
\beq\label{Uni-3}
[\phi_n]|_{\cal P}=[\psi_n]|_{\cal P}\andeqn {\rm cel}(\langle
\phi\rangle(u)\langle\psi \rangle (u^*))\le L(u)
\eneq
$u\in {\cal U}_n$ satisfying the following:
\beq\label{Uni-4}
\inf\{\sup\{\|v^*{\rm diag}(\phi_n(a), S_n(a))v-{\rm
diag}(\psi_n(a), S_n(a))\|: a\in {\cal F}\}\ge \ep_0
\eneq
where the infimum is taken among all unital $T_1$-${\cal H}_n$-full
and $\dt_n$-${\cal G}_n$-multiplicative \morp s $\sigma_n : C\to
M_{l(n)}(A_n)$ and where
$$
S_n(a)={\rm diag}({\overbrace{\sigma_n(a), \sigma_n(a),\cdots,
\sigma_n(a)}^{k(n)}}),
$$
and among all unitaries $v$ in $M_{l(n)k(n)+1}(A_n).$

Let $B_0=\bigoplus_{n=1}^{\infty}A_n,$ $B=\prod_{n=1}^{\infty}B_n,$
$Q(B)=B/B_0$ and $\pi: B\to Q(B)$ be the quotient map.  Define
$\Phi, \Psi: C\to  B$  by $\Phi(a)=\{\phi_n(a)\}$ and
$\Psi(a)=\{\psi_n(a)\}$ for $a\in C.$ Note that $\pi\circ \Phi$ and
$\pi\circ \Psi$ are \hm.

For any $u\in {\cal U}_m,$ since $A_n$ has stable rank one,
when $n\ge m,$
\beq\label{Uni-5}
\langle \phi_n\rangle (u)(\langle\psi_n\rangle(u))^*\in U_0(A_n)
\andeqn {\rm cel}(\langle \phi_n\rangle
(u)(\langle\psi_n\rangle(u))^*)\le L(u).
\eneq
It follows that, for all $n\ge m,$ (by Lemma 1.1 of \cite{GLk} for
example), there is a continuous path $\{U(t)\in
\prod_{n=m}^{\infty}A_n: t\in [0,1]\}$ such that
$$
U(0)=\{\langle \phi_n\rangle (u)\}_{n\ge m}\andeqn U(1)=\{\langle
\psi_n\rangle (u)\}_{n\ge m}.
$$
Since this holds for each $m,$  it follows that
\beq\label{Uni-6}
(\pi\circ \Phi)_{*1}=(\pi\circ \Psi)_{*1}
\eneq
It follows from (2) of Corollary 2.1 of \cite{GLk} that
\beq\label{Uni-7}
K_0(B)=\prod_bK_0(B_n)\andeqn
K_0(Q(B))=\prod_bK_0(B_n)/\bigoplus_nK_0(B_n).
\eneq
Then,  by (\ref{Uni-3}) and by using the fact that each $B_n$ has
stable rank one again, one concludes that
\beq\label{Uni-8}
(\pi\circ \Phi)_{*0}=(\pi\circ \Psi)_{*0}
\eneq
Moreover, with the same argument, by (\ref{Uni-3}) and by applying
(2) of Corollary 2.1 of \cite{GLk},
\beq\label{Uni-9}
[\pi\circ \Phi]|_{K_i(C, \Z/k\Z)}=[\pi\circ \Psi]|_{K_i(C, \Z/k\Z)},
\,\,\,k=2,3,...,\andeqn i=0,1.
\eneq
Since $C$ satisfies the UCT, by \cite{DLd},
\beq\label{Uni-10}
[\pi\circ \Phi]=[\pi\circ \Psi]\,\,\,{\rm in}\,\,\, KL(C, Q(B)).
\eneq
On the other hand, since each $\sigma_n$ is $\dt_n$-${\cal
G}_n$-multiplicative and $T_1$-${\cal H}_n$-full, we conclude that
$\pi\circ \Sigma$ is a full \hm, where $\Sigma: C\to B$ is defined by $\Sigma(c)=\{\sigma_n(c)\}$ for $c\in
C.$

It follows from Theorem 3.9 of \cite{Lnauct} that there exists an
integer $N$ and a unitary  ${\bar W}\in Q(B)$ such that
\beq\label{OUni-11}
&&\|{\bar W}^*{\rm diag}(\pi\circ \Phi(c), {\overbrace{\pi\circ
\Sigma(c), \cdots, \pi\circ \Sigma(c)}^N}){\bar W}\\
&&\hspace{0.7in}-{\rm diag}(\pi\circ\Psi(c), \pi\circ
{\overbrace{\pi\circ \Sigma(c), \cdots, \pi\circ
\Sigma(c)}^N}\|<\ep_0/2
\eneq
for all $c\in {\cal F}.$  There exists a unitary $u_n\in A_n$ for
each $n$ such that $\pi(\{u_n\})={\bar W}.$ Therefore, by
(\ref{Uni-12}), for some large $n_0\ge 0,$
\beq\label{Uni-12}
&&\|u_n^*{\rm diag}(\phi_n(c), \overbrace{\sigma_n(c), \cdots,
\sigma_n(c)}^N)u_n\\
&&\hspace{0.7in} -{\rm diag}(\psi_n(c),\overbrace{\sigma_n(c),
\cdots, \sigma_n(c)}^N)\|<\ep_0
\eneq
for all $c\in {\cal F}.$ This contradicts with (\ref{Uni-4}).

\end{proof}

\begin{rem}\label{Runi}
{\rm Suppose that $U(C)/U_0(C)=K_1(C).$ Then, from the proof, one
sees that we may only consider ${\cal U}\subset U(C).$

}
\end{rem}

\begin{thm}\label{ACT}
Let $C$ be a unital separable simple amenable \CA\, with $TR(C)\le
1$ satisfying the UCT and let $D=C\otimes C(\T).$  Let $T=N\times K:
D_+\setminus\{0\}\to \N_+\times \R_+\setminus \{0\}.$

Then, for any $\ep>0$ and any finite subset ${\cal F}\subset D,$
there exists $\dt>0,$ a finite subset ${\cal G}\subset D,$ a finite
subset ${\cal H}\subset D_+\setminus\{0\},$  a finite subset ${\cal
P}\subset \underline{K}(C)$ and a finite subset ${\cal U}\subset
U(D)$ satisfying the following: Suppose that $A$ is a unital
 simple \CA\, with $TR(A)\le 1$ and $\phi, \psi: D\to A$
are two unital $\dt$-${\cal G}$-multiplicative \morp s such that
$\phi, \psi$ are $T$-${\cal H}$-full,
\beq\label{ACT-1}
|\tau\circ \phi(g)-\tau\circ \psi(g)|<\dt\tforal g\in {\cal G}
\eneq
for all $\tau\in \text{T}(A),$
\beq\label{ACT-2}
&&[\phi]|_{\cal P}=[\psi]|_{\cal P}\tand \\\label{ACT-3} && {\rm
dist}(\phi^{\ddag}({\bar w}), \psi^{\ddag}({\bar w}))<\dt
\eneq
for all $w\in {\cal U}.$ Then there exists a unitary $u\in U(A)$
such that
\beq\label{ACT-4}
{\rm ad}\, u\circ \psi\approx_{\ep} \phi\,\,\,\text{on}\,\,\,{\cal
F}.
\eneq
\end{thm}

\begin{proof}
 Let $\ep>0$ and a finite subset ${\cal F}\subset D.$ Fix
 $T=N\times K$ as given. Let $T_1=N\times 2K.$
Let $1>\dt_1>0,$ ${\cal G}_1\subset D,$
 ${\cal H}_1\subset D_+\setminus\{0\},$ ${\cal P}_1\subset \underline{K}(D),$
 ${\cal U}_1\subset U(M_{\infty}(D)),$ integer $l$ and $k$ as required
 by \ref{Uni} for $\ep/8,$  ${\cal F}$ and $T$ as well as for $b=2,$
 $T(n,m)=1,$ $L(u)=2{\rm cel}(u)+8\pi+1,$ $E(l, k)=8\pi+l/k$ .
 We may assume that $\dt_1<\min\{\ep/8,1/8\pi\}$ and  $k\ge 2.$
Without loss of generality, we may assume that ${\cal G}_1=\{g\otimes 1, g\in{ \cal
 G}_0\}\cup\{1\otimes z\},$ where ${\cal G}_0\in C$ and $z$ is the identity
 function on $\T,$ the unit circle. Note that $K_1(D)=K_1(A)\bigoplus K_0(A).$
It is clear that $K_1(D)$ is generated by $u\otimes 1$ and
$(p\otimes z)+(1-p)\otimes 1$ for $u\in U(A)$ and projections $p\in
A.$ In particular, $K_1(D)=U(D)/U_0(D).$ Thus (see the remark
\ref{Runi}), we may assume that ${\cal U}_1\subset U(A).$

Since $TR(C)\le 1,$ for any $\dt_2>0,$ there exists a projection
$e\in C$ and a \SCA\, $C_0\in {\cal I}$ with $1_{C_0}=e$  and a
\morp\, $j_1: C\to C_0$ such that
\begin{enumerate}

\item $\|[x,e]\|<\dt_2$ for $x\in {\cal G}_0;$

\item ${\rm dist}(exe, j_1(x))<\dt_2/4$ for $x\in {\cal G}_0$ and

\item $(2kl+1)\tau(1-e)<\tau(e)$ and $\tau(1-e)<\dt_2/(2kl+1)$ for
all $\tau\in T(C).$

\end{enumerate}

Put $z_0=(1-e)\otimes z,$ $z_1=e\otimes z$ and $j_0(c)=(1-e)c(1-e)$
for $c\in C.$
We may also assume that $\dt_2<\dt_1/4.$ Put ${\cal
G}_{00}=j_1({\cal G}_0).$
Thus
\beq\label{ACT-10}
{\rm dist}(exe, {\cal G}_{00})<\dt/4\rforal x\in {\cal G}_0.
\eneq
Let $D_0=C_0\otimes C(\T).$  Let $T'=T|_{(D_0)_+\setminus\{0\}}.$
Let $\dt_3>0$ (in place $\dt$), let ${\cal G}_2$ (in place of ${\cal
G}$) be a finite subset of $D_0,$ let ${\cal H}_2\subset
(D_0)_+\setminus\{0\},$ let ${\cal P}_2$ (in place of ${\cal P}$) be
a finite subset of $\underline{K}(D_0)$ and let ${\cal U}_2$ (in
place of ${\cal U}$) be a finite subset of $U(M_{\infty}(D_0))$
required by Theorem 11.5  of \cite{Lnn1} for $\dt_1/4$ (in place of
$\ep$), ${\cal G}_{00}\cup\{z_1\}$ (in place of ${\cal F}$) (and
$D_0$ in place of $C$). Here we identify $e$ with $1_{D_0}.$
Let $J=j_1\otimes {\rm id}_{C(\T)}: D_0\to D$ be the obvious
embedding and $J_0=j_0\otimes {\rm id}_{C(\T)}.$
 Let ${\cal P}_2'\in \underline{K}(D)$ be the image of ${\cal
P}_2$ under $[J].$

 Now let $\dt=\min\{\dt_2/(8kl+1), \dt_3/(8kl+1)\},$
 ${\cal G}={\cal G}_1\cup {\cal G}_2\cup\{e, (1-e)\}.$ Here we also view
${\cal G}_2$ as a subset of $D.$  Let ${\cal H}={\cal H}_1\cup {\cal
H}_2,$ let ${\cal P}={\cal P}_1\cup{\cal P}_2'$ and ${\cal U}={\cal
U}_1\cup \{ (e+\langle j_0\rangle(u): u\in {\cal U}_1\}\cup
\{(1-e)+v: v\in {\cal U}_2\}.$

Suppose that $\phi$ and $\psi$ satisfy the assumptions of the
theorem for the above ${\cal G},$ ${\cal H},$ ${\cal P}$ and ${\cal
U}.$ Let $\phi'=\phi\circ J,$ $\psi'=\psi\circ J.$ There is a
unitary $u_0\in A$ such that
$$
u_0^*\psi'(e)u_0=\phi'(e)=e_0\in A.
$$
Put $A_1=e_0Ae_0.$
We have $[{\rm ad}\, u_0\circ \psi']|_{{\cal P}_2}=[\phi']|_{{\cal P}_2}$ and, for $g\in {\cal G},$
\beq\label{ACTn-1}
|t\circ {\rm ad}\, u_0\circ \psi'(g)-t\circ
\phi'(g)|&<&{\dt\over{1-{\dt/(2kl+1)}}}<\dt_3\tforal t\in T(eAe).
\eneq
Moreover, by the first part of \ref{ph},
\beq\label{ACTn-2}
{\rm dist}(({\rm ad}\, u_0\circ \psi')^{\ddag}({\bar w}),
(\phi')^{\ddag}({\bar w}))<(2+1)\dt<\dt_3
\eneq
for all $w\in {\cal U}_2.$

By the choice of ${\cal G}_2,$ ${\cal H}_2,$ ${\cal U}_2$ and ${\cal
P}_2,$  and by applying 11.5 of \cite{Lnn1}, there is a unitary
$u_1\in A_1$ such that
\beq\label{Uni-11}
{\rm ad}\, u_1\circ{\rm ad}\, u_0\circ  \psi'\approx_{\ep/2}
\phi'\,\,\,{\rm on}\,\,\, {\cal G}_{00}.
\eneq

Let ${\cal G}_{00}'$ be a finite subset containing ${\cal
G}_{00}\cup j({\cal H}_1)$ and $\dt_4>0.$  Since $TR(A_1)\le 1,$ by
Lemma 5.5 of \cite{Lnctr1}, there are mutually orthogonal
projections $q_0, q_1,q_2,...,q_{8kl+4}$ with $[q_0]\le [q_1]$ and
$[q_1]=[q_i],$ $i=1,2,...,8kl+4,$ and there are unital
$\dt_4$-${\cal G}_{00}'$ multiplicative \morp s $L_0: D_0\to
q_0Aq_0$ and $L_i: D_0\to q_iAq_i$ ($i=1,2,...,8kl+4$) such that
\beq\label{ACT-12}
\phi'\approx_{\dt_4} L_0\oplus L_1\oplus L_2\oplus\cdots\oplus
L_{8kl+4}\,\,\,{\rm on}\,\,\, {\cal G}_{00}'',
\eneq
and there exists a unitary $W_i\in (q_1+q_i)A(q_1+q_i)$ such that
\beq\label{Uni-13-}
{\rm ad} W_i\circ L_i=L_1,\,\,\,i=1,2,...,8kl.
\eneq

Since $\phi$ is $T$-${\cal H}$-full, with sufficiently small $\dt_4$
and sufficiently large ${\cal G}_{00}'$ we may also assume that each
$L_i\circ j_1$ is also $T_1$-${\cal H}_1$-full and $\dt_4<\dt/4.$
Define  $Q_i=\sum_{j=4+8(i-1)}^{8i+4}q_i,$  $Q_0=\sum_{i=1}^{4} q_i$
and $\Phi_i=\sum_{j=4+8(i-1)}^{8i+4} L_i,$ $i=1,2,...,kl.$ Note by
\ref{Uni-13-}), $\Phi_i$ are unitarily equivalent to $\Phi_1.$

Since $K_0(A)$ is weakly unperforated (see Theorem 6.11  of
\cite{Lntr}), we check that
\beq\label{Uni-13}
[p_0+q_0+Q_0]\le [Q_i] \andeqn 2[p_0+q_0+Q_0]\ge [Q_i],\,\,\,
i=1,2,...,kl.
\eneq
Put $\phi_0=\phi\circ J_0\oplus L_0\circ J\oplus
\sum_{i=1}^{4}L_i\circ J$ and $\psi_0=\psi\circ J_0\oplus L_0\circ
J\oplus \sum_{i=1}^{4}L_i\circ J.$
By \ref{ph},
we compute that
\beq\label{Uni-14}
{\rm dist}(\phi_0^{\ddag}({\bar w}), \psi_0^{\ddag}({\bar w}))<\dt_1
\tforal w\in {\cal U}
\eneq
It follows Lemma 6.9 of \cite{Lnctr1} that
\beq\label{Uni-14+}
{\rm cel}(\langle \phi_0\rangle(u)\langle \psi_0\rangle(u)^*)<8\pi+1
\tforal u\in {\cal U}.
\eneq

We also have
\beq\label{Uni-15}
[\phi_0]|_{{\cal P}_1}=[\psi_0]|_{{\cal P}_1}.
\eneq

Since $\Phi_1\circ j$ is $T_1$-${\cal H}_1$-full, by applying
\ref{Uni}, we obtain a unitary $w\in U(A),$
\beq\label{Uni-16}
\hspace{-0.3in}\|w^*{\rm diag}(\psi_0(c), \overbrace{\Phi_1\circ
J(c),...,\Phi_1\circ J(c)}^{kl})w-{\rm
diag}(\phi_0(c),\overbrace{\Phi_1\circ J(c),...,\Phi_1\circ
J(c)}^{kl})\|<\ep/8
\eneq
for all $c\in {\cal F}.$ Since $\Phi_i\circ j_1$ is unitarily
equivalent to $\Phi_1\circ j_1,$ there is a unitary $w'\in U(A)$
such that
\beq\label{Uni-17}
&&\hspace{-0.8in}\|(w')^*{\rm diag}(\psi_0(c), \Phi_1\circ
J(c),...,\Phi_{kl}\circ J(c))w'\\
&&-{\rm diag}(\phi_0(c), \Phi_1\circ J(c),...,\Phi_{kl}\circ
J(c))\|<\ep/8
\eneq
for all $c\in {\cal F}.$ It follows that
\beq\label{Uni-18}
\|(w')^*{\rm diag}(\psi\circ J_0(c), L_0\circ J(c),
\phi'(c))w'-{\rm diag}(\phi\circ J_0(c), L_0\circ J(c),
\phi'(c))\|<\ep/4
\eneq
for all $c\in {\cal F}.$

Let $u=((1-e_0)\oplus e_0u_0u_1)w'.$ Then, by (\ref{Uni-11}), we
have
\beq\label{Uni-19}
\|u^*{\rm diag}(\psi\circ J_0(c), \psi'(c))u-{\rm diag}(\phi\circ
J_0(c), \phi'(c))\|<\ep/2
\eneq
for all $c\in {\cal F}.$ It follows that
\beq\label{Uni-20}
{\rm ad}\, u\circ \psi\approx_{\ep} \phi\,\,\,{\rm on}\,\,\, {\cal
F}.
\eneq

\end{proof}

\begin{cor}\label{CTAC}
Let $C$ be a unital separable amenable simple \CA\, with $TR(C)\le
1$ which satisfies the UCT,  let $D=C\otimes C(\T)$ and let $A$ be a
unital simple \CA\, with $TR(A)\le 1.$  Suppose that $\phi,
\psi: D\to A$ are two unital monomorphisms. Then $\phi$ and $\psi$
are approximately unitarily equivalent, i.e., there exists a
sequence of unitaries $\{u_n\}\subset A$ such that
$$
\lim_{n\to\infty} {\rm ad}\, u_n\circ \psi(d)=\phi(d) \tforal
d\in D,
$$
if and only if
$$
[\phi]=[\psi]\,\,\, {\rm in}\,\,\, KL(D, A),\\
\tau\circ \phi=\tau\circ \psi\tforal \tau\in \text{T}(A)\andeqn
\psi^{\ddag}=\phi^{\ddag}.
$$
\end{cor}

\section{The Main Basic Homotopy Lemma}

\begin{lem}\label{trace}
Let $C$ be a unital separable simple \CA\, with $TR(C)\le 1$ and let
$\Delta: (0,1)\to (0,1)$ be a  non-decreasing map.
There exists a map $T=N\times K: D_+\setminus \{0\}\to \N_+\times
\R_+\setminus\{0\},$ where $D=C\otimes C(\T),$  satisfying the
following:

For any $\ep>0,$ any finite subset ${\cal F}\subset C$ and any
finite subset ${\cal H}\subset D_+\setminus\{0\},$  there exists
$\dt>0,$ $\eta>0$ and a finite subset ${\cal G}\subset C$ satisfying
the following: for any unital separable unital simple \CA\, $A,$ any
unital \hm\, $\phi: C\to A$ and any unitary $u\in A$ such that
\beq\label{trace-1}
\|[\phi(c), \, u]\|<\dt\tforal c\in {\cal G}\tand\\
\mu_{\tau\circ \imath}(O_a)\ge \Delta(a)\tforal \tau\in \text{T}(A)
\eneq
and for all open balls $O_a$ with radius $a\ge \eta,$ where
${\imath}: C(\T)\to A$ is defined by $\imath(f)=f(u),$  there is a
unital \morp\, $L: D\to A$ such that
\beq\label{trace-2}
\|L(c\otimes 1)-\phi(c)\|<\ep \,\,\,\|L(c\otimes z)-\phi(c)u\|<\ep
\tforal c\in {\cal F}
\eneq
and $L$ is $T$-${\cal H}$-full.

\end{lem}

\begin{proof}

We identify $D$ with $C(\T, C).$ Let $f\in D_+\setminus\{0\}.$
There is positive number $b\ge 1,$  $g\in D_+$ with $0\le g\le
b\cdot 1$ and $f_1\in D_+\setminus\{0\}$ with $0\le f_1\le 1$ such
that
\beq\label{trace-n}
gfgf_1=f_1.
\eneq

There is a point $t_0\in \T$ such that $f_1(t_0)\not=0.$ There is
$r>0$ such that
$$
\tau(f_1(t))\ge \tau(f_1(t_0))/2
$$
for all $\tau\in T(C)$ and for all $t$ with ${\rm dist}(t, t_0)<r.$

Define $\Delta_0(f)=\inf\{\tau(f_1(t_0))/4:\tau\in T(C)\}\cdot
\Delta(r).$ There is an integer $n\ge 1$ such that
\beq\label{trace-3}
n\cdot \Delta_0(f)>1.
\eneq
Define $T(f)=(n, b).$
Put
$$
\eta=\inf\{ \Delta_0(f): f\in {\cal H}\}/2\andeqn \ep_1=\min\{\ep,
\eta\}.
$$

We claim that there exists an
$\ep_1$-${\cal F}\cup{\cal
H}$-multiplicative \morp\, $L: D\to A$ such that
\beq\label{trace-4}
\|L(c\otimes 1)-\phi(c)\|<\ep\tforal c\in {\cal
F},\,\,\,\|L(1\otimes z)-u\|<\ep\andeqn\\\label{trace-4+}
|\tau\circ L(f_1)-\int_\T \tau(\phi(f_1(s))) d\mu_{\tau\circ
\imath}(s)|<\eta \tforal \tau\in \text{T}(A)
\eneq
and for all $f\in {\cal H}.$
Otherwise, there exists a sequence of unitaries $\{u_n\}\subset
U(A)$ for which $\mu_{\tau\circ \imath_n}(O_a) \ge \Delta(a)$ for
all $\tau\in \text{T}(A)$ and for any open balls $O_a$ with radius $a\ge
a_n$ with $a_n\to 0,$ and for which
\beq\label{tarce-5}
\lim_{n\to\infty}\|[\phi(c),\, u_n]\|=0
\eneq
for all $c\in C$ and suppose for any sequence of \morp s $L_n: D\to
A$ with
\beq\label{trace-5+1}
\lim_{n\to\infty}\|L_n(ab)-L_n(a)L_n(b)\|&=&0\tforal a,\, b\in
D,\\\label{trace-5+1+} \lim_{n\to\infty}\|L_n(c\otimes
f)-\phi(c)f(u_n)\|&=&0,
\eneq
for all $c\in C,$ $f\in  C(\T)$ and
\beq\label{trace-5+2}
\lim\inf_n\{\max\{|\tau\circ L_n(f_1)-\int_\T\tau(\phi(f_1(s))) d
\mu_{\tau\circ \imath_n}(s)|: f\in {\cal H}\}\}\ge \eta
\eneq
for some $\tau\in \text{T}(A),$ where $\imath_n: C(\T)\to D$ is defined by
$\imath_n(f)=f(u_n)$ for $f\in C(\T)$ (or no \morp s $L_n$ exists so
that (\ref{trace-5+1}), (\ref{trace-5+1+}) and (\ref{trace-5+1+})).

Put $A_n=A,$ $n=1,2,...,$ and $Q(A)=\prod_n A_n/\bigoplus_nA_n.$ Let
$\pi: \prod_nA_n\to Q(A)$ be the quotient map. Define a linear map
$L': D\to \prod_n A_n$ by $L(c\otimes 1)=\{\phi(c)\}$ and
$L'(1\otimes z)=\{u_n\}.$ Then $\pi\circ L': D\to Q(A)$ is a unital
\hm.  It follows from a theorem of Effros and Choi (\cite{CE}) that
there exists a \morp\, $L: D\to \prod_nA_n$ such that $\pi\circ
L=\pi\circ L'.$ Write $L=\{L_n\},$ where $L_n: D\to A_n$ is a \morp.
Note that
$$
\lim_{n\to\infty}\|L_n(a)L_n(b)-L_n(ab)\|=0\tforal a\, b\in D.
$$

Fix $\tau\in \text{T}(A), $ define $t_n: \prod_n A_n\to \C$ by
$t_n(\{d_n\})=\tau(d_n).$ Let $t$ be a limit point of $ \{t_n\}.$
Then $t$ gives a state on $\prod_nA_n.$ Note that if $\{d_n\}\in
\bigoplus_n A_n,$ then $t_m(\{d_n\})\to 0.$ It follows that $t$
gives a state ${\bar t}$ on $Q(A).$ Note that (by
(\ref{trace-5+1+}))
$${\bar t}(\pi\circ L(c\otimes
1))=\tau(\phi(c))$$
 for all $c\in C.$
It follows that
\beq\label{trace-10}
{\bar t}(\pi\circ L(f))=\int_{\T} {\bar t}(\pi\circ L(f(s)\otimes 1))d\mu_{{\bar
t}\circ \pi\circ L|_{1\otimes C(\T)}}
=\int_{\T} \tau(\phi(f(s))) d\mu_{{\bar
t}\circ \pi\circ L|_{1\otimes C(\T)}}
\eneq
for all $f\in C(\T, C).$
Therefore, for a subsequence $\{n(k)\},$
\beq\label{trace-11}
|\tau\circ L_{n(k)}(f_1)-\int_{\T} \tau(\phi(f_1(s)) )d
\mu_{\tau\circ \imath_{n(k)}}(s) |<\eta/2
\eneq
for all $f\in {\cal H}.$ This contradicts with (\ref{trace-5+2}).
Moreover, from this, it is easy to compute that
 $$
 \mu_{{\bar t}\circ \pi\circ L|_{1\otimes C(\T)}}(O_a)\ge \Delta(a)
 $$
 for all open balls $O_a$ of $t$ with radius $1>a.$
This proves the claim.

Note that
$$
\int_{\T} \tau\circ \phi(f_1(s)) d\mu_{\tau\circ \imath}(s)\ge
(\tau(\phi(f_1(t_0)/2)))\cdot \Delta(r)
$$
for all $\tau\in \text{T}(A).$ It follows that
\beq\label{trace-12}
\tau(L(f_1))\ge \inf\{t(f_1(t_0))/2: t\in T(C)\}-\eta/2\ge
(4/3)\Delta_0(f)
\eneq
for all $f\in {\cal H}.$

By Corollary 9.4 of \cite{Lnn1}, there exists a projection $e\in
\overline{L(f_1)AL(f_1)}$ such that
\beq\label{trace-13}
\tau(e)\ge \Delta_0(f)\tforal \tau\in \text{T}(A).
\eneq
It follows from (\ref{trace-3}) that there exists a partial isometry $w\in M_n(A)$ such
that
$$
w^*{\rm diag}(\overbrace{e,e,\cdots, e}^n)w\ge 1_A.
$$
Thus there $x_1, x_2,...,x_n\in A$ with $\|x_i\|\le 1$ such that
\beq\label{trace-14}
\sum_{i=1}^n x_i^*ex_i\ge 1.
\eneq
Hence
\beq\label{trace-15}
\sum_{i=1}^n x_i^*gfgx_i\ge 1.
\eneq

It then  follows that there are $y_1,y_2,...,y_n\in A$ with
$\|y_i\|\le b$ such that
\beq\label{trace-16}
\sum_{i=1}^n y_i^*fy_i=1.
\eneq
Therefore $L$ is $T$-${\cal H}$-full.

\end{proof}

\begin{lem}\label{perlength}
Let $C$ be a unital separable  amenable simple \CA\, with $TR(C)\le
1$ satisfying the UCT. For $1/2>\sigma>0,$  any finite subset ${\cal
G}_0$ and any projections $p_1, p_2,...,p_m\in C.$ There is
 $\dt_0>0,$ a finite subset ${\cal G}\subset C$
and a finite subset of projections $P_0\subset C$  satisfying the
following: Suppose that $A$ is a unital simple \CA\, with $TR(A)\le
1,$ $\phi: C\to A$ is a unital \hm\, and $u\in U_0(A)$ is a unitary
such that
\beq\label{pl-1}
\|[\phi(c), \, u]\|<\dt<\dt_0\tforal c\in {\cal G}\cup {\cal G}_0
\tand  {\rm bott}_0(\phi, u)|_{{\cal P}_0}=\{0\},
\eneq
where ${\cal P}_0$ is the image of $P_0$ in $K_0(C).$ Then there
exists a continuous path of unitaries $\{u(t): t\in [0,1]\}$ in $A$
with $u(0)=u$ and $u(1)=w$ such that
\beq\label{pl-2}
\|[\phi(c), \, u(t)]\|&<& 3\dt \tforal c\in {\cal G}\cup{\cal G}_0\tand\\
w_j\oplus (1-\phi(p_j))&\in& CU(A),
\eneq
 where $w_j\in U_0(\phi(p_j)A\phi(p_j))$ and
\beq\label{pl-3}
\|w_j-\phi(p_j)w\phi(p_j)\|<\sigma,
\eneq
$j=1,2,...,m.$

Moreover,
\beq\label{pl-4}
{\rm cel}(w_j\oplus (1-\phi(p_j)))\le 8\pi+1/4,\,\,\,j=1,2,...,m.
\eneq

\end{lem}

\begin{proof}
It follows from the combination of Theorem 4.8 and Theorem 4.9 of
\cite{EGL} and theorem 10.10 of \cite{Lnctr1} that one may write
$C=\lim_{n\to\infty}(C_n, \psi_n),$ where each
$C_n=\bigoplus_{j=1}^{R(n)}P_{n,j}C(X_{n,j}, M_{r(n,j)})P_{n,j}$ and
where $P_{n,i}\in C(X_{n,i}, M_{r(n,i)})$ is a projection and
$X_{n,i}$ is a connected finite CW complex of dimension no more than
two with torsion free $K_1(C(X_{n,i}))$ and
 $K_0(C(X_{n,j}))=\Z\bigoplus \Z/{s(j)}\Z$ ($s(i)\ge
1$) and with positive cone $\{(0,0)\cup (m,x): m\ge 1\}$ (when
$s(j)=1,$ we mean $K_0(C(X_{n,j}))=\Z$),{ \it{or}} $X_{n,i}$ is a
connected finite CW complex of dimension three  with
$K_0(C(X_{n,i}))=\Z$ and torsion $K_1(C(X_{n,i})).$
 Let $d(j)$ be the rank of
$P_{n,j}.$ It is known that one may assume that $d(j)\ge
\prod_{j=1}^{R(n)}s(j)+6,$ $j=1,2,...,R(n).$ This can be seen, for
example,  from Lemma 2.2,  2.3 (and the proof of Theorem 2.1) of
\cite{EGLd}.

Without loss of generality, we may assume that ${\cal G}_0\subset
\psi_{n, \infty}(C_n)$ and that there are projections $p_{i,0}\in
C_n$ such that $\psi_{n, \infty}(p_{i,0})=p_i,$ $i=1,2,...,m.$
Choose, for each $j,$ mutually orthogonal rank one projections $q^{(0)}_{j,0}, q_{j,1}^{(0)} \in
P_{n,j}(C(X_{n,j}, M_{r(n,j)}))P_{n,j}$ such that
$$
[q^{(0)}_{j,0}]=(1,0)\andeqn [q^{(0)}_{j,1}]=(1,{\bar 1})\in \Z\bigoplus
\Z/s(j)\Z,
$$
or $q^{(0)}_{j,1}=0,$ if $K_0(C(X_{n,j}))=\Z,$ $j=1,2,...,R(n).$ Put $q_{j,i}'=\psi_{n, \infty}(q_{j,i}^{(0)})$ and
$q_{j,i}=\phi(q'_{j,i}),$ $i=0,1$ and $j=1,2,...,R(n).$ Clearly, in $C,$
$p_k$ may be written as  $W_j^*Q_jW_j,$ where $Q_j$ is a finite  orthogonal sum of $q_{j,0}$ and $q_{j,1},$
$j=1,2,...,R(n)\}.$

By choosing a  sufficiently large ${\cal G}$ which contains ${\cal
G}_0$ (and which contains $Q_j,$ $q_{j,i}$ as well as $W_j,$ among other elements) and sufficiently small $\dt_0>0,$ one sees that it suffices to
show the case that $\{p_1,p_2,...,p_m\}\subset \{q_{j,0}, q_{j,1}:
j=1,2,...,R(n)\}.$ Thus we obtain a finite subset ${\cal G}'$ and
$\dt_0'$ so that when ${\cal G}\supset {\cal G}'$ and $\dt_0<\dt_0'$
one can make the assumption that $\{p_1,p_2,...,p_m\}\subset
\{q_{j,i}, i=0,1,\, \,\, j=1,2,...,R(n)\}.$  In particular,
$\{q_{j,i}, i=0,1,\, \,\, j=1,2,...,R(n)\}\subset {\cal G}'.$

Let ${\cal G}_0'={\cal G}_0\cup\{q_{j,0}', q_{j,1}':
j=1,2,...,R(n)\}.$
Fix $0<\eta<\min\{\sigma/4, \dt_0'/2, 1/16\}.$ Note that $P_{n,j}$
is locally trivial in $C(X_{n,j}, M_{r(n,j)}).$ Since $TR(C)\le 1,$
it has (SP) (see \cite{Lntr}). It is then easy to find a projection
$e_j\in \psi_{n,\infty}(P_{n,j})C\psi_{n,\infty}(P_{n,j})$ and
$B_j\cong M_{d(j)}\subset
\psi_{n,\infty}(P_{n,j})C\psi_{n,\infty}(P_{n,j})$  with $1_{B_j}=e_j$ such that
\beq\label{pl-6}
\|[x, \, e_j]\|&<&\eta\tforal x\in {\cal G}_0'\\\label{pl-6+1}
{\rm dist}(e_jxe_j,B_j)&<&\eta \tforal x\in {\cal G}_0'\andeqn
e_jq_{j,1}'e_j,\,e_jq_{j,0}'e_j\not=0,
\eneq
$j=1,2,...,R(n).$ Furthermore, one may  require that there is a projection ${\bar
q_{j,i}'}\in B_j$ with rank one  in $B_j$ such that
\beq\label{pl-7}
\|{\bar
q_{j,i}'}-e_jq_{j,i}'e_j\|<2\eta,\,\,\,i=0,1,\,\,\,j=1,2,...,R(n).
\eneq
To simplify notation further, by replacing $q_{j,i}'$ by one of its nearby projections, we may assume that
$q_{j,i}'={\bar q_{j,i}'}+(q_{j,i}'-{\bar q'_{j,i}})$ and $q_{j,i}'\ge {\bar q'_{j,i}},$ $i=0,1$ and $j=1,2,...,R(n).$
Since $s(j)[q_{j,1}']=s(j)[q_{j,0}'],$
there is a unitary $Y_j\in P_{n,j}C(X_{n,j}, M_{r(n,j)})P_{n,j}$
such that
\beq\label{pl-16}
Y_j^*{\rm diag}(\overbrace{q_{j,1}',q_{j,1}',...,q_{j,1}'}^{s(j)},
q_{j,0}', q_{j,0}', q_{j,0}') Y_j ={\rm diag}(\overbrace{q_{j,0}',
q_{j,0}',...,q_{j,0}'}^{s(j)+3}).
\eneq
(Note that $d(j)\ge \prod_{j=1}^{R(n)}s(j)+6$ and each $q_{j,i}'$ has
rank one.)

Let $\{e_{i,k}^{(j)}\}$ be a matrix unit for $B_j,$
$j=1,2,...,R(n).$ We choose a finite subset ${\cal G}$ which
contains ${\cal G}_0'$ as well as $\{e_{i,k}^{(j)}\},$  ${\bar
q_{j,0}'},$ ${\bar q_{j,1}}'$ and $\{Y_j, Y_j^*\},$ $j=1,2,...,R(n).$ Suppose that
$v_{j,0}\in U_0(\phi(e_{1,1}^{(j)})A\phi(e_{1,1}^{(j)}))$ and
\beq\label{pl-8-1}
v_j={\rm
diag}(\overbrace{v_{j,0},v_{j,0},...,v_{j,0}}^{d(j)}),\,\,\,j=1,2,...,R(n).
\eneq
Then
\beq\label{pl-8}
\phi(x)v_j=v_j\phi(x)\tforal x\in B_j,\,\,\,j=1,2,...,R(n).
\eneq

Choose
$$
{\cal P}_0=\{ [q_{j,0}'], [q_{j,1}'],[e_{i,i}^{(j)}], [{\bar
q_{j,0}'}], [{\bar q_{j,1}'}], j=1,2,...,R(n)\}.
$$
Put ${\bar q_{j,i}}=\phi({\bar q_{i,j}'}),$ $i=0,1$ and
$j=1,2,...,R(n).$ We choose $\dt_0''>0$ such that ${\rm
bott}_0(\phi,\,u)|_{{\cal P}_0}$ is well-defined which is zero and
there is a unitary  $u_{j,i}'\in U_0({\bar q_{j,i}}A{\bar q_{j,i}})$
such that
$$
\|u_{j,i}'-{\bar q_{j,i}}u{\bar q_{j,i}}\|<2\dt_0'',\,\,\,i=0,1,\,\,\,
j=1,2,....,R(n),
$$
whenever, $\|[\phi(c),\, u]\|<\dt_0''$ for all $c\in {\cal G}.$

Let $\dt_0=\{1/32, \dt_0'/4,\dt_0''/4,\sigma/8\}.$
Suppose that (\ref{pl-1}) holds for the above ${\cal G},$ ${\cal
P}_0$ and $0<\dt<\dt_0.$ One obtains a unitary $u_{j,i}'\in
U_0({\bar q_{j,i}}A{\bar q_{j,i}})$ and a unitary $u_{j,i}''\in U_0((q_{j,i}-{\bar q_{j,i}})A(q_{j,i}-{\bar q_{j,i}}))$ such that
\beq\label{pl-10}
\|u_{j,i}-q_{j,i}uq_{j,i}\|<2\dt,
\eneq
where $u_{j,i}=u_{j,i}'+u_{j,i}'',$  $i=0,1\andeqn j=1,2,...,R(n).$
It follows \ref{inject} (see also Theorem 6.6 of \cite{Lnctr1})
that there is $v_{j,0}\in U_0(\phi(e_{1,1}^{(j)})A\phi(e_{1,1}^{(j)}))$ such
that
\beq\label{pl-11}
d(j)(\overline{v_{j,0}+(1-\sum_{i=2}^{d(j)}\phi(e_{i,i}^{(j)}))})=\overline{u_{j,0}^*},
\,\,\,{\rm in}\,\,\, U_0(A)/CU(A),\,\,\,j=1,2,...,R(n).
\eneq
Put $v_j$ as in (\ref{pl-8-1}).
It follows from  (\ref{pl-11}) that
\beq\label{pl-12}
\hspace{-0.5in}\overline{(v_j\oplus (1-\phi(e_j)))(u_{j,0}\oplus
(1-q_{j,0})} ={\bar 1},
\eneq
$j=1,2,...,R(n).$ Since $v_{j,0}\in U_0(\phi(e_j)A\phi(e_j)),$ one
has a continuous path of unitaries $\{v_{j,0}(t): t\in [0,1]\}$
such that $v_{j,0}(0)=\phi(e_{1,1}^{(j)})$ and
$v_{j,0}(1)=v_{j,0},$ $j=1,2,...,R(n).$ Put
$$
v_j(t)={\rm diag}(\overbrace{v_{j,0}(t),
v_{j,0}(t),...,v_{j,0}(t)}^{d(j)}),\,\,\, j=1,2,...,R(n).
$$
It follows that
\beq\label{pl-13}
\phi(x)v_j(t)=v_j(t)\phi(x)\tforal x\in B_j
\eneq
and $t\in [0,1],$ $j=1,2,...,R(n).$ Put
$$
u(t)=(\sum_{j=1}^{R(n)}v_j(t)+(1-\sum_{j=1}^{R(n)}\phi(e_j))u\,\,\,{\rm
for}\,\,\,t\in [0,1].
$$
Note that, $u(0)=u$ and, if $\eta$ is sufficiently small,
\beq\label{pl-13+1}
\|[\phi(c), \,u(t)]\|<2(\dt+\eta)<3\dt \tforal c\in {\cal G}.
\eneq
Put
\beq\label{pl-13+2}
&&w=u(1),\,\,\,w_{j,0}=(v_j\oplus (1-\phi(e_j))(u_{j,0}\oplus
(1-q_{j,0}))\andeqn
 \\
&&w_{j,1}=(v_j\oplus (1-\phi(e_j)))(u_{j,1}\oplus (1-q_{j,1}),
w_{j,i}'=v_ju_{j,i},
\eneq
$i=0,1,$ and $j=1,2,...,R(n).$  Define
\beq\label{pl-13+3}
{\bar w}_j=(v_j\oplus (1-\phi(e_j))(u_{j,0}\oplus u_{j,1}\oplus
(1-q_{j,0}-q_{j,1})).
\eneq
We have that
$$
{\bar
w}_jq_{j,i}=w_{j,i}q_{j,i}=v_ju_{j,i}=w_{j,i}'=q_{j,i}w_{j,i},
$$
$i=0,1$ and $j=1,2,...,R(n).$
 Note that, by (\ref{pl-13}),(\ref{pl-6}) and
(\ref{pl-7}),
\beq\label{pl-14}
 \|w_{j,i}q_{j,i}-q_{j,i}wq_{j,i}\|<\sigma
\eneq
$i=0,1,\,\,\,j=1,2,...,R(n).$ By (\ref{pl-12}),
\beq\label{pl-15}
w_{j,0}\in CU(A),\,\,\,j=1,2,...,R(n).
\eneq
It follows from Lemma 6.9 of \cite{Lnctr1} that
\beq\label{pl-15+1}
{\rm cel}(w_{j,0})\le 8\pi+1/4.
\eneq
Put $$E_j=1-\phi(Y_j^*{\rm
diag}(\overbrace{q_{j,1}',q_{j,1}',...,q_{j,1}'}^{s(j)}, q_{j,0}',
q_{j,0}', q_{j,0}') Y_j).$$

It follows from (\ref{pl-16}) that in $U_0(A)/CU(A),$
\beq\label{pl-17}
\hspace{-0.4in}\overline{w_{j,1}^{s(j)}w_{j,0}^3}
&=&\overline{{\rm diag}(\underbrace{{\bar w_j}q_{j,1}, {\bar
w_j}q_{j,1},...,{\bar w_j}q_{j,1}}_{s(j)},
w_{j,0},w_{j,0},w_{j,0})\oplus E_j}\\
&=&\overline{\phi(Y_j^*){\rm
diag}(\underbrace{{\bar w_j}q_{j,1}, {\bar w_j}q_{j,1},...,{\bar
w_j}q_{j,1}}_{s(j)},
{\bar w_j}q_{j,0},{\bar w_j}q_{j,0},{\bar w_j}q_{j,0})\phi(Y_j)\oplus E_j}\\
&=& \overline{{\rm diag}(\underbrace{w_{j,0}, w_{j,0},...,w_{j,0}}_{s(j)+3})
\oplus E_j} = {\bar1},
\eneq
where $j=1,2,...,R(n).$ By (\ref{pl-15}), the above implies that
\beq\label{pl-18}
{\overline{w_{j,1}^{s(j)}}}={\overline{1}},\,\,\,j=1,2,...,R(n).
\eneq
It follows from Theorem 6.11 that
\beq\label{pl-18+1}
w_{j,1}\in CU(A),\,\,\,j=1,2,...,R(n).
\eneq
It follows from Lemma 6.9 of \cite{Lnctr1} that
$$
{\rm cel}(w_{j,1})\le 8\pi+1/4, \,\,\,j=1,2,...,R(n).
$$

\end{proof}

\begin{lem}\label{Tfull1}
Let $C$ be a unital separable simple amenable \CA\, with $TR(C)\le
1$ satisfying the UCT. Let $\Delta: (0,1)\to (0,1)$ be a
non-decreasing map. Then, for any $\ep>0$ and any finite subset
${\cal F}\subset C,$ there exists $\dt>0,$ $\eta>0,$ a finite subset
${\cal G}\subset C$ and a finite subset ${\cal P}\subset
\underline{K}(C)$ satisfying the following:

For any unital  simple \CA\, $A$ with $TR(A)\le 1,$ any
unital \hm\, $\phi: C\to A$ and any unitary $u\in U(A)$ with
\beq\label{TTf-1}
&&\|[\phi(f), \, u]\|<\dt,\,\,\, {\rm Bott}(\phi,\,u)|_{\cal
P}=\{0\} \tand\\
&&\mu_{\tau\circ \imath}(O_a)\ge \Delta(a)\tforal a\ge \eta,
\eneq
where $\imath: C(\T)\to A$ is defined by $\imath(f)=f(u)$ for all $f\in
C(\T),$ there exists a continuous path of unitaries $\{u(t): t\in
[0,1]\}\subset A$ such that
\beq\label{TTf-2}
u(0)=u,\,\,\, u(1)=1\andeqn \|[\phi(f),\,u(t)]\|<\ep
\eneq
for all $f\in {\cal F}$ and $t\in [0,1].$

\end{lem}

\begin{proof}
Let $\Delta_1: (0,1)\to (0,1)$ be defined by
$\Delta_1(a)=\Delta(a)/2$ for all $a\in (0,1).$ Put $D=C\otimes
C(\T).$
 Let $T=N\times K: D_+\setminus\{0\}\to \N\times \R_+\setminus\{0\}$ be associated with
$\Delta$ as in \ref{trace} and $T'=N'\times K':D_+\setminus\{0\}\to
\N\times \R_+\setminus\{0\}$ be associated with $\Delta_1$ as in
\ref{trace}.
Let $N_1=\max\{N, N'\}$ and $K_1=\max\{K, K'\}.$ Define
$T_0(h)=N_1(h)\times K_1(h)$ for $h\in D_+\setminus\{0\}.$

Let $\ep>0$ and ${\cal F}\subset C$ be a finite subset. Let ${\cal
F}_1=\{f\otimes g: f\in {\cal F}\cup\{1_C\},\,\,\, g\in
\{z,1_{C(\T)}\}\}.$
Let $\dt_1>0$ ( in place of $\dt$), ${\cal G}_1\subset D$ ( in place
of ${\cal G}$), ${\cal H}_0\subset D_+\setminus\{0\},$ ${\cal
P}_1\subset \underline{K}(D)$ (in place of ${\cal P}$) and ${\cal
U}\subset U(D)$ be as required by \ref{ACT} for $\ep/256$ (in place
of $\ep$),  ${\cal F}_1$ and $T_0$ ( in place of $T$). We may assume that
$\dt_1<\ep/256.$

To simplify notation, without loss of generality, we may assume that
${\cal H}_0$ is in the unit ball of $D$ and ${\cal G}_1=\{c\otimes
g: c\in {\cal G}_1'\andeqn g=1_{C(\T)},\, g=z\},$ where $1_C\in
{\cal G}_1'$ is a finite subset of $C.$ Without loss of generality,  we may assume that ${\cal U}={\cal
U}_1\cup\{z_1,z_2,...,z_n\},$ where \\${\cal U}_1\subset \{w\otimes
1_{C(\T)}: w\in U(C)\}$ is a finite subset and $z_i=q_i\otimes
z\oplus (1-q_i)\otimes 1_{C(\T)},$ $i=1,2,...,n$ and
$\{q_1,q_2,...,q_n\}\subset C$ is a set of  projections.
We write $\underline{K}(D)=\underline{K}(C)\bigoplus
\boldsymbol{\bt}(\underline{K}(C))$ (see \ref{Dbot2}).
Without loss of generality, we may also assume that ${\cal
P}_1={\cal P}_0\cup \boldsymbol{\bt}({\cal P}_2),$ where ${\cal
P}_0, {\cal P}_2\in \underline{K}(C)$ are finite subsets.
Furthermore, we assume that $q_j\in {\cal G}_1'$ and $[q_j]\in {\cal
P}_2,$ $j=1,2,...,n.$
Let $\dt_0>0$ and let ${\cal G}_0\subset C$ be finite subset such
that there is a unital completely positive linear map $L': D\to A$
such that
\beq\label{TF-n}
\|L'(c\otimes g)-\phi(c)g(u)\|<\dt_1/2\tforal c\in {\cal G}_1'
\andeqn g=1\,\,\,{\rm or}\,\,\, g=z,
\eneq
whenever there is a unitary $u\in A$ such that $\|[\phi(c),\,
u]\|<\dt_0$ for all $c\in {\cal G}_0.$ By applying \ref{trace}, we
may assume that, $L'$ is $T'$-${\cal H}_0$-full if, in addition,
$$
\mu_{\tau\circ \imath}(O_a)\ge \Delta_1(a)
$$
for all open balls $O_a$ of $\T$ with radius $a\ge \eta_0$ for some
$\eta_0>0$ and for all $\tau\in \text{T}(A),$ where $\imath: C(\T)\to A$ is
defined by $\imath(g)=g(u)$ for all $g\in C(\T).$

We may assume that
\beq\label{TF-k}
[L']|_{{\cal P}_0}=[L'']|_{{\cal P}_0}
\eneq
for any pair of unital completely positive linear maps $L',L'':
C\otimes C(\T)\to A$ for which (\ref{TF-n}) holds for both $L'$ and
$L''$ and
\beq\label{TF-k2}
L'\approx_{\dt_1} L''\,\,\, {\rm on}\,\,\,{\cal G}_1'.
\eneq

Choose an integer $K_0\ge 1$ such that $[{K_0-1\over{\dt_1}}]\ge
128/\dt_1.$  In particular, $(8\pi+1)/[{K_0-1\over{\dt_1}}]<\dt_1.$


Since $TR(C)\le 1,$ there is a projection $p\in C$ and a \SCA\,
$B=\bigoplus_{j=1}^k C(X_j,M_{r(j)}),$ where $X_j=[0,1],$ or $X_j$ is a
point, with $1_B=p$ such that
\beq\label{Tf-1}
\|[x, \,p]\|&<&\min\{\ep/256, \dt_0/4, \dt_1/16\}\tforal x\in {\cal
G}_1'\cup{\cal G}_0\\\label{Tf-1+1} {\rm
dist}(pxp,B)&<&\min\{\ep/256, \dt_0/4,\dt_1/16\}\tforal x\in {\cal
G}_1'\cup{\cal G}_0 \andeqn\\\label{Tf-1+2}
\tau(1-p)&<&\min\{\dt_1/K_0, \Delta(\eta_0)/4,\dt_0/4\} \tforal
\tau\in T(C).
\eneq

We may also assume that there are projections $q_1',q_2',...,q_n'\in
(1-p)C(1-p)$ such that
\beq\label{Tf-1+1++}
\|q_i'-(1-p)q_i(1-p)\|<\min\{\ep/16, \dt_0/4,
\dt_1/16\},\,\,\,i=1,2,...,n.
\eneq
To simplify notation, without loss of generality, we may assume that
$p$ commutes with ${\cal G}'\cup{\cal G}_0.$

Moreover, we may assume that there is a unital completely positive
linear map $L_{00}: C\to pCp\to B$ (first sending $c$ to $pcp$ then
to $B$) such that
\beq\label{Tf-2-1}
\|x-((1-p)x(1-p)+L_{00}(x))\|<\min\{\ep/16,\dt_0/2,\dt_1/4\}\tforal
x\in {\cal G}_1.
\eneq
Put $L_0'(c)=(1-p)c(1-p)$ and $L_0(c)=L_0'(c)+L_{00}(pcp)$ for all
$c\in C.$ We may further assume  that $[L_{00}]({\cal P}_2)$ and
$[L_0']({\cal P}_2)$ are well-defined and
\beq\label{Tf-2-2}
[L_0]|_{{\cal P}_0\cup{\cal P}_2}=[{\rm id}_C]|_{{\cal P}_0\cup{\cal
P}_2}.
  \eneq
Put ${\cal P}_3=[L_0']({\cal P}_2)\cup\{[q_i']:1\le i\le n\}\cup
P_0$ and ${\cal P}_4=[L_{00}]({\cal P}_2).$ From the above,
$x=[L_0'](x)+[L_{00}](x)$ for $x\in {\cal P}_2.$

We also assume that
\beq\label{Tf-k3-1}
[L']|_{\boldsymbol{\bt}({\cal P}_2\cup {\cal P}_3\cup {\cap P}_4)}=
[L'']|_{\boldsymbol{\bt}({\cal P}_2\cup {\cal P}_3\cup {\cap P}_4)}
\eneq
for any pair of unital completely positive linear maps from
$C\otimes C(\T)\to A$ such that
\beq\label{Tf-k4-1}
L_1\approx_{\dt_2'} L_2\,\,\,{\rm on}\,\,\, {\cal G}_2'
\eneq
and items in (\ref{Tf-k3-1}) are well-defined for some $\dt_2'>0$
and a finite subset ${\cal G}_2'.$

Let $\dt_2>0$ (in place of $\dt_0$), ${\cal G}_2\subset (1-p)C(1-p)$
and $P_0\subset (1-p)C(1-p)$ be as required by \ref{perlength} for
$C=(1-p)C(1-p),$
$\sigma=\dt_1/16,$ ${\cal G}_1'\cup {\cal G}_0$ (in place of ${\cal
G}_0$) and $q_1',q_2',...,q_n'$ (in place of $p_1,p_2,...,p_m$).
Note that we may assume that $P_0\subset {\cal G}_2.$

Put ${\cal P}_3'=[L_{0}']({\cal P}_2)\cup \{[q]:q\in P_0\}.$ Note again that elements in ${\cal P}_3'$
are represented by elements in $(1-p)C(1-p).$
We may assume that
\beq\label{Tf-11+1n}
{\rm Bott}(\phi,\, u)|_{{\cal P}_3'}={\rm Bott}(\phi, u')|_{{\cal
P}_3'}
\eneq
for any pair of unitaries $u$ and $u'$ in $A$ for which
$$
\|[\phi(c),\, u]\|<\min\{\dt_1,\dt_0\},\,\,\, \|[\phi(c),\,
u']\|<2\min\{\dt_1,\dt_0\}
$$
and for which there exists a continuous path of unitaries $\{W(t):
t\in [0,1]\}\subset (1-\phi(p))A(1-\phi(p))$ with
\beq
&&\|W(0)-(1-\phi(p))u(1-\phi(p))\|<\min\{\dt_1,\dt_0\}\andeqn\\
&&\|W(1)-(1-\phi(p))u'(1-\phi(p))\|<\min\{\dt_1,\dt_0\},
\eneq and
$$
\|[\phi(c),\, W(t)]\|<\min\{\dt_1,\dt_0\}
$$
for all $c\in {\cal G}_2$ and $t\in [0,1].$

Write $p_i=1_{C(X_i,M_{r(i)})}\in B,$ $i=1,2,...,k.$ Let ${\cal
F}_{0,i}=\{p_ixp_i: x\in {\cal F}\},$ $i=1,2,...,k.$
We may assume that $X_j=[0,1],$ $j=1,2,...,k_0\le k$ and $X_j$ is a
point for $i=k_0+1,k_0+2,...,k.$

Put $D_j=C(X_j, M_{r(j)})\otimes C(\T).$ Define
$T_i=N|_{(D_j)_+\setminus\{0\}}\times 2R|_{(D_j)_+\setminus \{0\}},$
 $j=1,2,...,k_0.$ Let $\dt_{0,i}>0$ (in place of
$\dt$), ${\cal H}_i\subset (D_i)_+\setminus\{0\}$ and ${\cal
G}_{0,i}\subset C(X_i, M_{r(i)})$ be required by \ref{CL1} for
$\ep/256k$ and ${\cal F}_{0,i}$ and $T_i,$ $i=1,2,...,k_0.$ Let
$\dt_{0,i}>0$ (in place of $\dt$), ${\cal G}_{0,i}\subset M_{r(i)}$
be required by \ref{point} for $\ep/256k$ and ${\cal F}_{0,i},$
$i=k_0+1,k_0+2,...,k.$

Denote by $\{e_{s,j}^{(i)}\}$ a matrix unit for $M_{r(i)},$
$i=1,2,...,k.$ Put
$${\bar R}=\max\{N(h)R(h):h\in {\cal H}_i,\,\,\,i=1,2,...,k_0\}.$$

Let $\dt_3=\min\{\ep/512,\dt_2/2,\dt_2'/2,\dt_1/16, \dt_{0,1}/2,
\dt_{0,2},...,\dt_{0,k}/2\}.$ Let ${\cal G}_3={\cal G}_2'\cup {\cal
G}_1'\cup {\cal G}_2\cup\{1-p, p\}\cup_{i=1}^{k_0}{\cal G}_{0,i}.$
Let ${\cal H}={\cal H}_0\cup\{php: h\in {\cal H}_0\}\cup_{i=1}^{k_0}
{\cal H}_i$ and let ${\cal P}={\cal P}_2\cup{\cal P}_3\cup{\cal
P}_4\cup \{[1-p], [p], [e_{j,j}^{(i)}],[p_i], i=1,2,...,k\}.$




It follows from \ref{trace} that there exists  $\dt_4>0, \eta>0$ and
a finite subset ${\cal G}'\subset C$ satisfying the following:
 there exists
a \morp\, $L: D\to A$ which is $T$-${\cal H}$-full such that
\beq\label{Tf-p2}
\|L(c\otimes 1)-\phi(c)\|<\dt_3/16k{\bar R}\andeqn \|L(c\otimes
z)-\phi(c)w\|<\dt_3/16k{\bar R}\tforal c\in {\cal G}_3
\eneq
and $[L]|_{{\cal P}_1\cup\boldsymbol{\bt}({\cal P}_2')}$ is
well-defined, provided that $w\in A$ is a unitary with
\beq\label{Tf-p3}
\|[\phi(b),\, w]\|&<&3\dt_4\rforal b\in {\cal G}'\andeqn\\
\mu_{\tau\circ \imath}(O_a)&\ge &\Delta(a)
\eneq
for all open balls $O_a$ of $\T$ with radius $a\ge \eta$ for all
$\tau\in \text{T}(A),$ where $\imath: C(\T)\to A$ is defined by
$\imath(f)=f(w)$ for $f\in C(\T).$  We may assume that
$\eta<\eta_0$ and $\dt_4<\ep/256.$

Note that, for $h\in {\cal H}_i,$
\beq\label{Tf-p4}
L(h)\le L(\|h\|p_i)\le \|h\|L(p_i),\,\,\,i=1,2,...,k.
\eneq
Therefore, we may assume that (with a smaller $\dt_4$),
\beq\label{Tf-p4+}
\|L(h)-\phi(p_i)L(h)\phi(p_i)\|<\dt_3/2k{\bar R}
\eneq
for any $h\in {\cal H}_i,$ $i=1,2,...,k_0.$ We may also assume that
\beq\label{Tf-p5}
\|\phi(p_i)L(c\otimes z)\phi(p_i)-\phi(c)w_i'\|<\dt_3/16k{\bar R}
\tforal c\in p_i{\cal G}_3p_i,
\eneq
provided that $w_i'\in U(\phi(p_i)A\phi(p_i))$ such that
\beq\label{Tf-p6}
\|w_i'-\phi(p_i)u\phi(p_i)\|<3\dt_4,\,\,\,i=1,2,...,k.
\eneq

For any  function $g\in C(\T)$ with $0\le g\le 1$ and for any
unitary $u\in U(A),$
\beq\label{Tf-meas1}
\tau(g(u))&=&\tau(\phi(p)g(u)\phi(p))+\tau((1-\phi(p))g(u)(1-\phi(p)))\andeqn\\
\tau(\phi(p)g(u)\phi(p))&\ge& \tau(g(u))-\tau(1-\phi(p))\tforal
\tau\in \text{T}(A).
\eneq
Thus, we may assume (by choosing smaller $\dt_4$ ) that
\beq\label{Tf-meas2}
\mu_{\tau\circ \imath'}(O_a)\ge \Delta(a)/2
\eneq
for all $a\ge \eta$ and $\tau\in \text{T}(A),$ where $\imath': C(\T)\to A$
is defined by $\imath'(f)=f(w)$ (for $f\in C(\T)$) for any $w\in
U(A)$ for which $w=w_0\oplus w_1,$ where $w_0\in
U((1-\phi(p))A(1-\phi(p)))$ and $w_1\in U(\phi(p)A\phi(p)),$ such
that
\beq\label{Tf-meas3}
\|w_1-\phi(p)u\phi(p)\|<2\dt_4,
\eneq
where $u$ and $\phi$ satisfy (\ref{Tf-p2}) and (\ref{Tf-p3}).
Put $\dt=\min\{\dt_4/12, \dt_3/12\}$
and ${\cal G}_4={\cal G}'\cup {\cal G}_3.$ Put ${\cal G}={\cal
G}_4\cup \{(1-p)g(1-p): g\in {\cal G}_3\}\cup \{e_{i,s}^{(0)},
[q_{j,0}]\}.$

Now suppose that $\phi$ and $u\in A$ satisfy the assumptions of the lemma for the
above $\dt,$ $\eta,$ ${\cal G}$ and ${\cal P}.$ In particular, $u\in
U_0(A).$ To simplify notation, without loss of generality, we may
assume that all elements in ${\cal G}$ and in ${\cal H}$ have norm
no more than 1.

By applying \ref{perlength}, one obtains a continuous path of
unitaries $\{w_0(t): t\in [0,1]\}\subset (1-\phi(p))A(1-\phi(p))$
and unitaries $w_j'\in U_0(\phi(q_j')A\phi(q_j'))$ such that
\beq\label{TF-exp}
\|[\phi(c),\,w_0(t)]\|<3\dt\tforal c\in p{\cal G}p\andeqn
\eneq
for all $t\in [0,1],$
\beq\label{TF-exp2}
\|w_0(0)-(1-\phi(p))u(1-\phi(p))\|&<&\dt_1/16,\\\label{TF-exp3}
\|w_j'-\phi(q_j')w_0(1)\phi(q_j')\|&<&\dt_1/16\andeqn\\\label{TF-exp4}
w_j'\oplus (1-\phi(p)-\phi(q_j'))&\in& CU((1-\phi(p))A(1-\phi(p))),
\eneq
$j=1,2,...,n.$  Define $w=w_0(1)\oplus w_1$ for some unitary $w_1$
for which (\ref{Tf-meas3}) holds.

We compute  (by (\ref{TTf-1}), (\ref{Tf-meas3}) and (\ref{Tf-11+1n})) that
\beq\label{Tf-u+3}
{\rm Bott}(\phi,\, w)|_{\cal P}=\{0\}.
\eneq
By (\ref{Tf-meas3}), one also has that
\beq\label{Tf-u+4}
\mu_{\tau\circ \imath'}(O_a)\ge \Delta_1(a)\tforal \tau\in \text{T}(A)
\eneq
and for any open balls $O_a$ of $\T$ with radius $a\ge \eta,$ where
$\imath': C(\T)\to A$ is defined by $\imath'(g)=g(w)$ for all $g\in
C(\T).$

Let $L: D\to A$ be a unital \morp\, which satisfies (\ref{Tf-p2}).
We may also assume that $[L]|_{\cal P}$ is well-defined
\beq\label{Tf-n3}
[L]|_{{\cal P}_0}=[\phi]|_{{\cal P}_0}\andeqn
[L]|_{\boldsymbol{\bt}({\cal P})}=\{0\}\hspace{0.5in}{\rm (by\,\,\,
\ref{Tf-u+3})}.
\eneq

There is a unital completely positive linear map $\Phi:
(1-p)C(1-p)\otimes C(\T)\to (1-\phi(p))A(1-\phi(p))$ such that
\beq\label{TF-n4}
\|\Phi(c\otimes g(z))-\phi(c)g(w_0(1))\|<\dt_1/2
\eneq
for all $c\in {\cal G}_1'\cup{\cal G}_0$ and $g=1_{C(\T)}$ and
$g=z.$

Define $L_1,\, L_2: C\otimes C(\T)\to A$ as follows:
\beq\label{TF-n5}
L_1(c\otimes g(z))&=&\Phi((1-p)c(1-p)\otimes g)\oplus \phi(p)L(c\otimes g)\phi(p)\andeqn\\
L_2(c\otimes g)&=&\phi((1-p)c(1-p))g(1)\oplus \phi(p)L(c\otimes
g)\phi(p)
\eneq
for all $c\in C$ and $g\in C(\T).$ By (\ref{TF-k}), we compute that,
\beq\label{Tf-n6}
[L]|_{{\cal P}_0}=[L_1]|_{{\cal P}_0}=[L_2]|_{{\cal P}_0}.
\eneq
It is easy to see  that that
\beq\label{Tf-n7}
[\phi(1-p)L_2\phi(1-p)]|_{\boldsymbol{\bt}({\cal P}_2)}=\{0\}.
\eneq
One also has, by (\ref{TF-k2}),
\beq\label{Tf-n8}
[L_2]|_{\boldsymbol{\bt}([L_{00}]({\cal P}_2))}&=& [L\circ
L_{00}]|_{\boldsymbol{\bt}({\cal P}_2))}\\\label{Tf-n9}
&=&[L]|_{\boldsymbol{\bt}([L_{00}]({\cal P}_2))}={\rm Bott}(\phi,
u)|_{[L_{00}]({\cal P}_2)}=\{0\}.
\eneq
Combining (\ref{Tf-n7}) and (\ref{Tf-n9}), one obtains that
\beq\label{Tf-n10}
[L_2]|_{\boldsymbol{\bt}({\cal P}_2)}=\{0\}.
\eneq

From (\ref{Tf-11+1n}), one computes that
\beq\label{Tf-n8+}
[L_1]|_{\boldsymbol{\bt}({\cal P}_2)}=[L]|_{\boldsymbol{\bt}({\cal
P}_2)}= {\rm Bott}(\phi,\, u)|_{{\cal P}_2}=\{0\}.
\eneq
It follows that
\beq\label{Tf-n9+}
[L_1]|_{{\cal P}_1}=[L]|_{{\cal P}_1}.
\eneq
It is routine to check that
\beq\label{Tf-n10+}
|\tau\circ L(g)-\tau\circ L_1(g)|<\dt_1\tforal g\in {\cal
G}_1\tforal \tau\in \text{T}(A).
\eneq

If $v\in {\cal U}_1,$ since $\|\phi(v)-L_1(v\otimes
1)\|<\dt_1/2\andeqn \|\phi(v)-L_2(v\otimes 1)\|<\dt_1/2,$ it follows
that
\beq\label{Tf-n11}
{\rm dist}(L_1^{\ddag}({\bar v}), L_2^{\ddag}({\bar v}))<\dt_1.
\eneq
If $\zeta_j=q_j\otimes z,$ $j=1,2,...,n,$ by (\ref{TF-exp2}),
(\ref{TF-exp3}) and (\ref{TF-exp4}), by the choice of $K_0$
 and by applying \ref{length}, one has
that
\beq\label{Tf-n13}
{\rm dist}((L_1^{\ddag}({\bar \zeta_j}), L_2^{\ddag}({\bar \zeta_j}))<\dt_1.
\eneq
Note also that, by (\ref{Tf-u+4}) and by \ref{trace}, both $L_1$ and
$L_2$ are $T_0$-${\cal H}_0$-full.
It then follows from (\ref{Tf-n9}), (\ref{Tf-n10}), (\ref{Tf-n11}),
(\ref{Tf-n13}) and \ref{ACT} that there exists a unitary $W\in U(A)$
such that
\beq\label{Tf-n14}
{\rm ad}\, W\circ L_2\approx_{\ep/256} L_1\,\,\, {\rm on}\,\,\,{\cal
F}_1.
\eneq

We may assume that
\beq\label{Tf-n15}
\|u_i-\phi(p_i)u\phi(p_i)\|<2\dt\andeqn w_1=\sum_{i=1}^ku_i
\eneq
for some  $u_i\in U(\phi(p_i)A\phi(p_i)),$ $i=1,2,...,R(n)$ and
\beq\label{Tf-n16}
u_i\in U_0(\phi(p_i)A\phi(p_i)),\,\,\,i=1,2,...,k
\eneq
(since ${\rm Bott}(\phi, \, u)|_{\cal P}=\{0\}$).
There is a positive element $a_i\in \phi(p_i)A\phi(p_i)$ such that
\beq\label{Tf-n17}
a_iL(p_i)a_i=\phi(p_i)\andeqn \|a_i-\phi(p_i)\|<\dt_3/8k{\bar
R},\,\,\,i=1,2,...,k.
\eneq
Let $\Psi_i: D_i\to \phi(p_i)A\phi(p_i)$ be defined by $
\Psi_i(a)=a_i\phi(p_i)L(a)\phi(p_i)a_i$ for all $a\in D_i,$
$i=1,2,...,k.$ Thus
\beq\label{Tf-n17+}
\|\Psi_i(h)-\phi(p_i)L(h)\phi(p_i)\|<\dt_3/4k{\bar R}
\eneq
for all $h\in {\cal H}_i,$ $i=1,2,...,k.$ Note also that (by
\ref{Tf-p5}))
\beq\label{Tf-n18}
\|\Psi_i(c\otimes 1)-\phi(c)\|<\dt+\dt_3/4k{\bar R}\andeqn
\|L_i(c\otimes z)-\phi(c)u_i\|<\dt_3/4k{\bar R}
\eneq
for all $c\in {\cal G}_{0,i},$ $i=1,2,...,k.$ Note also that
\beq\label{Tf-n19}
{\rm bott}_0(\phi|_{C(X_i, M_{r(i)})},\,
u_i)=\{0\},\,\,\,i=1,2,...,k.
\eneq

Furthermore, for each $h\in {\cal H}_i,$ there exist $x_1(h),
x_2(h),...,x_{N(h)}(h)$ with
 and $\|x_j\|\le R(h),$ $j=1,2,...,N(h)$ such that
\beq\label{Tf-n20}
\sum_{j=1}^{N(h)}x_j(h)^*L(h)x_j(h)=1_A.
\eneq
It follows from (\ref{Tf-n17+}) that
\beq\label{Tf-n21}
\|\sum_{j=1}^{N(h)}x_j(h)^*\Psi_i(h)x_j(h)-1_A\|<
N(h)R(h)({\dt_3\over{4k{\bar R}}})<\dt_3/4k
\eneq
Therefore  that there exists $y(h)\in A_+$ with $\|y(h)\|\le
\sqrt{2}$ such that
\beq\label{Tf-n22}
\sum_{j=1}^{N(h)}y(h)(x_j(h))^*\Phi_i(h)(x_j(h))y(h)=\phi(p_i).
\eneq
It follows that $\Phi_i$ is $T_i$-${\cal H}_i$-full, $i=1,2,...,k.$

It follows from \ref{CL1} and \ref{point} that there is a continuous
path of unitaries $\{u_i(t): t\in [0,1]\}\subset
\phi(p_i)A\phi(p_i)$ such that
\beq\label{Tf-n23}
&&u_i(0)=u_i,\,\,\, u_i(1)=p_i\andeqn\\\label{Tf-n24}
&&\|[\Psi_i(c),\, u_i(t)]\|<\ep/k256 \tforal c\in {\cal F}_{0,i}
\eneq
and for all $t\in [0,1],$ $i=1,2,...,k.$

Define a continuous path of unitaries $\{z(t): t\in [0,1]\}\subset
A$ by
$$
z(t)=(1-\phi(p))\oplus \sum_{i=1}^k u_i(t)\,\,\,\tforal t\in [0,1].
$$
Then $z(0)=(1-\phi(p))+\sum_{i=1}^k u_i$ and $z(1)=1_A.$ By (\ref{Tf-n24}), (\ref{Tf-n17+}) and (\ref{Tf-p2}),
\beq\label{Tf-n24+}
\|[\phi(c),\, z(t)]\|<\ep/128\tforal c\in {\cal F}.
\eneq


Define $u'(t)=(w_0(t)w_0(1)^*\oplus (1-\phi(p)))W^*z(t)W$ for $t\in [0,1].$ Then $u'(1)=1_A$ and
we estimate by (\ref{Tf-meas3}), (\ref{Tf-p2}), (\ref{Tf-n14}), (\ref{TF-n4}) and  (\ref{Tf-meas3}) again
that
\beq\label{Tf-n29-}
u'(0) &\approx_{2\dt_4+\dt_3/2{\bar R}}& (w_0(0)w_0(1)^*\oplus (1-\phi(p)))W^*L_2(1\otimes z)W \\
&\approx_{\ep/256}& (w_0(0)w_0(1)^*\oplus (1-\phi(p))L_1(1\otimes z)\\
&\approx_{\dt_1/2+\dt_3/2{\bar R}} & (w_0(0)\oplus (1-\phi(p))((1-\phi(p))\oplus w_1)\\
&\approx_{\dt_1/16+2\dt_4}& (1-\phi(p))u(1-\phi(u))\oplus \phi(p)u\phi(u).
\eneq
It follows that
\beq\label{Tf-n29}
\|u'(0)-u\| <\ep/8.
\eneq
Moreover, by (\ref{Tf-n14}), $W^*\phi(c)W\approx_{\ep/256} \phi(c)$ for all $c\in {\cal F}.$ It follows that
\beq\label{Tf-n30}
\|[\phi(c), \, u'(t)]\|<\ep/2\tforal c\in {\cal F}\andeqn t\in
[0,1].
\eneq
The lemma follows when one connects $u$ with $u'(0)$ with a
continuous path of length no more than $(\ep/8)\pi.$

\end{proof}

\begin{thm}\label{MT}
Let $C$ be a unital separable  amenable simple \CA\, with $TR(C)\le
1$ which satisfies the UCT.  For any $\ep>0$ and any finite subset
${\cal F}\subset C,$ there exists $\dt>0,$ a finite subset ${\cal
G}\subset C$ and a finite subset ${\cal P}\subset\underline{K}(C)$
satisfying the following:

Suppose that $A$ is a unital simple  \CA\, with
$TR(A)\le 1,$ suppose that $\phi: C\to A$ is a unital \hm\, and
$u\in U(A)$ such that
\beq\label{MT-1}
\|[\phi(c),\, u]\|<\dt\tforal c\in {\cal G}\tand {\rm
Bott}(\phi, u)|_{\cal P}=0.
\eneq
Then there exists a continuous and piece-wise smooth path of
unitaries $\{u(t): t\in [0,1]\}$ such that
\beq\label{MT-2}
u(0)=u,\,\,\,u(1)=1\andeqn \|[\phi(c),\, u(t)]\|<\ep\tforal c\in
{\cal F}
\eneq
and for all $t\in [0,1].$
\end{thm}

\begin{proof}
Fix $\ep>0$ and a finite subset ${\cal F}\subset C.$
Let $\dt_1>0$ ( in place of $\dt$), $\eta>0,$  ${\cal G}_1\subset C$
(in place of ${\cal G}$) be a finite subset and ${\cal P}\subset
\underline{K}(C)$ be finite subset as required by \ref{Tfull1}
for $\ep,$ ${\cal F}$ and $\Delta=\Delta_{00}.$ We may assume that
$\dt_1<\ep.$

Let $\dt=\dt_1/2.$
Suppose that $\phi$ and $u$ satisfy the conditions in the theorem
for the above $\dt,$ ${\cal G}$ and ${\cal P}.$
It follows from \ref{Full1} that there is a continuous path of
unitaries $\{v(t): t\in [0,1]\}\subset U(A)$ such that
\beq\label{MT-2+p}
v(0)=u, \,\,\, v(1)=u_1\andeqn \|[\phi(c), v(t)]\|<\dt_1
\eneq
for all $c\in {\cal G}_1$ and for all  $t\in [0,1],$ and
\beq\label{MT-3}
\mu_{\tau\circ \imath}(O_a)\ge \Delta(a)\tforal \tau\in \text{T}(A)
\eneq
and for all open balls of radius $a\ge \eta.$

By applying \ref{Tfull1}, there is a continuous path of unitaries
$\{w(t): t\in [0,1]\}\subset A$ such that
\beq\label{MT-4}
w(0)=u_1,\,\,\, w(1)=1\andeqn \|[\phi(c), \, w(t)]\|<\ep
\eneq
for all $c\in {\cal F}$ and $t\in [0,1].$ Put
$$
u(t)=v(2t)\tforal t\in [0,1/2) \andeqn u(t)=w(2t-1/2)\tforal t\in
[1/2,1].
$$

Remark \ref{CCrem} shows that we can actually require, in addition,
the path is piece-wise smooth.
\end{proof}


\bibliographystyle{amsalpha}
\bibliography{}

\end{document}